\documentclass[12pt]{amsart}

\usepackage{graphicx} 
\usepackage{scrextend}
\usepackage{amsfonts, amsmath, amssymb,amsthm,comment}  
\usepackage{times,enumerate}
\usepackage{mathdots}%
\usepackage{nccmath}
\usepackage{mathtools}
\usepackage{multicol}
\usepackage{kbordermatrix}
\renewcommand{\kbldelim}{(}
\renewcommand{\kbrdelim}{)}
\renewcommand{\kbrowstyle}{\displaystyle}
\renewcommand{\kbcolstyle}{\displaystyle}
\def\VR{\kern-\arraycolsep\strut\vrule &\kern-\arraycolsep}
\def\vr{\kern-\arraycolsep & \kern-\arraycolsep}

    {\end{pmatrix}\end{medsize}}%
\usepackage{dsfont,bbm}
\usepackage{rotating}
\usepackage{lscape}
\usepackage{url}
\usepackage{multicol}
\usepackage{thmtools, thm-restate}
\usepackage{blkarray}
\usepackage{multicol}
\usepackage[margin=2.5cm]{geometry}
 \usepackage[foot]{amsaddr}
 \usepackage{hyperref}
\usepackage{cancel}

\usepackage[usenames, dvipsnames]{xcolor}
\DeclareGraphicsExtensions{.pdf,.png,.jpg}

\hypersetup{
    colorlinks,
    linkcolor={blue},
    citecolor={blue},
    urlcolor={blue}
}

\DeclareMathOperator{\Span}{span}

\DeclareMathOperator{\Rank}{rank}
\DeclareMathOperator{\card}{card}

\makeatletter
\def\widebreve{\mathpalette\wide@breve}
\def\wide@breve#1#2{\sbox\z@{$#1#2$}%
     \mathop{\vbox{\m@th\ialign{##\crcr
\kern0.08em\brevefill#1{0.8\wd\z@}\crcr\noalign{\nointerlineskip}%
                    $\hss#1#2\hss$\crcr}}}\limits}
\def\brevefill#1#2{$\m@th\sbox\tw@{$#1($}%
  \hss\resizebox{#2}{\wd\tw@}{\rotatebox[origin=c]{90}{\upshape(}}\hss$}
\makeatletter

\newcommand{\RR}{\mathbb R}

\newcommand{\NN}{\mathbb N}
\newcommand{\ZZ}{\mathbb Z}

\newcommand{\cV}{\mathcal V}

\newcommand{\cM}{\mathcal M}

\newcommand{\cT}{\mathcal C^{(2)}}

\newcommand{\cA}{\mathcal A}

\newcommand{\cZ}{\mathcal Z}
\newcommand{\cC}{\mathcal C}

\newcommand{\benu}{\begin{enumerate}}
\newcommand{\eenu}{\end{enumerate}}
\newcommand{\bop}{\begin{opomba}}
\newcommand{\eop}{\end{opomba}}

\newcommand{\supp}{\mathrm{supp}}

\newtheorem{theorem}{Theorem}[section]
\newtheorem{corollary}[theorem]{Corollary}
\newtheorem{lemma}[theorem]{Lemma}
\newtheorem{proposition}[theorem]{Proposition}

\theoremstyle{definition}

\newtheorem{example}[theorem]{Example}

\newcommand{\mbb}{\mathbb}
\newcommand{\mbf}{\mathbf}

\theoremstyle{remark}
\newtheorem{remark}[theorem]{Remark}

\numberwithin{equation}{section}

\begin{document}

\title{The truncated moment problem on the union of parallel lines}

%----------Author 1
\author[Alja\v z Zalar]{Alja\v z Zalar}

\address{%
Faculty of Computer and Information Science\\
University of Ljubljana\\
Ve\v cna pot 113\\
1000 Ljubljana\\
Slovenia}

\email{aljaz.zalar@fri.uni-lj.si}

\thanks{Supported by the Slovenian Research Agency grants J1-2453, P1-0288.}


\date{\today}
\maketitle

\begin{abstract}
	In this article we study the bivariate truncated moment problem (TMP) of degree $2k$ on the union of parallel lines. 
	First we present an alternative proof of Fialkow's solution \cite{Fia15} to the TMP on the union of two parallel lines (TMP--2pl) 
	using the solution of the truncated Hamburger moment problem (THMP). 	
	We add a new equivalent solvability condition, which is then used together with the THMP, to solve 
	the TMP on the union of three parallel lines (TMP--3pl), our second main result of the article. 
	Finally, we establish a sufficient condition for the existence of a solution to the TMP 
	on the union of $n$ parallel lines in the pure case, i.e.\ when the moment matrix $M_k$ is of the highest possible rank, 
	or equivalently the only column relations 
	come from the union of $n$ lines. The condition is based on the feasibility of a certain linear matrix inequality, corresponding to the extension of $M_k$ by adding rows
	and columns indexed by some monomials of degree $k+1$. The proof is by induction on $n$, where $n\geq 2$ and for the base of induction $n=2$
	we use the solution of the TMP--2pl.%and the THMP. %Finally, we prove that this sufficient condition is also necessary in case $n=4$.%, but it is not necessary in case $n\geq 5$. 
\end{abstract}

\section{Introduction}
\label{S1}

For $x=(x_1,\ldots,x_d)\in \RR^d$ and $i=(i_1,\ldots,i_d)\in \ZZ^d_+$, we set $|i|=i_1+\ldots+i_d$ and $x^i=x_1^{i_1}\cdots x_d^{i_d}$.
Given a real $d$-dimensional multisequence $\beta=\beta^{ (2k)}=\{\beta_i\}_{i\in \ZZ_+^d,|i|\leq 2k}$ of degree $2k$
and a closed subset $K$ of $\RR^d$, the \textbf{truncated moment problem ($K$--TMP)} on $K$ for $\beta$
asks to characterize the existence of a positive Borel measure $\mu$ on $\RR$ with support in $K$ such that
	\begin{equation}\label{moment-measure-cond}
		\beta_i=\int_{K}x^id\mu(x)\quad \text{for}\quad i\in \ZZ^d_+, |i|\leq 2k.
	\end{equation}
If such a measure exists, we say that $\beta$ has a representing measure supported on $K$ and $\mu$ is its $K$--\textbf{representing measure}.

We denote by $M_k=M_k(\beta)=(\beta_{i+j})_{i,j\in \ZZ_+^d,|i|,|j|\leq k}$ the moment matrix associated with $\beta$, where the rows and columns are indexed by 
$X^i:=X_1^{i_1}\cdots X_d^{i_d}$, $i=(i_1,\ldots,i_d)\in \ZZ^d_+$, $|i|\leq k$, in degree-lexicographic order.
Let $\RR[x]:=\RR[x_1,\ldots,x_d]$ be the set of real polynomials in $d$ variables. We write
$\RR[x]_k:=\{p\in \RR[x]\colon \deg p\leq k\}$ for the set of polynomials in $d$ variables
of degree at most $k$. Here the degree stands for the total degree, i.e., the maximal sum of the exponents of the variables over all monomials.
For every $p:=\sum_{i\in \ZZ^d_+, |i|\leq k} a_ix^i\in \RR[x]_k$, we denote by $p(X)=\sum_{i\in \ZZ^d_+, |i|\leq k} a_iX^i$ the 
corresponding vector from the column
space $\cC(M_k)$ of the matrix $M_k$. 
We say that the matrix $M_k$ is \textbf{recursively generated (rg)} if whenever $p,q,pq\in \RR[x]_k$ and $p(X)=\bf 0$, 
also $(pq)(X)=\bf 0,$ where $\bf 0$ stands for the zero vector.
%
%Recall from \cite{CF96}, that $\beta$ has a representing measure $\mu$ with the support $\supp \mu$ being a 
%subset of 
%$\cZ_p:=\{x\in \RR^d\colon p(x)=0\}$ if and only if $p(X)=\bf 0$, where $\bf 0$ stands for the zero vector.

A \textbf{concrete solution} to the $K$--TMP is a set of necessary and sufficient conditions for the existence of a $K$--representing measure $\mu$, 
that can be tested in numerical examples. 
Among necessary conditions, $M_k$ must be positive semidefinite (psd) and rg \cite{CF04,Fia95}, and by \cite{CF96}
for every polynomial $p\in \RR[x]_k$ satisfying $p(X)=\bf 0$, 
the inclusion $\supp(\mu) \subseteq \cZ(p):=\{x\in \RR^d\colon p(x)=0\}$ must hold.
%and satisfies the \textbf{variety condition} \cite[Proposition 3.1 and Corollary 3.7]{CF96}, which
%states that the inequality $\Rank M_k\leq \Card \mathcal V(\beta^{(2k)})$ holds.
In 1991, Curto and Fialkow \cite{CF91} started their investigation of the TMP by solving the three well-known univariate cases, i.e.\  the truncated Hamburger moment problem for $K=\RR$ (THMP),
the truncated Stieltjes moment problem for $K=[a,\infty)$, $a\in \RR$, and the truncated Hausdorff moment problem for $K=[a,b]$, $a,b\in \RR$, $a<b$. 
Then, in the following few decades they completely solved the TMP on quadratic varieties 
in a series of papers \cite{CF02, CF04, CF05, Fia15} by applying their far-reaching flat extension theorem (FET) (see \cite[Theorem 7.10]{CF96}, \cite[Theorem 2.19]{CF05b}, \cite{Lau05}). 
This theorem was the main tool also in some other cases of the TMP \cite{FN10,CY16,Fia11,CF13,CFM08}.

%It turns out that the variety $y^2=1$ is technically especially demanding for this approach \cite{Fia15}.
%, similarly to the cubic variety $y=x^3$ solved in \cite{Fia11}. 

In our previous work we presented how the TMP on some varieties can be reduced to the univariate setting. More precisely, 
the TMP on $xy=0$ can be solved by the use of the THMP \cite[Section 6]{BZ21}, the TMP on $xy=1$ by the use of the strong THMP where also negative moments (i.e.\  $\beta_{i}$ for $2k\leq i \leq -1$) are given 
\cite{Zal++} and the TMP on $y=x^3$ of degree $2k$ by the use of the THMP of degree $6k$ with a missing moment $\beta_{6k-1}$ \cite{Zal+}.
This approach also gives solutions to the TMP on some other varieties such as $y^2=x^3$, $xy^2=1$ and special cases of $y=x^4$, $y^3=x^4$ \cite{Zal+,Zal++}.
These results motivated us to investigate if the TMP on the other three canonical quadratic varieties, i.e.\  $y=x^2$, $x^2+y^2=1$ and $y^2=1$, can be solved by reducing it to 
the univariate setting. Substituting $y$ with $x^2$
the moment $\beta_{i,j}$ corresponds to the moment $\beta_{i+2j,0}$ and it is easy to check that the TMP of degree $2k$ on the variety $y=x^2$ corresponds to the THMP of degree $4k$.
The case $x^2+y^2=1$ is solved in \cite[Section 2]{CF02} by reduction to the univariate trigonometric TMP \cite[Theorem 6.12]{CF91}.
So it remains to study the case $y^2=1$ or equivalently the TMP on the union of two parallel lines (TMP--2pl), first solved in \cite{Fia15} using the FET as the main tool. 

Concerning the TMP on varieties beyond the quadratic ones, the solutions typically require testing some additional numerical conditions which depend on given moments \cite{CY14, CY15,Yoo17a,Yoo17b}.
%(see \cite{CY14, CY15,Yoo17a,Yoo17b}).
Among other cases of the TMP let us mention the recent core variety approach which yielded important new results for the TMP \cite{Fia17,DS18,BF20}.
For the solution of the cubic TMP see \cite{Kim14}. For some other results and variants of the TMP see also \cite{Vas03,KLS11,BK12,Nie14,Ble15,CIK16,IKLS17,Kim21,CGIK+}.

The first contribution of this article is an alternative solution to the TMP--2pl (Theorem \ref{y2=1}), which is more concrete than Fialkow's original solution \cite[Theorem 1.2]{Fia15} in
the sense explained in the remainder of this paragraph. 
\cite[Theorem 1.2]{Fia15} states that $\beta$ has a representing measure on $K:=\{(x,y)\in \RR^2\colon (y-\alpha_1)(y-\alpha_2)=0\}$, $\alpha_1,\alpha_2\in \RR$, $\alpha_1\neq \alpha_2$,
if and only if $M_k$ is psd, rg, has a column relation $(Y-\alpha_1)(Y-\alpha_2)=\mbf{0}$ and satisfies the \textbf{variety condition}
$\displaystyle\Rank M_k\leq \card \cV(M_k)$, where $\cV(M_k):=\bigcap_{p\in \RR[x]_k, \; p(X)=\mbf{0}} \cZ(p)$.
These conditions can be easily verified numerically,
but are sequence specific in the sense that rg relations and the variety $\cV(M_k)$ are not the same for all $\beta$.
In our solution we replace these two conditions by 
rank conditions on certain submatrices of $M_k$, which are the same for all $\beta$.
The first advantage of our solution is that computing ranks of matrices is numerically an easier task than checking the variety condition, since the latter requires 
computing zero sets of polynomials. The second advantage is that our solution is also concrete enough to
be used when solving the TMP on the union of three parallel lines (TMP--3pl) by reduction to the solvability of the TMP--2pl and the THMP.
%of the TMP--2pl it follows that the existence of the decomposition \eqref{decomp-0404-1541}
%is equivalent to the fact that $M_k$ is psd, rg, has a column relation $Y^2=Y$ and satisfies the \textbf{variety condition}, i.e., 
%	$\displaystyle\Rank M_k\leq \card \cV(M_k)$, where $\displaystyle\cV(M_k):=\bigcap_{p\in \RR[x]_k, p(X)=\mbf{0}} \cZ(p)$.

In this paragraph we describe the basic idea for our solution to the TMP--2pl.
%Since the condition of being rg and the variety condition are specific for each $M_k$, they are not explicit enough when using
%the solution of the TMP--2pl for solving the TMP--3pl.
The approach is based 
%for establishing our solution of the TMP--2pl is 
on the reduction to the univariate setting.
The crucial technical step is the application of the affine linear transformation (ALT) such that the lines become $y=0$ and $y=1$.
Then $\beta$ has a representing measure on $y(y-1)=0$ if and only if it can be decomposed as 
%\begin{equation}\label{decomp-0404-1541}
	$\beta=\widetilde\beta+\widehat\beta,$
%\end{equation} 
where 
	$\widetilde\beta=\{\widetilde\beta_i\}_{i\in \ZZ_+^2,|i|\leq 2k}$ has a 
representing measure on $y=0$ 
and 
	$\widehat\beta=\{\widehat\beta_i\}_{i\in \ZZ_+^2,|i|\leq 2k}$ has a 
representing measure on $y=1$. It turns out that all the moments of $\widetilde\beta,\widehat\beta$
are uniquely determined except $\widetilde\beta_{2k,0},\widehat\beta_{2k,0}$ which satisfy the relation $\widetilde\beta_{2k,0}+\widehat\beta_{2k,0}=\beta_{2k,0}$.
Using the solution to the THMP we characterize exactly in terms of the ranks of certain submatrices of $M_k$ 
when the decomposition $\beta=\widetilde\beta+\widehat\beta$ 
%\eqref{decomp-0404-1541} 
exists. %and these rank conditions replace the rg and the variety condition in \cite[Theorem 1.2]{Fia15}.
Moreover, our approach is constructive and produces a $(\Rank M_k)$--atomic representing measure,
which is also the minimal possible in terms of the number of atoms of any representing measure.

The second contribution of the article is the solution to the TMP--3pl (Theorem \ref{y3-2402}).
In this case we again apply the ALT such that one of the lines becomes $y=0$. 
Then $\beta$ has a representing measure on the union of three parallel lines if and only if it can be decomposed as
in $\beta=\widetilde\beta+\widehat\beta$, where
$\widetilde\beta$ has a representing measure on $y=0$ and $\widehat\beta$ has a representing measure on the union of two other horizontal lines. 
Based on our solution to the TMP--2pl 
and the THMP, we exactly characterize when such
a decomposition exists.
It turns out that all the moments of $\widetilde\beta,\widehat\beta$
are uniquely determined except $\widetilde\beta_{2k-j,0},\widehat\beta_{2k-j,0}$, $j=0,1,$ which satisfy the relations $\widetilde\beta_{2k-j,0}+\widehat\beta_{2k-j,0}=\beta_{2k-j,0}$, $j=0,1$.
The special case of the sextic (i.e.\  $2k=6$) TMP--3pl was studied by Yoo in \cite{Yoo17a}. The author's main focus was on the pure case, i.e.\ when $M_3$
has only one column relation define by the three lines. 
He characterized the existence of a representing measure in terms of the solvability of two quadratic equations in the unknowns $\widetilde\beta_{5,0}$, $\widetilde\beta_{6,0}$.
Our solution to the TMP--3pl in the pure case is general for any $2k\geq 6$ and requires (besides the obvious column relations and $M_k$ being psd) 
only checking if a certain matrix precisely determined by $\beta$ is psd. In \cite{Yoo17a}, the author also studies the non-pure sextic case but only under the
symmetry assumption on the variety $\cV(M_3)$, while our solution of the non-pure case is general for any $2k\geq 6$ and does not require any additional
assumption on the variety or the sequence $\beta$. 

Finally, for $n\in \NN$ we study the TMP on the union of $n$ parallel lines (TMP--$n$pl) in the pure case, i.e.\ when the moment matrix $M_k$ has the highest possible rank or equivalently the column
relation defined by the $n$ lines and the ones obtained from it by recursive generation are the only nontrivial column relations of $M_k$.
We apply the ALT such that one of the lines becomes $y=0$.
Then $\beta$ has a representing measure on the union of $n$ parallel lines if and only if it can be decomposed as in $\beta=\widetilde\beta+\widehat\beta$, where
$\widetilde\beta$ has a representing measure on $y=0$ and $\widehat\beta$ has a representing measure on the union of other $n-1$ horizontal lines.
Using our solution to the TMP--3pl and the THMP we noticed a sufficient condition for the solvability of the pure TMP--4pl,
which is based on the feasibility of a certain linear matrix inequality corresponding to the extension of $M_k$
with the addition of the rows and columns indexed by some monomials of degree $k+1$.
It turns out that this condition extends to any $n\geq 4$, where the proof goes by induction on $n$ and we can in fact use the $n=2$ case as a base case (Theorem \ref{cor-0104-1556}).
%Finally we prove that the sufficient condition for the solvability of the pure TMP--$n$pl is
%also necessary in case $n=4$.%, but show by an example that it is not necessary already in case $n=5.$

%%%%%%%%%%%%%%%%%%%%%%%%%%%%%%%%%%%%%%%%%%%%%%%%%%%%%%%%%%%%%%%%%%%%

\subsection{ Readers Guide}
The paper is organized as follows. 
In Section \ref{S2} we fix notation and introduce some tools needed in the proofs of our main results.
In Section \ref{S3} we present a variant of the solution of the TMP--2pl with the proof based on the use of the THMP (see Theorem \ref{y2=1}),
apply it to solve the pure and almost pure cases (see Corollary \ref{cor-0304-1917})
and provide a numerical example (see Example \ref{ex-1103-824}).
In Section \ref{S4} we solve the TMP--3pl (see Theorem \ref{y3-2402}), apply it to solve the pure case (see Corollary \ref{cor-0304-2008}) and give numerical examples demonstrating
the statement of the solution (see Examples \ref{ex1-0504}--\ref{ex4-0504}). Finally, in Section \ref{S5} a sufficient condition for
the solvability of the pure TMP--$n$pl is established (see Theorem \ref{cor-0104-1556}). \\

%%%%%%%%%%%%%%%%%%%%%%%%%%%%%%%%%%%%%%%%%%%%%%%%%%%%%%%%%%%%%%%%%%%%
%%%%%%%%%%%%%%%%%%%%%%%%%%%%%%%%%%%%%%%%%%%%%%%%%%%%%%%%%%%%%%%%%%%%
%%%%%%%%%%%%%%%%%%%%%%%%%%%%%%%%%%%%%%%%%%%%%%%%%%%%%%%%%%%%%%%%%%%%

\noindent \textbf{Acknowledgement}.  I would like to thank Jaka Cimpri\v c and Igor Klep for valuable comments on the preliminary version of this paper.
I also thank three anonymous reviewers for their detailed reading of the paper and many helpful suggestions for improvement. 

\section{Preliminaries}\label{S2}

In this section we fix some terminology, notation and present some tools needed in the proofs of our main results.\\

Let $k\in \NN$ and $\beta=\beta^{ (2k)}=\{\beta_{i,j}\}_{i,j\in \ZZ_+,\; 0\leq i+j\leq 2k}$ be a bivariate sequence of degree $2k$.
In the degree-lexicographic order $1,X,Y,X^2,XY,Y^2,\ldots,X^k,X^{k-1}Y,\ldots,Y^k$ of rows and columns, the corresponding moment matrix to $\beta$
is equal to 
	\begin{equation}\label{281021-1448}
		M_k(\beta):=
		\left(\begin{array}{cccc}
		M[0,0](\beta) & M[0,1](\beta) & \cdots & M[0,k](\beta)\\
		M[1,0](\beta) & M[1,1](\beta) & \cdots & M[1,k](\beta)\\
		\vdots & \vdots & \ddots & \vdots\\
		M[k,0](\beta) & M[k,1](\beta) & \cdots & M[k,k](\beta)
		\end{array}\right),
	\end{equation}
where
	$$M[i,j](\beta):=
		\left(\begin{array}{ccccc}
		\beta_{i+j,0} & \beta_{i+j-1,1} & \beta_{i+j-2,2} & \cdots & \beta_{i,j}\\
		\beta_{i+j-1,1} & \beta_{i+j-2,2} & \beta_{i+j-3,3} & \cdots & \beta_{i-1,j+1}\\
		\beta_{i+j-2,2} & \beta_{i+j-3,3} & \beta_{i+j-4,4} & \cdots & \beta_{i-2,j+2}\\
		\vdots & \vdots & \vdots & \ddots &\vdots\\
		\beta_{j,i} & \beta_{j-1,i+1} & \beta_{j-2,i+2} & \cdots & \beta_{0,i+j}\\
		\end{array}\right).$$ 
	Note that each matrix $M[i,j](\beta)$ is a $(i+1)\times (j+1)$ Hankel matrix, i.e., it is constant on each cross-diagonal. 
Let $Q_1, Q_2$ be subsets of the set $\{X^iY^j\colon  i,j \in \ZZ_+,\; 0\leq i+j\leq k\}$.
We denote by 
$(M_k)|_{Q_1,Q_2}$ 
%$M_k(S_1,S_2)$ 
the submatrix of $M_k$ consisting of the rows indexed by the elements of $Q_1$
and the columns indexed by the elements of $Q_2$. In case $Q:=Q_1=Q_2$, we write 
$(M_k)|_{Q}=(M_k)|_{Q,Q}$
%$(M_k)(S):=(M_k)(S,S)$ 
for short.
We write $\RR^{n\times m}$ for the set of $n\times m$ real matrices. 
For a matrix $M$ we denote by $\cC(M)$ its column space.
The set of real symmetric matrices of size $n$ will be denoted by $S_n$. 
For a matrix $A\in S_n$ the notation $A\succ 0$ (resp.\ $A\succeq 0$) means $A$ is positive definite (pd) (resp.\ positive semidefinite (psd)).
We write $I_n$ for the $n\times n$ identity matrix. We will also use $I$ to denote the identity matrix of appropriate size.
%Let $Q$ be a subset of the set $\{1,\ldots,n\}$. 
%By $A_Q$ we denote the restriction of $A\in S_n$ to rows and columns from the set $Q$.

For $x\in \RR^d$, we use $\delta_x$ to denote the probability measure on $\RR^d$ such that $\delta_x(\{x\})=1$. 
By a \textbf{finitely atomic positive measure} on $\RR^d$ we mean a measure of the form $\mu=\sum_{j=0}^\ell \rho_j \delta_{x_j}$, 
where $\ell\in \NN$, each $\rho_j>0$ and each $x_j\in \RR^d$. The points $x_j$ are called 
\textbf{atoms} of the measure $\mu$ and the constants $\rho_j$ the corresponding \textbf{densities}.

\subsection{Affine linear transformations} \label{affine linear-trans}

%By \cite[Proposition 1.9]{CF04}, 
Let $K\subseteq \RR^2$ and $\beta$ as above. The existence of a $K$--representing measure for $\beta$ is invariant under invertible affine linear transformations of the form 
	$$\phi(x,y)=(\phi(x,y),\phi(x,y)):=(a+bx+cy,d+ex+fy),\; (x,y)\in \RR^{2},$$
$a,b,c,d,e,f\in \RR$ with $bf-ce \neq 0$.
Indeed, if $L_{\beta}:\mbb{R}[x,y]_{\leq 2k}\to \RR$ is the \textbf{Riesz functional} of the sequence $\beta$ defined by 
$$
	L_{\beta}(p):=\sum_{\substack{i,j\in \ZZ_+,\\ 0\leq i+j\leq 2k}} a_{i,j}\beta_{i,j},\qquad \text{where}\quad p=
	\sum_{\substack{i,j\in \ZZ_+,\\ 0\leq i+j\leq 2k}} a_{i,j}x^iy^j,
$$
and we denote by $\widetilde \beta$ the $2$--dimensional sequence defined by
%We define $\widetilde \beta=\{\widetilde \beta_{i,j}\}_{i,j\in \ZZ_,\; 0\leq i+j\leq 2k}$ by
	$$\widetilde \beta_{i,j}=L_{\beta}\big(\phi(x,y)^i \cdot \phi(x,y)^j\big),$$%\quad \text{for every }0\leq i+j\leq 2k.$$
then:

\begin{proposition}[{\cite[Proposition 1.9]{CF05}}] \label{251021-2254}
	Assume the notation above.
	\begin{enumerate}
		\item $M_k(\beta)$ is psd if and only if $M_k(\widetilde \beta)$ is psd.
		\item $\Rank M_k(\beta)=\Rank M_k(\widetilde \beta)$.
		\item $M_k(\beta)$ is rg if and only if $M_k(\widetilde \beta)$ is rg.
		\item\label{291021-2333} $\beta$ admits a $r$--atomic $K$--representing measure iff $\widetilde \beta$ admits a $r$--atomic $\phi(K)$--representing measure.
	\end{enumerate}
\end{proposition}

In the rest of the paper we write $\phi(\beta)$ and $\phi(M_k(\beta))$ to denote $\widetilde \beta$ and $M_k(\widetilde \beta)$, respectively.

%Notice that 
%	$$\beta_w=L_{\beta^{(2k)}}(w)\quad \text{for every }|w|\leq 2k.$$
%
%An important result for converting a given moment problem into a simpler, equivalent moment problem is the application of affine linear transformations to a sequence $\beta$. For $a,b,c,d,e,f\in \RR$ with $bf-ce \neq 0$, let us define 
%	$$\phi(x,y)=(\phi(x,y),\phi(x,y)):=(a+bx+cy,d+ex+fy),\; (x,y)\in \RR^{2}.$$
%Let $\widetilde \beta^{(2k)}$ be the sequence obtained by the rule
%	$$\widetilde \beta_{i,j}=L_{\beta^{(2k)}}(\phi(x,y)^i \cdot \phi(x,y)^j)\quad \text{for every }0\leq i+j\leq 2k.$$
%By \cite[Proposition 1.9]{CF04}, $\beta$ admits a $r$-atomic measure supported on $K$ iff $\widetilde \beta$ admits a $r$-atomic measure supported on $\phi(K)$.
%We write $\phi(M_k):=M_k(\widetilde \beta)$ to denote the image of the moment matrix $M_k$ with transformation $\phi$.
%Notice that 
%	$$L_{\widetilde\beta^{(2k)}}(p) =L_{\beta^{(2k)}}(p\circ \phi)\quad \text{for every }p \in \mbb{R}\!\langle X,Y\rangle_{\leq k}.$$

\subsection{Generalized Schur complements}\label{SubS2.1}
Let $n,m\in \NN$ and 
	\begin{equation*}%\label{matrixM}
		\cM=\left( \begin{array}{cc} A & B \\ C & D \end{array}\right)\in S_{n+m},
	\end{equation*}
where $A\in \RR^{n\times n}$, $B\in \RR^{n\times m}$, $C\in \RR^{m\times n}$  and $D\in \RR^{m\times m}$.
The \textbf{generalized Schur complement} of $A$ (resp.\ $D$) in $\cM$ is defined by
	$$\cM/A=D-CA^\dagger B\quad(\text{resp.}\; \cM/D=A-BD^\dagger C),$$
where $A^\dagger $ (resp.\ $D^\dagger $) stands for the Moore-Penrose inverse of $A$ (resp.\ $D$)  \cite{Zha05}. 

Let us recall now a characterization of psd $2\times 2$ block matrices in terms of Schur
complements.

\begin{theorem}[{\cite{Alb69}}] \label{block-psd} 
	Let $n,m\in \NN$ and
		\begin{equation*}%\label{form-of-M}
			\cM=\left( \begin{array}{cc} A & B \\ B^{T} & C\end{array}\right)\in S_{n+m},
		\end{equation*} 
	where $A\in S_n$, $B\in \RR^{n\times m}$ and $C\in S_m$.
	Then: 
	\begin{enumerate}
		\item The following conditions are equivalent:
			\begin{enumerate}
				\item\label{pt1-281021-2128} $\cM\succeq 0$.
				\item\label{pt2-281021-2128} $C\succeq 0$, $\cC(B^T)\subseteq\cC(C)$ and $\cM/C\succeq 0$.
				\item\label{pt3-281021-2128} $A\succeq 0$, $\cC(B)\subseteq\cC(A)$ and $\cM/A\succeq 0$.
			\end{enumerate}
		\item\label{021121-1052} If $\cM\succeq 0$, then $\Rank \cM=\Rank A$ if and only if $\cM/A=0$.
	\end{enumerate}
\end{theorem}

%
%If $m=1$ in \eqref{form-of-M}, then $\Rank M\in \{\Rank A,\Rank A+1\}$.
%The following proposition characterizes w.r.t.\ the value of $M/A$ when each of the possibilities occurs in the case $M$ is psd.
%
%\begin{proposition}\label{rank-13-07}
%	Let 
%		\begin{equation*}%\label{form-of-M-2}
%			M=\left( \begin{array}{cc} A & b \\ b^{T} & c\end{array}\right)\in S_{n+1}
%		\end{equation*} 
%	be a real symmetric matrix where $A\in S_n$, $b\in \RR^n$ and $c\in \RR$.
%	Then $\Rank M=\Rank A$ if and only if $M/A=0$. Otherwise $\Rank M=\Rank A+1$.
%\end{proposition}
The following proposition expresses the rank of a psd $2\times 2$ block matrix in terms of the ranks of a diagonal block and its
Schur complement.

\begin{proposition}\label{prop-2604-1140}
	Assume the notation of Theorem \ref{block-psd}. Let $\cM\succeq 0$. Then:
	 \begin{align}\label{prop-2604-1140-eq}
		\left(\begin{array}{cc} I_n & 0 \\ -B^TA^\dagger & I_m\end{array}\right)
		\cM
		\left(\begin{array}{cc} I_n & -A^\dagger B \\ 0 & I_m\end{array}\right)&=
		\left(\begin{array}{cc} A & 0 \\ 0 & \cM/A\end{array}\right),\\
		\label{prop-2604-1140}
		\left(\begin{array}{cc} I_n & -BC^\dagger \\ 0 & I_m\end{array}\right)
		\cM
		\left(\begin{array}{cc} I_n & 0 \\ -C^\dagger B^T & I_m\end{array}\right)&=
		\left(\begin{array}{cc} \cM/C & 0 \\ 0 & C\end{array}\right),
	\end{align}
	and hence
	\begin{equation}\label{prop-2604-1140-eq2}
		\Rank \cM= \Rank A+\Rank \cM/A=\Rank C+\Rank \cM/C.
	\end{equation}	
\end{proposition}

\begin{proof}
	First note that the condition $\cM\succeq 0$ implies by Theorem \ref{block-psd} that 
	$\cC(B)\subseteq \cC(A)$ and hence 
	\begin{equation}\label{201021-0820}
		AA^\dagger B=B\qquad \text{and} \qquad B^TA^\dagger A=B^T.
	\end{equation}
	We will use \eqref{201021-0820} to verify the equality \eqref{prop-2604-1140-eq}:
\begin{align*}
		&\left(\begin{array}{cc} I_n & 0 \\ -B^TA^\dagger & I_m\end{array}\right)
		\left(\begin{array}{cc} A & B \\ B^T & C\end{array}\right)
		\left(\begin{array}{cc} I_n & -A^\dagger B \\ 0 & I_m\end{array}\right)\\
		=&
		\left(\begin{array}{cc} A & B \\ 
			-B^TA^\dagger A+B^T & -B^TA^\dagger B+C \end{array}\right)
			\left(\begin{array}{cc} I_n & -A^\dagger B \\ 0 & I_m\end{array}\right)\\
		\underbrace{=}_{\eqref{201021-0820}}& 
		\left(\begin{array}{cc} A & B \\ 
			0 & -B^TA^\dagger B+C \end{array}\right)
		\left(\begin{array}{cc} I_n & -A^\dagger B \\ 0 & I_m\end{array}\right)\\
		=& 
		\left(\begin{array}{cc} A & -AA^\dagger B+B \\ 
			0 & -B^TA^\dagger B+C \end{array}\right)\\
		\underbrace{=}_{\eqref{201021-0820}}& 
		\left(\begin{array}{cc} A & 0 \\ 
			0 & \cM/A \end{array}\right).
\end{align*}
	
	Finally, since $\left(\begin{array}{cc} I_n & -A^\dagger B \\ 0 & I_m\end{array}\right)$ and 
	$\left(\begin{array}{cc} I_n & 0 \\ -B^TA^\dagger & I_m\end{array}\right)$ are invertible, 
	the equality \eqref{prop-2604-1140-eq} implies that 
	$$\Rank \cM=\Rank \left(\begin{array}{cc} A & 0 \\ 0 & \cM/A\end{array}\right)$$
	and consequently the first equality in \eqref{prop-2604-1140-eq2} holds.

	The equality \eqref{prop-2604-1140}
	can be easily derived from  \eqref{prop-2604-1140-eq}. Let $P=\left(\begin{array}{cc} 0 & I_m \\ I_n & 0\end{array}\right)$ be the block permutation matrix.
	Replacing $\cM$ by $P\cM P^T$, using \eqref{prop-2604-1140-eq}  for $P\cM P^T$ as $\cM$ and multiplying the obtained equality by $P^T$ from the left and by $P$ from the right side, we obtain \eqref{prop-2604-1140}.
\end{proof}

An interesting application of the previous result is the following extension principle for psd matrices.

\begin{lemma}\label{extension-principle}
	Let $\cA\in S_n$ be positive semidefinite, 
	$Q$ a subset of the set $\{1,\ldots,n\}$
	and	
	$\cA|_Q$ the restriction of $\cA$ to the rows and columns from the set $Q$. 
	If $\cA|_Qv=0$ for a nonzero vector $v$,
	then $\cA\widehat v=0$ where $\widehat{v}$ is a vector with the only nonzero entries in the rows from $Q$ and such that the restriction 
	$\widehat{v}|_Q$ to the rows from $Q$ equals to $v$. 
\end{lemma}

\begin{proof}
	%For reader's convenience we give a proof of the lemma. 
	We may assume that $\cA|_Q$ is the upper left-hand corner of $\cA$, i.e.\ 
	$\cA=\left(\begin{array}{cc} \cA|_Q & B \\ B^T & C\end{array}\right)$. 
	Taking $\cM=\cA$ in Proposition \ref{prop-2604-1140}, using 
	%It is easy to check that
	%	$$\left(\begin{array}{cc} I & 0 \\ -B^TA_Q^\dagger & I\end{array}\right)
	%		\left(\begin{array}{cc} A_Q & B \\ B^T & C\end{array}\right)
	%		\left(\begin{array}{cc} I & -A_Q^\dagger B \\ 0 & I\end{array}\right)=
	%		\left(\begin{array}{cc} A_Q & 0 \\ 0 & A/A_Q\end{array}\right)$$
	\eqref{prop-2604-1140-eq} and the equality
		$$\left(\begin{array}{cc} I & 0 \\ -B^T{(\cA|_Q)}^\dagger & I\end{array}\right)^{-1}
		=\left(\begin{array}{cc} I & 0 \\ B^T{(\cA|_Q)}^\dagger & I\end{array}\right),
		$$
	it follows that
	\begin{equation}\label{decomp-26-04-1046}
		\left(\begin{array}{cc} \cA|_Q & B \\ B^T & C\end{array}\right)=
			\left(\begin{array}{cc} I & 0 \\ B^T{(\cA|_Q)}^\dagger & I\end{array}\right)
			\left(\begin{array}{cc} \cA|_Q & 0 \\ 0 & \cA/\cA|_Q\end{array}\right)
			\left(\begin{array}{cc} I & {(\cA|_Q)}^\dagger B \\ 0 & I\end{array}\right).
	\end{equation}
	Now, the statement of the lemma easily follows by applying $\widehat{v}$
	on both sides of \eqref{decomp-26-04-1046}.
\end{proof}
%We will need the following proposition on psd extensions of a given matrix.
%
%\begin{proposition}\cite{Smu59}\label{rank-13-07}
%	Let 
%		\begin{equation*}%\label{form-of-M-2}
%			\widetilde{A}=\left( \begin{array}{cc} A & B \\ B^{T} & C\end{array}\right)\in S_{n+m}
%		\end{equation*} 
%	be a real symmetric matrix where $A\in S_n$, $B\in \RR^{n\times m}$ and $C\in S_m$.	
%	Assume that $A\succeq 0$.
%	Then:
%	\begin{enumerate}
%		\item \label{pt2-0110-155} $\widetilde A\succeq 0$ if and only if $\cC(B)\subseteq \cC(A)$ and $C\succeq B^T A^\dagger  B$.
%		\item\label{pt1-0110-155} $\Rank \widetilde A=\Rank A$ if and only if $C=B^T A^\dagger  B$.
%		%\item\label{pt3-0110-155} $\widetilde A\succ 0$ if and only if $C\succ B^T A^\dagger  W$.
%	\end{enumerate}
%\end{proposition}

\subsection{Hankel matrices}\label{SubS2.2}
Let $k\in \NN$.
For 
	%\begin{equation*}\label{vector-v}
		$\beta=(\beta_0,\ldots,\beta_{2k} )\in \RR^{2k+1}$
	%\end{equation*}
we define the corresponding Hankel matrix as
	\begin{equation}\label{vector-v}
		A_{\beta}:=\left(\beta_{i+j} \right)_{i,j=0}^k
					=\left(\begin{array}{ccccc} 
							\beta_0 & \beta_1 &\beta_2 & \cdots &\beta_k\\
							\beta_1 & \beta_2 & \iddots & \iddots & \beta_{k+1}\\
							\beta_2 & \iddots & \iddots & \iddots & \vdots\\
							\vdots 	& \iddots & \iddots & \iddots & \beta_{2k-1}\\
							\beta_k & \beta_{k+1} & \cdots & \beta_{2k-1} & \beta_{2k}
						\end{array}\right)
					\in S_{k+1}.
	\end{equation}
Observe that
	$A_\beta$ is precisely the associated moment matrix $M_k(\beta)$ to $\beta$.
Let
	$\mbf{v_j}:=\left( \beta_{j+\ell} \right)_{\ell=0}^k$ be the $(j+1)$--th column of $A_{\beta}$, $0\leq j\leq k$.
	In this notation, we have that
		$$A_{\beta}=\left(\begin{array}{ccc} 
								\mbf{v_0} & \cdots & \mbf{v_k}
							\end{array}\right).$$
As in \cite{CF91}, the \textbf{rank} of $\beta$, denoted by $\Rank \beta$, is defined by
	$$\Rank \beta=
	\left\{\begin{array}{rl} 
		k+1,&	\text{if } A_{\beta} \text{ is nonsingular},\\
		\min\left\{i\colon \bf{v_i}\in \Span\{\bf{v_0},\ldots,\bf{v_{i-1}}\}\right\},&		\text{if } A_{\beta} \text{ is singular}.
	 \end{array}\right.$$
For $m\in \NN$ with $m\leq 2k$, the notation $A_{\beta}(m)$ stands for the upper left--hand corner of $A_{\beta}$ of size $m+1$, i.e.\  
		$$A_{\beta}(m)=\left(\beta_{i+j} \right)_{i,j=0}^m\in S_{m+1}.$$
The following proposition is an alternative description of $\Rank \beta$ if $A_\beta$ is singular.

\begin{proposition}[{\cite[Proposition 2.2]{CF91}}]\label{alternative-rank} 
	 Let $k\in \NN$, 
	$\beta=(\beta_0,\ldots,\beta_{2k})\in \RR^{2k+1}$,
	and assume that $A_\beta$ is positive semidefinite and singular.
	Then
		$$\Rank \beta=\min\{j\colon 0\leq j\leq k \text{ such that }A_\beta(j) 		
			\text{ is singular}\}.$$
\end{proposition}

An important property of psd Hankel matrices is the following rank principle.

\begin{theorem}[{\cite[Corollary 2.5]{CF91}}]\label{rank-principle} 
	 Let $k\in \NN$, 
	$\beta=(\beta_0,\ldots,\beta_{2k})\in \RR^{2k+1}$, 
	$\widetilde \beta=(\beta_0,\ldots, \beta_{2k-2})\in \RR^{2k-1}$,
	$A_\beta\succeq 0$ and $r=\Rank \widetilde \beta$.
	Then:
  \begin{enumerate}
	\item\label{pt1-10-07-20} $\Rank A_{\widetilde \beta}=r$.
	\item $r\leq \Rank A_\beta \leq r+1$.
	\item $\Rank A_\beta = r$ if and only if 
				$\beta_{2k}=\varphi_0 \beta_{2k-r}+\ldots+
					\varphi_{r-1}\beta_{2k-1},$
	  where
			$$\left(\begin{array}{ccc}\varphi_0 & \cdots &\varphi_{r-1}\end{array}\right)^{T}:=
			(A_{\beta}(r-1))^{-1}\left(\begin{array}{ccc}\beta_r & \cdots &\beta_{2r-1}\end{array}\right)^{T}.$$
  \end{enumerate}
\end{theorem}

We will use the following corollary of Proposition \ref{alternative-rank} and Theorem \ref{rank-principle} in the sequel.

\begin{corollary}\label{rank-theorem-2}
	In the notation of Theorem \ref{rank-principle}, the assumptions 
	$A_\beta\succeq 0$, $A_\beta$ is singular and $r=\Rank \widetilde\beta$, imply that
	  $$r=\Rank \beta=\Rank A_\beta(r-1)=\Rank A_\beta(r)=\ldots=
	  	\Rank A_{\beta}(k-1)=\Rank A_{\widetilde{\beta}}.$$
\end{corollary}

%A sequence $\beta=(\beta_0,\ldots,\beta_{2k})$ with $r:=\Rank \beta$ is \textbf{positively recursively generated}
%if $A_\beta(r-1)\succ 0$ and denoting $(\varphi_0,\ldots,\varphi_{r-1}):=A_{\beta}(r-1)^{-1}(\beta_r,\ldots,\beta_{2r-1})^{T}$, it is true that
%\begin{equation}\label{recursive-generation}
%  \beta_j=\varphi_0\beta_{j-r}+\cdots+\varphi_{r-1}\beta_{j-1}\quad \text{for}\quad j=r,\ldots,2k.
%\end{equation}
%Note that \eqref{recursive-generation} is equivalent to
%\begin{equation}\label{recursive-generation-equivelant}
%  \mbf{v_j}=\varphi_0 \mbf{v_{j-r}}+\cdots+\varphi_{r-1}\mbf{v_{j-1}}\quad \text{for}\quad j=r,\ldots,k.
%\end{equation}

\subsection{Solution of the truncated Hamburger moment problem}\label{SubS2.3}
%The solution of the truncated Hamburger moment problem is the following:

For $x=(x_0,\ldots,x_m)\in \RR^{m+1}$ we denote by $V_x\in \RR^{(m+1)\times (m+1)}$ the Vandermonde matrix 
	$$V_x:=
		\left(\begin{array}{cccc}
		1 & 1 & \cdots & 1\\
		x_0 & x_1 & \cdots & x_m\\
		\vdots & \vdots &  & \vdots\\
		x_0^m & x_1^m & \cdots & x_m^m
		\end{array}\right).$$ 

\begin{theorem} [{\cite[Theorems 3.9 and 3.10]{CF91} and \cite[Theorem 2.7.5]{BW11}}]\label{Hamburger} 
	For $k\in \NN$ and $\beta=(\beta_0,\ldots,\beta_{2k})\in \RR^{2k+1}$ with $\beta_0>0$, the following statements are equivalent:
\begin{enumerate}	
	\item There exists a $\RR$--representing measure for $\beta$, i.e. supported on $\RR$.
	\item There exists a $(\Rank \beta)$--atomic representing measure for $\beta$.
	%\item $\beta$ is positively recursively generated.
	\item\label{pt4-v2206} $A_\beta\succeq 0$ and $\Rank A_\beta=\Rank \beta$.
	\item\label{pt5-v2206} $A_\beta\succeq 0$ and \big($A_\beta(k-1)\succ 0$ or $\Rank A_\beta(k-1)=\Rank A_\beta$\big).
	\item\label{pt6-v2206} $A_\beta\succeq 0$ and 
		$\left(\begin{array}{ccc} \beta_{k+1} & \cdots & \beta_{2k}\end{array}\right)^T
			\in \cC(A_{\beta}(k-1)).$
\end{enumerate}
%Moreover, if $A_\beta$ is singular, then the statements above are equivalent to:
%\begin{enumerate}
%	\setcounter{enumi}{4}
%	\item $A_{\beta}\succeq 0$ and $\Rank A_{\beta}=\Rank A_{\beta}(k-1)$.
%\end{enumerate}

Moreover, if $\beta$ with $r:=\Rank\beta$ has a $\RR$--representing measure and:
\begin{enumerate}[(i)]
	\item\label{031121-2051} 
$r\leq k$, then the $\RR$--representing measure $\mu$ is unique and of the form
	$\mu=\sum_{i=0}^{r-1}\rho_i\delta_{x_i},$
where 
\begin{align*}
	\{x_0,\ldots,x_{r-1}\}
		&=\mathcal Z(t^r-(\varphi_0+\varphi_1 t+\ldots+\varphi_{r-1}t^{r-1})),\\ 
	\left(\begin{array}{ccc}\varphi_0 & \cdots &\varphi_{r-1}\end{array}\right)^{T}
		&:=A_{\beta}(r-1)^{-1}\left(\begin{array}{ccc}\beta_r & \ldots &\beta_{2r-1}\end{array}\right)^{T}, \\
	\left(\begin{array}{ccc}\rho_0 & \cdots &\rho_{r-1}\end{array}\right)^{T}
		&:=V_x^{-1}\mathbf{v_0^{(r-1)}},\quad x=(x_0,\ldots,x_{r-1})\quad \text{and}\quad 
	\mathbf{v_0^{(r-1)}}=\left(\begin{array}{ccc}\beta_0 & \cdots &\beta_{r-1}\end{array}\right)^{T}.
\end{align*}
	\item $r=k+1$, then there are infinitely many $\RR$--representing measures for $\beta$. All $(k+1)$--atomic ones
		are obtained by choosing $\beta_{2k+1}\in \RR$ arbitrarily, defining 
		$\beta_{2k+2}:=A_\beta^{-1}\mathbf{v_{k+1}^{(k)}}$, where  
		$\mathbf{v_{k+1}^{(k)}}=\left(\begin{array}{ccc}\beta_{k+1} & \cdots &\beta_{2k+1}\end{array}\right)^{T}$,
		and proceeding as in \eqref{031121-2051} for $\widetilde \beta:=(\beta_0,\ldots,\beta_{2k+1},\beta_{2k+2})\in \RR^{2k+3}.$ 
\end{enumerate}
\end{theorem}

\section{The TMP on the union of two parallel lines}
\label{S3}

For $k\geq 2$, let 
		$\beta:=\beta^{(2k)}=(\beta_{i,j})_{i,j\in \ZZ_+,i+j\leq 2k}$ 
	be a real bivariate sequence of degree $2k$
	such that $\beta_{0,0}>0$ and let $M_k$ be its associated moment matrix. 
	To establish the existence of a representing measure for $\beta$ supported
	on the union of two parallel lines, we can assume, after applying the appropriate affine linear transformation, 
	that the variety is
		$$K=\{(x,y)\in \RR^2\colon (y-\alpha_1)(y-\alpha_2)=0\},$$
	where $\alpha_1,\alpha_2\in \RR$ are pairwise distinct nonzero real numbers with $\alpha_1<\alpha_2$.
We write 
		$$\vec{X}^{(i)}:=(1,X,\ldots,X^{k-i})\qquad \text{and}\qquad Y^j\vec{X}^{(i)}:=(Y^j,Y^jX,\ldots,Y^jX^{k-i})$$
	for $i=0,\ldots,k$ and $j\in \NN$ such that $j\leq i$.
Let $P$ be a permutation matrix such that moment matrix $PM_kP^T$ has rows and columns indexed in the order $\vec{X}^{(0)}, Y\vec{X}^{(1)},Y^2\vec{X}^{(2)},\ldots,Y^k$.
Let
		$$M:=
		(PM_kP^T)|_{\{\vec{X}^{(0)},Y\vec{X}^{(1)}\}}=
		%M_k(\{\vec{X},Y\vec{X}^{(1)}\})=
		\bordermatrix{
		& \vec{X}^{(0)} & Y\vec{X}^{(1)} \cr
			(\vec{X}^{(0)})^T & A & B \cr 
			(Y\vec{X}^{(1)})^T & B^T & C }=
		\bordermatrix{
		& \vec{X}^{(1)} & X^k & Y\vec{X}^{(1)} \cr
			(\vec{X}^{(1)})^T & A_{1} & a & B_1\cr \rule{0pt}{0.8\normalbaselineskip} 
			X^k & a^T & \beta_{2k,0} & b^T \cr \rule{0pt}{0.8\normalbaselineskip} 
			(Y\vec{X}^{(1)})^T &  B_1 & b & C}
		$$
	and
		$$N:=
		(PM_kP^T)|_{\{\vec{X}^{(1)},Y\vec{X}^{(1)}\}}=
		%M_k(\{\vec{X}^{(1)},Y\vec{X}^{(1)}\})=
		\bordermatrix{
		& \vec{X}^{(1)} & Y\vec{X}^{(1)} \cr
			(\vec{X}^{(1)})^T & A_1 & B_1 \cr
			(Y\vec{X}^{(1)})^T & B_1 & C}$$
	be the restrictions of the moment matrix $PM_kP^T$ to the rows and the columns in the sets $\{\vec{X}^{(0)},Y\vec{X}^{(1)}\}$ and $\{\vec{X}^{(1)},Y\vec{X}^{(1)}\}$, 
	respectively.

\begin{theorem}\label{y2=1}
	Let $K:=\{(x,y)\in \RR^2\colon (y-\alpha_1)(y-\alpha_2)=0\}$, $\alpha_1,\alpha_2\in \RR$, $\alpha_1\neq \alpha_2$,
	be a union of two parallel lines and 	$\beta:=\beta^{(2k)}=(\beta_{i,j})_{i,j\in \ZZ_+,i+j\leq 2k}$, where $k\geq 2$.
	Then the following statements are equivalent:
\begin{enumerate}	
	\item\label{pt1-1003-1125} $\beta$ has a $K$--representing measure.
	\item\label{pt2-1003-1125} $\beta$ has a $(\Rank M_k)$--atomic $K$--representing measure.
	\item\label{pt3-1003-1125} $M_k$ is positive semidefinite, recursively generated and satisfies the column relation
		\begin{equation}\label{col-relation-2303-1028}
			(Y-\alpha_1)(Y-\alpha_2)=\mbf{0}.
		\end{equation}
	\item\label{pt3-1803-0930} $M$ is positive semidefinite, the relations
	\begin{equation}\label{moment-relations-2303-1028-v2}
		\beta_{i,j+2}=(\alpha_1+\alpha_2)\cdot \beta_{i,j+1}-\alpha_1\alpha_2 \cdot \beta_{i,j}
	\end{equation}
	hold for every $i,j\in \ZZ_+$ with $i+j\leq 2k-2$
	 and one of the following statements holds:
		\begin{enumerate}
			\item\label{pt1-1703-2257} $B_1-\alpha_1 A_1$ is invertible.
			\item\label{pt2-1703-2257} $\alpha_2 A_1-B_1$ is invertible.
			\item\label{pt3-1703-2319} $\Rank M=\Rank N$.
		\end{enumerate}
\end{enumerate}
\end{theorem}

The following corollary states that $\beta$ such that $M_k$ is psd always admits a representing measure on the union of two lines in the following two cases:
\begin{itemize}
	\item $\beta$ is \textit{pure}, i.e.\ the only column relations of $M_k$ come from the union of the lines.
	\item $\beta$ is \textit{almost pure}, i.e.\ except the column relations coming from the union of the lines we also have $X^k\in \cC(\vec{X}^{(1)},Y\vec{X}^{(1)})$.
\end{itemize}

\begin{corollary}\label{cor-0304-1917}
	Let $K:=\{(x,y)\in \RR^2\colon (y-\alpha_1)(y-\alpha_2)=0\}$, $\alpha_1,\alpha_2\in \RR$, $\alpha_1\neq \alpha_2$,
	be a union of two parallel lines, $\beta:=\beta^{(2k)}=(\beta_{i,j})_{i,j\in \ZZ_+,i+j\leq 2k}$ where $k\geq 2$
	and $M, N$ as in Theorem \ref{y2=1}. If
	the relations
		$\beta_{i,j+2}=(\alpha_1+\alpha_2)\cdot \beta_{i,j+1}-\alpha_1\alpha_2 \cdot \beta_{i,j}$
	hold for every $i,j\in \ZZ_+$ with $i+j\leq 2k-2$,
	$M$ is positive semidefinite and $\Rank N\in \{2k-1,2k\}$, then $\beta$ has a $K$--representing measure.
\end{corollary}
\begin{remark}
	\begin{enumerate}
	\item\label{pt1-091221-2355}
		As already described in Section \ref{S1}, the main idea behind the proof of Theorem \ref{y2=1} 
	is applying the ALT such that one of the lines becomes $y=0$ and then studying the existence of the decompositions 
	$\beta=\widetilde \beta+\widehat \beta$ such that $\widetilde\beta$, $\widehat \beta$ have representing measures 
	supported on $y=0$ and on $y=1$, respectively. Note that due to the form of the atoms it suffices to study the representations of
	$(M_k(\widetilde \beta))|_{\{\vec{X}^{(0)}\}}$ and $(M_k(\widehat \beta))|_{\{\vec{X}^{(0)}\}}$, since $M_k(\widetilde \beta)$ is non-zero
	only when restricted to the rows/columns in $\{\vec{X}^{(0)}\}$ and in $M_k(\widehat \beta)$ the blocks $(M_k(\widehat \beta))|_{\{Y^{\ell_1}\vec{X}^{(i)}\},\{Y^{\ell_2}\vec{X}^{(j)}\}}$
	are copies of $(M_k(\widehat \beta))|_{\{\vec{X}^{(i)}\},\{\vec{X}^{(j)}\}}$.
	To study the representations of $(M_k(\widetilde \beta))|_{\{\vec{X}^{(0)}\}}$ and $(M_k(\widehat \beta))|_{\{\vec{X}^{(0)}\}}$ we need to use the solution to the THMP.
	Note that due to the column relation $Y=1$, which must hold in $M_k(\widehat \beta)$, the only undetermined moment is $\widehat\beta_{2k,0}$.
	Computing a Schur complement of $(M_k)|_{\{Y\vec{X}^{(1)}\}}$ in the $2\times 2$ block decomposition of $(M_k)|_{\{\vec{X}^{(0)},Y\vec{X}^{(1)}\}}$
	gives us a candidate for $\widehat\beta_{2k,0}$ and hence $\widetilde\beta_{2k,0}=\beta_{2k,0}-\widehat\beta_{2k,0}$. 
	Further, we are able to characterize in terms of the ranks of certain submatrices of $M_k$ (see Theorem \ref{y2=1})  
	the conditions on $\widehat\beta_{2k,0}$ such that $\widetilde\beta$, $\widehat \beta$ both satisfy the THMP. 
	\item\label{pt2-091221-2355}
		Let us compare our Theorem \ref{y2=1} with Fialkow's original solution to the TMP--2pl \cite[Theorem 2.1]{Fia15}.
	As already commented in Section \ref{S1}, \cite[Theorem 2.1]{Fia15} requires $M_k$ being rg and fulfilling the variety condition $\Rank M_k\leq \card \cV(M_k)$.
	The variety condition is numerically more difficult to check than the rank conditions in Theorem \ref{y2=1} above, since it requires computing the variety of $M_k$.
	Moreover, to further apply the solution to the TMP--2pl when solving the TMP--3pl, \cite[Theorem 2.1]{Fia15} is not concrete enough when applied for a symbolic 
	sequence $\beta$ which we will need in Section \ref{S4}, while using Theorem \ref{y2=1} turns out to be concrete enough for this aim. 
	\item 
		Let us also compare the proofs of Theorem \ref{y2=1} and \cite[Theorem 2.1]{Fia15}. The idea of the proof of Theorem \ref{y2=1} described in 
	\eqref{pt1-091221-2355} is technically not very demanding. The most demanding part is keeping track on the ranks of the matrices $M_k(\widetilde \beta)$
	and $M_k(\widehat \beta)$ when manipulating $\widetilde \beta_{2k,0}$, $\widehat \beta_{2k,0}$ but this is only to have a control on the number of
	atoms in the constructed measure. On the other hand the proof of \cite[Theorem 2.1]{Fia15} is technically more demanding. The proof 
	is separated in the pure and the non-pure case.
	In the pure case the basic tool used is the flat extension theorem \cite[Theorems 1.1,1.2]{CF05} which requires constructing a rank preserving extension
	of $M_k(\beta)$ to the moment matrix 
		$$M_{k+1}=
			\left(\begin{array}{cc} M_k(\beta) & B_{k+1} \\ B_{k+1}^T & C_{k+1}\end{array}\right)=
			\left(\begin{array}{cc} M_k(\beta) & B_{k+1} \\ B_{k+1}^T & B^T_{k+1}(M_k(\beta))^{\dagger}B_{k+1}\end{array}\right)	
			.$$ 
	It turns out that there is a two parametric family of possible blocks $B_{k+1}$ and 
	the difficult part is to argue about the existence of a block $B_{k+1}$ such that $B^T_{k+1}(M_k(\beta))^{\dagger}B_{k+1}$
	has a Hankel structure. This requires symbolically inverting $(M_k(\beta))|_{\{\vec{X}^{(0)},Y\vec{X}^{(1)}\}}$ and then comparing specific 	
	entries of $B^T_{k+1}(M_k(\beta))^{\dagger}B_{k+1}$, the approach first used in Fialkow's solution to the TMP on the variety $y=x^3$ \cite{Fia11}. 
	In the remaining non-pure case the author further separates two subcases. In one subcase $M_k$ has the property of being \textit{recursively determinate} \cite{Fia08}
	and is solved in \cite{CF13}, while in the other subcase $M_k$ does not have this property and a flat extension $M_{k+1}$ of $M_k$ is constructed with a less demanding
	analysis as for the pure case above.
	\item 
	Let us comment on the solution of the TMP--2pl in case a sequence $\beta=(\beta_{i+1})_{i,j\in \ZZ_+,i+j\leq 2k-1}$ of degree $2k-1$, $k\geq 2$, is given.
	Assume the notation of Theorem \ref{y2=1}.
	In this case the following statements are equivalent:
	\begin{enumerate}
		\item $\beta$ has a $K$--representing measure.
		\item $(PM_kP^T)|_{\{\vec{X}^{(1)}\}}$ is positive semidefinite, the relations \eqref{moment-relations-2303-1028-v2}
		hold for every $i,j\in \ZZ_+$ such that $i+j\leq 2k-3$ and $(PM_kP^T)|_{\{\vec{X}^{(2)}\},\{X^k\}}\in \cC((PM_kP^T)|_{\{\vec{X}^{(2)}\},\{\vec{X}^{(1)}\}})$.
	\end{enumerate}
	To justify this we need to study when $\beta$ can be extended to the degree $2k$ sequence $\widetilde \beta$ satisfying \eqref{pt3-1803-0930} of Theorem \ref{y2=1}.
	Since the relations \eqref{moment-relations-2303-1028-v2} must hold, all the moments $\widetilde\beta_{2k-j-2,j+2}$ with $0\leq j\leq 2k-2$ are uniquely determined.
	So the only undetermined moments are $\widetilde\beta_{2k,0}$ and $\widetilde\beta_{2k-1,1}$. The question is when can we choose those two moments such that 
	$M$ will by psd and one of the statements in \eqref{pt3-1803-0930} of Theorem \ref{y2=1} holds. The conditions 
	$(PM_kP^T)|_{\{\vec{X}^{(1)}\}}$ is psd and $(PM_kP^T)|_{\{\vec{X}^{(2)}\},\{X^k\}}\in \cC((PM_kP^T)|_{\{\vec{X}^{(2)}\},\{\vec{X}^{(1)}\}})$ 
	ensure that the moment matrix $(PM_kP^T)|_{\{\vec{X}^{(0)}\}}$ is partially psd,
	and hence can be extended to the psd matrix $M$. By decreasing $\widetilde\beta_{2k,0}$ so that $X^k$ becomes linearly dependent of the other columns gives the condition 
	$\Rank M=\Rank N$ which proves the equivalence above.
	\end{enumerate}
\end{remark}

Now we prove Theorem \ref{y2=1}.

\begin{proof}[Proof of Theorem \ref{y2=1}]
	We will prove the following implications: 
		$\eqref{pt3-1803-0930}\Rightarrow \eqref{pt2-1003-1125}\Rightarrow \eqref{pt1-1003-1125}\Rightarrow \eqref{pt3-1003-1125}\Rightarrow \eqref{pt3-1803-0930}.$
	
	First we prove the implication $\eqref{pt3-1803-0930}\Rightarrow \eqref{pt2-1003-1125}$.
	The relations \eqref{moment-relations-2303-1028-v2} imply that $M_k$ has column relations 
	$Y^2X^i=(\alpha_1+\alpha_2)YX^i-\alpha_1\alpha_2 X^i$,
	$i=0,1,\ldots,k-2$. 
	So the column space $\cC(M_k)$ 
	is spanned by the columns in the set
		$\{\vec{X}^{(0)},Y\vec{X}^{(1)}\}$.
	It follows that $M_k$ is of the form 
	\begin{equation}\label{031121-2348}
		P^T\left(\begin{array}{cc} M & MW \\ W^T M  & W^TMW\end{array}\right)P
		=P^T\left(\begin{array}{cc} I & \mbf{0} \\ W^T  & I\end{array}\right)M\left(\begin{array}{cc} I & W \\ \mbf{0}  & I\end{array}\right)P
	\end{equation}
	where $P$ is a permutation matrix such that moment matrix $PM_kP^T$ has rows and columns indexed in the order $\vec{X}^{(0)}, Y\vec{X}^{(1)},Y^2\vec{X}^{(2)},\ldots,Y^k$,
	$W\in \RR^{(2k+1)\times \frac{(k-1)k}{2}}$ is some matrix, while $I$ and $\mbf{0}$ are identity and zero matrices of appropriate sizes, respectively.
	The equality \eqref{031121-2348} first implies that
	\begin{equation}\label{291021-0844}
		\Rank M_k=\Rank M,
	\end{equation}
	and second since $M\succeq 0$, also $M_k\succeq 0$.
	Applying the affine linear transformation 
		$$\phi(x,y)=(x,(\alpha_2-\alpha_1)^{-1}(y-\alpha_1)),$$
%	we may assume that $\alpha_1=0$, $\alpha_2=1$ and 
	$\phi(M_k)$ has the column relations 
	\begin{equation}\label{301021-2350}
		Y^2X^i=YX^i, \quad i=0,1,\ldots,k-2,
	\end{equation}
	while $\phi(\beta)$ satisfies the equalities
	\begin{align}\label{291021-2326}
	\begin{split}
	\phi(\beta)_{i,0}&=\beta_{i,0}\quad \text{for}\quad i=0,\ldots,2k,\\
	\phi(\beta)_{i,1}&=\frac{1}{\alpha_2-\alpha_1}(\beta_{i,1}-\alpha_1 \beta_{i,0})\quad\text{for}\quad i=0,\ldots,2k-1,\\
	\phi(\beta)_{i,2}
		&=\frac{1}{(\alpha_2-\alpha_1)^2}(\beta_{i,2}-2\alpha_1 \beta_{i,1}-\alpha_1^2\beta_{i,0})\\
		&\underbrace{=}_{\eqref{moment-relations-2303-1028-v2}}\frac{1}{(\alpha_2-\alpha_1)^2}\left(\left((\alpha_1+\alpha_2)\beta_{i,1}-\alpha_1\alpha_2\beta_{i,0}\right)-2\alpha_1 \beta_{i,1}-\alpha_1^2\beta_{i,0}\right)\\	
		&=\frac{1}{\alpha_2-\alpha_1}(\beta_{i,1}-\alpha_1 \beta_{i,0})\quad\text{for}\quad i=0,\ldots,2k-2.
	\end{split}
	\end{align}
	By Proposition \ref{251021-2254}, $\phi(M_k)$ is psd and $\Rank \phi(M_k)=\Rank M_k$.
	By \eqref{291021-2326}, $(P\phi(M_k)P^T)|_{\{\vec{X}^{(0)},Y\vec{X}^{(1)}\}}$, which we denote by $\phi(M)$, is equal to
	$$
		\bordermatrix{
		& \vec{X}^{(1)} & X^k & Y\vec{X}^{(1)} \cr
			(\vec{X}^{(1)})^T & 
				A_{1} & a & \frac{1}{\alpha_2-\alpha_1}(B_1-\alpha_1 A_1) \cr \rule{0pt}{\normalbaselineskip}
			X^k & 
				a^T & \beta_{2k,0} & \frac{1}{\alpha_2-\alpha_1}(b^T-\alpha_1 a^T) \cr \rule{0pt}{\normalbaselineskip}
%			\cline{2-4}
			(Y\vec{X}^{(1)})^T & \frac{1}{\alpha_2-\alpha_1}(B_1-\alpha_1 A_1)  & 
				\frac{1}{\alpha_2-\alpha_1}(b-\alpha_1 a)& \frac{1}{\alpha_2-\alpha_1}(B_1-\alpha_1 A_1)}.
	$$
	Note that
	\begin{equation}\label{281021-2350}
		\phi(M)=\left(\begin{array}{ccc} I_{k} & \mbf{0} & -\frac{\alpha_1}{\alpha_2-\alpha_1}I_{k} \\ \rule{0pt}{\normalbaselineskip}
			 \mbf{0} & 1 & \mbf{0} \\ \rule{0pt}{\normalbaselineskip} 0 & \mbf{0} & \frac{1}{\alpha_2-\alpha_1}I_{k} \end{array}\right)^T 
	M \left(\begin{array}{ccc} I_{k} & \mbf{0} & -\frac{\alpha_1}{\alpha_2-\alpha_1}I_k \\ \rule{0pt}{\normalbaselineskip}
	 \mbf{0} & 1 & \mbf{0} \\ \rule{0pt}{\normalbaselineskip} 0 & \mbf{0} & \frac{1}{\alpha_2-\alpha_1}I_{k} \end{array}\right).
	\end{equation}
	%where $I_{k-1}$ is the $(k-1)\times (k-1)$ identity matrix and $\mbf{0}$ stand for zero matrices of appropriate sizes.
	The equality \eqref{281021-2350} implies that
	\begin{equation}\label{281021-2351}
		\Rank \phi(M)=\Rank M \qquad \text{and}\qquad \phi(M)\succeq 0.
	\end{equation}
	We write 
	\begin{equation}\label{031121-1416}
		D:=\frac{1}{\alpha_2-\alpha_1}(B_1-\alpha_1 A_1)\quad \text{and}\quad d:=\frac{1}{\alpha_2-\alpha_1}(b-\alpha_1 a).
	\end{equation}
	We use the equivalence between \eqref{pt1-281021-2128} and \eqref{pt2-281021-2128} of Theorem \ref{block-psd} for the $2\times 2$ block matrix decomposition
	%since
	\begin{equation}\label{031121-1408}
		\bordermatrix{
		& \vec{X}^{(0)} & Y\vec{X}^{(1)} \cr
			(\vec{X}^{(0)})^T & \phi(M)_{11} & \phi(M)_{12}\cr
			(Y\vec{X}^{(1)})^T & \phi(M)_{21} & \phi(M)_{22}}=
		\bordermatrix{
		& \vec{X}^{(0)} & \vr Y\vec{X}^{(1)} \cr
			(\vec{X}^{(0)})^T & \phi(M)_{11} & \VR\begin{array}{c} D\\ d^T \end{array} \cr
			\cline{2-4}
			(Y\vec{X}^{(1)})^T & \begin{array}{cc} D& d \end{array} & \VR D\cr},
	\end{equation}
	of a psd matrix $\phi(M)$, to conclude that
	\begin{equation}\label{cond-2402-838-v2}
		D\succeq 0,\qquad \cC(\left(\begin{array}{cc}  D & d \end{array}\right))\subseteq \cC(D)
		\qquad\text{and}\qquad 
		\phi(M)/\phi(M)_{22}\succeq 0.
	\end{equation}
	Further on, \eqref{cond-2402-838-v2} is equivalent to
	\begin{equation}\label{cond-2402-838}
		D\succeq 0,\qquad d\in\cC(D)
		\qquad\text{and}\qquad 
		\phi(M)/\phi(M)_{22}\succeq 0.
	\end{equation}
	%To shorten the notation we define $A:=M^{(x)}_{(0,2k-2)}$ and $B:=M^{(y)}_{(0,2k-2)}$.
	The third condition in \eqref{cond-2402-838} is equivalent to 
	\begin{align}\label{cond-0703-0942}
	\begin{split}
	 	E:=\phi(M)/\phi(M)_{22}
		&= \phi(M)_{11}-\left(\begin{array}{cc} D& d \end{array}\right)D^\dagger \left(\begin{array}{c} D\\ d^T \end{array}\right)\\
%		\left(\begin{array}{cc} A_1 & a \\ a^T & \beta_{2k,0} \end{array}\right)-
%			\left(\begin{array}{c} D \\ d^T \end{array}\right)
%			D^{+}
%			\left(\begin{array}{cc} D & d \end{array}\right)\\
		&= \left(\begin{array}{cc} A_1 & a \\ a^T & \beta_{2k,0} \end{array}\right)-
			\left(\begin{array}{cc} DD^\dagger D & DD^\dagger d \\ d^TD^\dagger D & d^TD^\dagger d \end{array}\right)\\
		&=\left(\begin{array}{cc} A_1 & a \\ a^T & \beta_{2k,0} \end{array}\right)-
			\left(\begin{array}{cc} D & d \\ d^T & d^TD^\dagger d \end{array}\right)\\
		&=\left(\begin{array}{cc} A_1-D  & a-d \\ 
				a^T-d^T & \beta_{2k,0}- d^TD^\dagger d \end{array}\right)\succeq 0,
	\end{split}
	\end{align}
	where in the third equality we used the facts that $DD^\dagger D=D$ and $DD^\dagger d=d$, 
	which follows from the solvability of the system $Dx=d$ (see the second condition in \eqref{cond-2402-838}).
	We write 
		$$F:=\left(\begin{array}{cc} D  & d \\ d^T & d^TD^\dagger d \end{array}\right)
		\qquad\text{and}\qquad
		D_1:= \left(\begin{array}{c} D  \\ d^T \end{array}\right).$$
	The first two conditions in \eqref{cond-2402-838} and $F/D=0$ imply, by the equivalence between \eqref{pt1-281021-2128} and \eqref{pt3-281021-2128}
	of Theorem \ref{block-psd}, that $F\succeq 0$.
	Note that 
		$\left(\begin{array}{cc} D & d\end{array}\right)=\left(\phi(\beta)_{i+j,1}\right)_{0\leq i\leq k-1, 0\leq j\leq k}$ 
	and
		$\phi(M)_{11}=\left(\phi(\beta)_{i+j,0}\right)_{0\leq i,j\leq k}=\left(\beta_{i+j,0}\right)_{0\leq i,j\leq k}$
	are Hankel matrices.
	%is, in the notation \eqref{281021-1448}, equal to the block $M[1,0](\beta)$ of the moment matrix
	%$\phi(M_k)$. 
	%Similarly 
	%$\phi(M)_{11}$
	%=\left(\begin{array}{cc} A_1 & a\\ a^T & \beta_{2k,0}\end{array}\right)$ 
	%is equal to the block 
	%$M[0,0](\beta)$ of the moment matrix
	%$\phi(M_k)$ and hence it is also a Hankel matrix. 
	So in the notation \eqref{vector-v}, $F=A_{\widehat \beta^{(1)}}\in S_{k+1}$
	and $\phi(M)_{11}=A_{\widehat \beta^{(2)}}\in S_{k+1}$
	where
	\begin{align*}
	\widehat \beta^{(1)}
	&:=(\phi(\beta)_{0,1},\phi(\beta)_{1,1},\ldots,\phi(\beta)_{2k-1,1},d^TD^\dagger d)
		\in \RR^{2k+1},\\
	\widehat \beta^{(2)}
	&:=(\beta_{0,0},\beta_{1,0},\ldots,\beta_{2k,0})
		\in \RR^{2k+1}.
	\end{align*}
	Finally, 
	$E=A_{\widehat \beta^{(2)}}-A_{\widehat \beta^{(1)}}=
	A_{\widehat \beta^{(2)}-\widehat \beta^{(1)}}$ 
	is a Hankel matrix.
	Further on, using \eqref{prop-2604-1140-eq2} of Proposition \ref{prop-2604-1140}
	for the block decomposition \eqref{031121-1408} of $\phi(M)$,
	%matrix
	%$\phi(M)=\left(\begin{array}{cc} \phi(M_{11}) & D_1 \\ D_1^T & D\end{array}\right),$ 
%	\begin{equation*}%\label{}
%		\phi(M)=\left(\begin{array}{cc}
%				E & 0\\
%				%\hline
%				0 & 0
%				\end{array}\right)+
%				\left(\begin{array}{cc}
%				F & D_1\\
%				D_1^T & D
%				\end{array}\right)
%	\end{equation*}
%	\begin{equation*}%\label{}
%		\cC(\phi(M))
%			=\cC(\left(\begin{array}{cc}
%				E+F & D_1\\
%				%\hline
%				D_1^T & D
%				\end{array}\right))
%			=\cC(\left(\begin{array}{cc}
%				E & D_1\\
%				%\hline
%				0 & D
%				\end{array}\right)),
%	\end{equation*}
%	where we used \eqref {pt1-0110-155} of Theorem \ref{rank-13-07} in the second equality,
	%and %by Proposition \ref{prop-2604-1140},
	%\begin{align}\label{eq-0703-1207}
	it follows that
	\begin{equation}\label{eq-0703-1207}
		\Rank \phi(M)
		=
		\Rank E+\Rank D.
	\end{equation}
	We use the equivalence between \eqref{pt1-281021-2128} and \eqref{pt3-281021-2128} of Theorem \ref{block-psd} for the 
	block decomposition \eqref{cond-0703-0942} of $E$, 
%	$2\times 2$ block matrix 
%	$$\left(\begin{array}{cc} A_1-D  & a-d \\ 
%				a^T-d^T & \beta_{2k,0}- d^TD^\dagger d \end{array}\right),$$
%	which is psd by \eqref{cond-0703-0942}, 
	to conclude that
	\begin{equation}\label{281021-1512}
		\begin{split}
		&A_1-D\succeq 0,\quad a-d\in \cC(A_1-D)\quad \text{and}\\
		&\delta:=(\beta_{2k,0}- d^TD^\dagger d)-(a^T-d^T)\left(A_1-D\right)^\dagger (a-d) \geq 0.
		\end{split}
	\end{equation}
	We separate three cases: 
	\begin{enumerate}[(i)]
		\item\label{case1-0703-1002}  $B_1-\alpha_1 A_1$ is invertible.	
		\item\label{case2-0703-1002}  $\alpha_2 A_1-B_1$ is invertible.
		\item\label{case3-2303-0740} $B_1-\alpha_1 A_1$, $\alpha_2 A_1-B_1$ are singular and $\Rank M=\Rank N$.\\
	\end{enumerate}

	%\noindent\textbf{Case 1.} $A-\alpha_2^{-1} B$ is invertible or $\Rank (A-\alpha_2^{-1} B)=\Rank (M_k/C)$.\\
	In the case \eqref{case1-0703-1002}, $D$ is also invertible since by \eqref{031121-1416} it is a scalar multiple of $B_1-\alpha_1 A_1$.
	We write  
	\begin{equation}\label{291021-0914}
		\widetilde F:=
		\left(\begin{array}{cc}
		D & d \\ d^T & d^TD^\dagger d+\delta
		\end{array}\right)
		\quad\text{and}\quad
		G:=
		%\widetilde{\phi(M)/\phi(M)_{22}} :=
		\left(\begin{array}{cc}
		A_1-D & a-d \\[0.5em] a^T-d^T & \beta_{2k,0}-d^TD^\dagger d-\delta
		\end{array}\right).
	\end{equation}
	%Let
%		$$\delta:= (\beta_{2k,0}-d^TD^\dagger d)- (a-d)^T(A_1-D)^+ (a-d)\geq 0.$$
	Since $\widetilde F-F\succeq 0$ and $F\succeq 0$, it follows that 
%	By the equivalence between \eqref{pt1-281021-2128} and 	
%	\eqref{pt3-281021-2128}
%	of Theorem \ref{block-psd}, used for $\cM=\widetilde F$, 
%	the conditions \eqref{cond-2402-838} and 
%	$\widetilde F/D=\delta\geq 0$ imply that 
	$\widetilde F\succeq 0$.
	By	\eqref{021121-1052} of Theorem \ref{block-psd} used for $\cM=\widetilde F$,
	we have that 	
	\begin{equation}\label{ranks-2303-0750}
		\Rank \widetilde F=\left\{\begin{array}{rl} \Rank D,&	\text{if }\delta=0,\\ \Rank D+1,&	\text{if }\delta>0.\end{array}\right.
	\end{equation}
	We use Theorem \ref{block-psd} for $\cM=G$ twice:
	\begin{itemize}
	\item The equivalence between \eqref{pt1-281021-2128} and 	
		\eqref{pt3-281021-2128} together with the first two conditions in
		\eqref{281021-1512} and $G/(A_1-D)=0$ gives us $G\succeq 0$.
	\item \eqref{021121-1052} implies that 
	\begin{equation}\label{031121-1436}	
		\Rank (A_1-D)=\Rank G=\left\{\begin{array}{rl} 
							\Rank E,&	\text{if }\delta=0,\\ 
							\Rank E-1,&	\text{if }\delta>0.\end{array}\right.
	\end{equation}
	\end{itemize}
	%of Theorem \ref{block-psd}, used for $\cM=G$, 
	%Similarly, taking $\cM=G$ in Theorem \ref{block-psd}, 
	%we conclude from the first two conditions in
	%\eqref{281021-1512} and $G/(A_1-D)=0$ that 
	%$G\succeq 0$.
	%Since $\widetilde F$ and $G$ are obtained from Hankel matrices $F$ and $E$, respectively, by changing only the lower right corner, 
	%they are also Hankel matrices. 
	Note that $\widetilde F$ and $G$ are Hankel matrices, in the notation \eqref{vector-v} equal to 
	$A_{\widehat \beta^{(3)}}\in S_{k+1}$
	and $A_{\widehat \beta^{(4)}}\in S_{k+1}$, respectively, 
	where 
	$$\widehat \beta^{(3)}=\widehat \beta^{(1)}+\delta e_{2k+1}\in \RR^{2k+1},\quad
		\widehat \beta^{(4)}=\widehat \beta^{(2)}-\widehat \beta^{(1)}-\delta e_{2k+1}\in \RR^{2k+1}$$
	and $e_{2k+1}=(0,\ldots,0,1)\in \RR^{2k+1}.$
%	\begin{align*}
%	\widehat \beta^{(3)}
%	&=(\beta_{0,1},\beta_{1,1},\ldots,\beta_{2k-1,1},d^TD^\dagger d+\delta)
%		\in \RR^{2k+1},\\
%	\widehat \beta^{(4)}
%	&=(\beta_{0,0},\beta_{1,0},\ldots,\beta_{2k,0})
%		\in \RR^{2k+1}.
%	\end{align*}
%Moreover, they are $(k+1)\times (k+1)$ matrices, so in the notation \eqref{vector-v}
%	we have that $\widetilde F=A_{\widehat\beta^{(1)}}$ and $G=A_{\widehat\beta^{(2)}}$ for appropriate $\widehat\beta^{(1)}, \widehat\beta^{(2)}\in \RR^{2k+1}$.
	By assumption of the case \eqref{case1-0703-1002}, $A_{\widehat\beta^{(3)}}(k-1)=D\succ 0$, 
	while by \eqref{031121-1436}, $\Rank A_{\widehat\beta^{(4)}}(k-1)=\Rank A_{\widehat\beta^{(4)}}$.
	%construction of $G$ we have that $G/(A_1-D)=0$ and by 
%	\eqref{021121-1052} of Theorem \ref{block-psd} used for $\cM=G$, it follows that 
%		$$\Rank A_{\widehat\beta^{(4)}}(k-1)=\Rank A_{\widehat\beta^{(4)}}=	\Rank G=\left\{\begin{array}{rl} 
%							\Rank E,&	\text{if }\delta=0,\\ 
%							\Rank E-1,&	\text{if }\delta>0.\end{array}\right.$$
	So $A_{\widehat\beta^{(3)}}$ and $A_{\widehat\beta^{(4)}}$ satisfy \eqref{pt5-v2206} of Theorem \ref{Hamburger}, 
	which implies that $\widehat\beta^{(3)}$ and $\widehat\beta^{(4)}$ have 
%	%\cite[Theorem 3.9]{CF91} 
%	moment matrices
%	$$
%		\widetilde F:=
%		\left(\begin{array}{cc}
%		D & d \\ d^T & d^TD^\dagger d+\delta
%		\end{array}\right)
%		\quad\text{and}\quad
%		G:=
%		%\widetilde{\phi(M)/\phi(M)_{22}} :=
%		\left(\begin{array}{cc}
%		A_1-D & a-d \\[0.5em] a^T-d^T & \beta_{2k,0}-d^TD^\dagger d-\delta
%		\end{array}\right)
%		$$
	$(\Rank \widetilde F)$--atomic and $(\Rank G)$--atomic representing measures, respectively. 
	%with real atoms 
	%	$x^{(1)}_1,\ldots, x^{(1)}_{\Rank \widetilde F}$ 
	%and 
	%	$x^{(2)}_1,\ldots, x^{(2)}_{\Rank G}$, respectively.
%	Note that by \eqref{021121-1052} of Theorem \ref{block-psd} used for $\cM=\widetilde F$ and $\cM=G$ we have that
%	\begin{equation}\label{ranks-2303-0750}
%		\Rank \widetilde F=\left\{\begin{array}{rl} \Rank D,&	\text{if }\delta=0,\\ \Rank D+1,&	\text{if }\delta>0,\end{array}\right.
%		\quad\text{and}\quad
%		\Rank G=\left\{\begin{array}{rl} 
%							\Rank E,&	\text{if }\delta=0,\\ 
%							\Rank E-1,&	\text{if }\delta>0.\end{array}\right.
%	\end{equation}
 	The equalities \eqref{ranks-2303-0750}, \eqref{031121-1436}	 together with \eqref{eq-0703-1207} imply that $\Rank \phi(M)=\Rank \widetilde F+\Rank G$.
	This together with \eqref{281021-2351} and \eqref{291021-0844} provides that $\Rank M_k=\Rank \widetilde F+\Rank G$.
	Note that 
	\begin{equation}\label{291021-0928}
		\phi(M)={\phi(M)}^{(1)}+{\phi(M)}^{(2)},
	\end{equation}
	where
		$$
		{\phi(M)}^{(1)}:=\bordermatrix{
		& \vec{X}^{(0)} & Y\vec{X}^{(1)} \cr
			(\vec{X}^{(0)})^T & \widetilde F  & D_1\cr
			(Y\vec{X}^{(1)})^T & D_1^T & D}
		\quad \text{and}\quad
		{\phi(M)}^{(2)}:=
		\bordermatrix{
		& \vec{X}^{(0)} & Y\vec{X}^{(1)} \cr
			(\vec{X}^{(0)})^T & G  & \mbf{0}\cr
			(Y\vec{X}^{(1)})^T & \mbf{0} & \mbf{0}
		}.
		$$
	By the form \eqref{291021-0914} of $\widetilde F=A_{\widehat \beta^{(3)}}$ and the forms of ${\phi(M)}^{(1)}$, ${\phi(M)}^{(2)}$, it follows that:
	\begin{itemize} 
		\item If $\sum_{i=1}^{\Rank \widetilde F} \rho^{(3)}_i \delta_{x_i^{(3)}}$, 
			where each  $\rho^{(3)}_i>0$ and each $x_i^{(3)}\in \RR$,
			%$x^{(1)}_i$ are atoms with densities $\rho^{(1)}_i$, $i=1,\ldots,\Rank \widetilde F$,
			is a representing measure for $\widehat \beta^{(3)}$,
			then 
			$\sum_{i=1}^{\Rank \widetilde F} \rho^{(3)}_i \delta_{(x_i^{(3)},1)}$
			is a representing measure for ${\phi(M)}^{(1)}$.
			%$(x^{(1)}_i,1)$ with densities $\rho^{(1)}_i$, $i=1,\ldots,\Rank \widetilde F$,
			%form a representing measure for ${\phi(M)}^{(1)}$. 
		\item If $\sum_{j=1}^{\Rank G} \rho^{(4)}_j \delta_{x_j^{(4)}}$,
			where each  $\rho^{(4)}_j>0$ and each $x_j^{(4)}\in \RR$, 
			%$x^{(1)}_i$ are atoms with densities $\rho^{(1)}_i$, $i=1,\ldots,\Rank \widetilde F$,
			is a representing measure for $\widehat \beta^{(4)}$,
 			%$x^{(2)}_i$ are atoms with densities $\rho^{(2)}_i$, $i=1,\ldots,\Rank G$,	
			%	in the representing measure for $G$,
			then $\sum_{j=1}^{\Rank G} \rho^{(4)}_j \delta_{(x_j^{(4)},0)}$
			is a representing measure for ${\phi(M)}^{(2)}$.
			%is a representing measure for ${\phi(M)}^{(2)}$.
			%the atoms 
			%$(x^{(2)}_i,0)$ with densities $\rho^{(2)}_i$, $i=1,\ldots,\Rank G$,
			%form a representing measure for ${\phi(M)}^{(2)}$.
	\end{itemize}
	Using the equality \eqref{291021-0928} we conclude that 
	\begin{equation}\label{291021-1400}
		\sum_{i=1}^{\Rank \widetilde F} \rho^{(3)}_i \delta_{(x_i^{(3)},1)}+
		\sum_{j=1}^{\Rank G} \rho^{(4)}_j \delta_{(x_j^{(4)},0)}
	\end{equation}
	is a representing measure for $\phi(M)$.	
	Now, since all atoms $(x_i^{(3)},1)$ and $(x_j^{(4)},0)$ satisfy the relations $y^2x^i=yx^i$, $i=0,\ldots,k-2$,
	the moment matrix $\widetilde M_k$ corresponding to the measure \eqref{291021-1400}
	has column relations \eqref{301021-2350}. %$X^iY^2=X^iY$ for $i=0,\ldots,k-2$.  
	Since $(P\phi(M_k)P^T)|_{\{\vec{X}^{(0)},Y\vec{X}^{(1)}\}}=\phi(M)=(P\widetilde{M}_kP^T)|_{\{\vec{X}^{(0)},Y\vec{X}^{(1)}\}}$ and
	both $\phi(M_k)$, $\widetilde{M}_k$ satisfy the column relations \eqref{301021-2350}, %$X^iY^2=X^iY$ for $i=0,\ldots,k-2$, 
	it follows that 
		$$P\phi(M_k)P^T=\left(\begin{array}{cc} \phi(M) & \phi(M)W_1 \\ W_1^T \phi(M)  & W_1^T\phi(M)W_1\end{array}\right)=P\widetilde{M}_kP^T$$ 
	for some matrix $W_1\in \RR^{(2k+1)\times \frac{(k-1)k}{2}}$.
	Hence, \eqref{291021-1400} is a $(\Rank M_k)$--atomic
	representing measure for $\phi(\beta)$. By \eqref{291021-2333} of Proposition \ref{251021-2254},
	$\beta$ also has a $(\Rank M_k)$--atomic representing measure.
	This concludes the proof of the implication $\eqref{pt3-1803-0930}\Rightarrow \eqref{pt2-1003-1125}$ in this case.\\

	In the case \eqref{case2-0703-1002}, 
		$0\prec (\alpha_2-\alpha_1)^{-1}(\alpha_2 A_1-B_1)=A_1-D=A_{\widehat \beta^{(2)}-\widehat \beta^{(1)}}(k-1).$
 	%The matrices $F$ and $E$ square matrices of size $(k+1)\times (k+1)$, so in the notation \eqref{vector-v}
	%we have that $F=A_{\widehat\beta^{(3)}}$ and $E=A_{\widehat\beta^{(4)}}$ for appropriate $\widehat\beta^{(3)}, \widehat\beta^{(4)}\in \RR^{2k+1}$.
	%The conditions \eqref{cond-2402-838} and $F/D=0$ implies, by the equivalence between \eqref{pt1-281021-2128} and \eqref{pt3-281021-2128}
	%of Theorem \ref{block-psd}, that $F\succeq 0$.
	Recall from above that $\Rank F=\Rank A_{\widehat\beta^{(1)}}=\Rank A_{\widehat\beta^{(1)}}(k-1)$ and $F\succeq 0$.
%	Further on, note that
%	$A_1-D=A_{\widehat \beta^{(2)}-\widehat \beta^{(1)}}(k-1)\succ 0$. 
	So $A_{\widehat\beta^{(1)}}$ and 
	$A_{\widehat \beta^{(2)}-\widehat \beta^{(1)}}$ satisfy \eqref{pt5-v2206} of Theorem \ref{Hamburger}, 
	which implies that $\widehat \beta^{(1)}$ and $\widehat \beta^{(2)}-\widehat \beta^{(1)}$ have
	$(\Rank A_{\widehat\beta^{(1)}})=(\Rank F)$--atomic and 
	$(\Rank E)=(\Rank A_{\widehat \beta^{(2)}-\widehat \beta^{(1)}})$--atomic representing measures, respectively. 
	%Using Theorem \ref{Hamburger}, the moment matrices 
	%\cite[Theorem 3.9]{CF91} implies that	
	%	$F$ and $E$
	%admit representing measures with real atoms $x^{(1)}_1,\ldots, x^{(1)}_{\Rank F}$ and $x^{(2)}_1,\ldots, 
	%x^{(2)}_{\Rank E}$, respectively.
	Note that 
	\begin{equation}\label{301021-2301}
		\phi(M)={\phi(M)}^{(3)}+{\phi(M)}^{(4)},
	\end{equation}
	where
		$$
		{\phi(M)}^{(3)}:=\bordermatrix{
		& \vec{X}^{(0)} & Y\vec{X}^{(1)} \cr
			(\vec{X}^{(0)})^T &F  & D_1\cr 
			(Y\vec{X}^{(1)})^T & D_1^T & D
		}
		\quad \text{and}\quad
		{\phi(M)}^{(4)}:=
		\bordermatrix{
		& \vec{X}^{(0)} & Y\vec{X}^{(1)} \cr
			(\vec{X}^{(0)})^T & E  & \mbf{0}\cr 
			(Y\vec{X}^{(1)})^T & \mbf{0} & \mbf{0}}.
		$$
	By a similar reasoning as in the case \eqref{case1-0703-1002}
	it follows that if $\sum_{i=1}^{\Rank F} \rho^{(1)}_i \delta_{x_i^{(1)}}$, where each $\rho^{(1)}_i>0$ and each $x_i^{(1)}\in \RR$,
	is a representing measure for $\widehat \beta^{(1)}$ and
	$\sum_{j=1}^{\Rank E} \rho^{(2)}_j \delta_{x_j^{(2)}}$, where each $\rho^{(2)}_j>0$ and each $x_j^{(2)}\in \RR$,
	is a representing measure for $\widehat \beta^{(2)}-\widehat \beta^{(1)}$, then
	$\sum_{i=1}^{\Rank F} \rho^{(1)}_i \delta_{(x_i^{(1)},1)}+\sum_{j=1}^{\Rank E} \rho^{(2)}_j \delta_{(x_j^{(2)},0)},$
	is a representing measure for $\phi(\beta)$. 
	Using \eqref{281021-2351}, \eqref{eq-0703-1207} and $\Rank F=\Rank D$ this measure is $(\Rank M_k)$--atomic.
	By \eqref{291021-2333} of Proposition \ref{251021-2254},
	$\beta$ also has a $(\Rank M_k)$--atomic representing measure.
	This concludes the proof of the implication $\eqref{pt3-1803-0930}\Rightarrow \eqref{pt2-1003-1125}$ in this case.\\

	It remains to study the case \eqref{case3-2303-0740}. In this case $D$ and $A_1-D$ are singular.
	Writing $\phi(N):={(P\phi(M_k)P^T)|}_{\{\vec{X}^{(1)},Y\vec{X}^{(1)}\}}$, note that 
	\begin{equation}\label{021121-0758}
		\phi(N)=\left(\begin{array}{cc} I_{k} & -\frac{\alpha_1}{\alpha_2-\alpha_1}I_k\\\rule{0pt}{\normalbaselineskip}
			 		\mbf{0} & \frac{1}{\alpha_2-\alpha_1}I_k \end{array}\right)^T 
					N
					\left(\begin{array}{cc} I_{k} & -\frac{\alpha_1}{\alpha_2-\alpha_1}I_k\\\rule{0pt}{\normalbaselineskip}
			 		\mbf{0} & \frac{1}{\alpha_2-\alpha_1}I_k \end{array}\right) 
		\end{equation}
	and
	\begin{equation}\label{021121-0800}
		\cC(\phi(N))
			=\cC(\left(\begin{array}{cc}
				A_1 & D\\
				%\hline
				D & D
				\end{array}\right))
			=\cC(\left(\begin{array}{cc}
				A_1-D & D\\
				%\hline
				0 & D
				\end{array}\right)).
	\end{equation}
	The equality \eqref{021121-0758} implies that
	\begin{equation}\label{021121-0801}
		\Rank \phi(N)=\Rank N,
	\end{equation}
	while the equality \eqref{021121-0800}
	that
	\begin{equation}\label{rank2-2303-0815}
		\Rank \phi(N)
			=\Rank (A_1-D)+ \Rank D.
	\end{equation}
	The assumption $\Rank M=\Rank N$ of the case \eqref{case3-2303-0740} together with 
	\eqref{eq-0703-1207}
	and 
	\eqref{rank2-2303-0815}
	leads to	
	\begin{equation}\label{021121-0811}
		\Rank A_{\widehat \beta^{(2)}-\widehat \beta^{(1)}}(k-1)=\Rank (A_1-D)=\Rank E=\Rank A_{\widehat \beta^{(2)}-\widehat \beta^{(1)}}.
	\end{equation}	
	Now we proceed as in the case \eqref{case2-0703-1002}:
	\begin{itemize}
		\item By Theorem \ref{Hamburger}, $\widehat \beta^{(1)}$ and $\widehat \beta^{(2)}-\widehat \beta^{(1)}$ have 
			$(\Rank F)$--atomic and $(\Rank E)$--atomic representing measures, respectively. 
		\item If $\sum_{i=1}^{\Rank F} \rho^{(1)}_i \delta_{x_i^{(1)}}$ and $\sum_{j=1}^{\Rank E} \rho^{(2)}_j \delta_{x_j^{(2)}}$, where each 
			$\rho^{(k)}_i>0$ and each $x_i^{(k)}\in \RR$, are representing measure for $\widehat \beta^{(1)}$ and $\widehat \beta^{(2)}-\widehat \beta^{(1)}$, 
			respectively, then
			$\sum_{i=1}^{\Rank F} \rho^{(1)}_i \delta_{(x_i^{(1)},1)}+\sum_{j=1}^{\Rank E} \rho^{(2)}_j \delta_{(x_j^{(2)},0)},$ 
			is a representing measure for $\phi(\beta)$, which is $(\Rank M_k)$--atomic by \eqref{281021-2351}, \eqref{021121-0801}, \eqref{rank2-2303-0815},
			$\Rank F=\Rank D$ and $\Rank M=\Rank N.$
		\item By \eqref{291021-2333} of Proposition \ref{251021-2254},
			$\beta$ also has a $(\Rank M_k)$--atomic representing measure.
	\end{itemize} 
%	Using Theorem \ref{Hamburger} implies that
%	the moment matrices 	
%		$F$ and $E$
%	admit representing measures with real atoms $x^{(1)}_1,\ldots, x^{(1)}_{\Rank F}$. and $x^{(2)}_1,\ldots, x^{(2)}_{\Rank E}$, 
%	respectively.
%	Using this
%	together with \eqref{eq-0703-1207} implies that $\phi(\beta)$ %(and thus $M_k$) 
%	admits a $(\Rank M_k)$-atomic measure with atoms
%		$$(x^{(1)}_1,0), \ldots, (x^{(1)}_{\Rank D},0), (x^{(2)}_1,1),\ldots, (x^{(2)}_{\Rank E},1).$$
	This concludes the proof of the implication $\eqref{pt3-1803-0930}\Rightarrow \eqref{pt2-1003-1125}$ in this case.\\

	The implication $\eqref{pt2-1003-1125}\Rightarrow \eqref{pt1-1003-1125}$ is trivial, while the implication $\eqref{pt1-1003-1125}\Rightarrow \eqref{pt3-1003-1125}$ 
	follows from the necessary conditions $M_k$ must satisfy if $\beta$ has a representing measure (see Section \ref{S1}).
	It remains to prove the implication $\eqref{pt3-1003-1125}\Rightarrow \eqref{pt3-1803-0930}$. The column relation \eqref{col-relation-2303-1028}
	and $M_k$ being rg imply that the relations \eqref{moment-relations-2303-1028-v2} must hold.
	If \eqref{pt1-1703-2257} or \eqref{pt2-1703-2257} holds, we are done. Otherwise $B_1-\alpha_1A_1$ and $\alpha_2 A_1-B_1$ are singular
	and we have to prove that $\Rank M=\Rank N$.
	We have that
	%Since $\alpha_2 A_{1}-B_1$ is singular, 
		$$(\alpha_2-\alpha_1)^{-1}(\alpha_2 A_{1}-B_1)=A_1-D=
		A_{\widehat \beta^{(2)}-\widehat \beta^{(1)}}(k-1).$$
	Since $A_{\widehat \beta^{(2)}-\widehat \beta^{(1)}}(k-1)$ is singular and 
	$A_{\widehat \beta^{(2)}-\widehat \beta^{(1)}}$ is psd, 
	it follows by Corollary \ref{rank-theorem-2} used for 
	$\beta=\widehat \beta^{(2)}-\widehat \beta^{(1)}$
	%that its rank is the same to the rank of its largest invertible upper left submatrix 
	%(\cite[Proposition 2.2 and Corollary 2.5]{CF91}). 
	%Hence, 
	that 
	$\Rank A_{\widehat \beta^{(2)}-\widehat \beta^{(1)}}(k-1)=
	\Rank A_{\widehat \beta^{(2)}-\widehat \beta^{(1)}}(k-2)$.
	Therefore
	there is a vector 
		$\widetilde v=\left(\begin{array}{cccc}v_0 & \cdots & v_{k-2} \end{array}\right)^T\in \RR^{k-1}$
	such that
	\begin{equation}\label{kernel-2403-828}
		\left(A_{\widehat \beta^{(2)}-\widehat \beta^{(1)}}(k-1)\right)v=
		(\alpha_2 A_1-B_1)v=0,
	\end{equation}
	where 
		$v=\left(\begin{array}{cc} \widetilde v^T & 1 \end{array}\right)^T$.
	By Lemma \ref{extension-principle} used for $\cA=E$,
	it follows that
	$E \left(\begin{array}{cc} v^T & 0 \end{array}\right)^T=0$. In particular,
	\begin{equation}\label{kernel2-2403-828}
		(a-d)^Tv=(\alpha_2-\alpha_1)^{-1}(\alpha_2 a-b)^Tv=0.
	\end{equation}
	Using \eqref{kernel-2403-828} and \eqref{kernel2-2403-828} we get
	\begin{align*}
		M\left(\begin{array}{c} \alpha_2 v \\ 0 \\ -v\end{array}\right)
		&=
		\left(\begin{array}{c} (\alpha_2 A_1-B_1)v \\ (\alpha_2 a-b)^Tv \\ (\alpha_2 B_1-C)v\end{array}\right)
		=
		\left(\begin{array}{c} (\alpha_2 A_1-B_1)v \\ (\alpha_2 a-b)^Tv \\ \alpha_1(\alpha_2 A_1-B_1)v\end{array}\right)
		=
		0,
	\end{align*}
	where in the second equality we used the fact that the column relation \eqref{col-relation-2303-1028} implies that 
	\begin{equation}\label{C-rel-2303-2303}
		C=(\alpha_1+\alpha_2)B_1-\alpha_1\alpha_2A_1.
	\end{equation}
	By Lemma \ref{extension-principle} used for $\cA=M_k$, 
	it follows that $M_k$ satisfies the column relation
	\begin{equation}\label{col-rel-2303-1345}
		\sum_{i=0}^{k-2}(\alpha_2v_i) X^i+\alpha_2 X^{k-1}-\sum_{i=0}^{k-2}v_i YX^i-YX^{k-1}=0.
	\end{equation}
	Similarly, the singularity of $B_1-\alpha_1 A_1$ implies that 
	$D=A_{\widetilde\beta^{(1)}}(k-1)$ is singular.
	Together with $A_{\widetilde\beta^{(1)}}\succeq 0$
	this implies, %and since it is the upper left submatrix of 
	%$F$, it follows 
	by Corollary \ref{rank-theorem-2} used for $\beta=\widetilde\beta^{(1)}$,
	%that its rank is the same to the rank of its largest invertible upper left submatrix 
	%(\cite[Proposition 2.2 and Corollary 2.5]{CF91}). 
	%Hence,
	that $\Rank A_{\widetilde\beta^{(1)}}(k-1)=
	\Rank A_{\widetilde\beta^{(1)}}(k-2).$
	Therefore
	that there is a vector 
		$u=\left(\begin{array}{cccc}u_0 & \cdots & u_{k-2} \end{array}\right)^T\in \RR^{k-1}$
	such that
	\begin{equation}\label{kernel3-2403-830}
		\left(A_{\widetilde\beta^{(1)}}(k-1)\right)u=(B_1-\alpha_1 A_1)u=0,
	\end{equation}
	where $u=\left(\begin{array}{cc} \widetilde u^T & 1\end{array}\right)^T$.
	By Lemma \ref{extension-principle} used for $\cA=F$, it follows that $F\left(\begin{array}{cc} u^T & 0 \end{array}\right)^T=0$. In particular,
	\begin{equation}\label{kernel4-2403-830}
		(\alpha_2-\alpha_1)^{-1}(b-\alpha_1 a)^Tu=0.
	\end{equation}
	Using \eqref{kernel3-2403-830} and \eqref{kernel4-2403-830} we get
	\begin{align*}
		M\left(\begin{array}{c} -\alpha_1 u \\ 0 \\ u\end{array}\right)
		=
		\left(\begin{array}{c} (B_1-\alpha_1A_1)u \\ (b-\alpha_1 a)^Tu \\ (C-\alpha_1 B_1)u\end{array}\right)
		=
		\left(\begin{array}{c} (B_1-\alpha_1A_1)u \\ (b-\alpha_1 a)^Tu \\ \alpha_2(B_1-\alpha_1A_1)u\end{array}\right)
		=
		0.
	\end{align*}
	where we used \eqref{C-rel-2303-2303} in the second equality.
	By Lemma \ref{extension-principle} used for $\cA=M_k$, it follows that
	$M_k$ satisfies the column relation
	\begin{equation}\label{col-rel-2303-1403}
		\sum_{i=0}^{k-2}(-\alpha_1u_i) X^i-\alpha_1 X^{k-1}+\sum_{i=0}^{k-2}u_i YX^i+YX^{k-1}=0.
	\end{equation}
	Summing up \eqref{col-rel-2303-1345} and \eqref{col-rel-2303-1403} we get the column relation
	\begin{equation}\label{col-rel-2303-2100}
		\sum_{i=0}^{k-2}(\alpha_2v_i-\alpha_1u_i) X^i+(\alpha_2-\alpha_1) X^{k-1}+\sum_{i=0}^{k-2}(u_i-v_i) YX^i=0.
	\end{equation}
	Since $M_k$ is rg, multiplying \eqref{col-rel-2303-2100} with $X$ gives a new column relation
	\begin{equation}\label{col-rel-2303-2101}
		\sum_{i=0}^{k-2}(\alpha_2v_i-\alpha_1u_i) X^{i+1}+(\alpha_2-\alpha_1) X^{k}+\sum_{i=0}^{k-2}(u_i-v_i) YX^{i+1}=0.
	\end{equation}
	Since $\alpha_2-\alpha_1\neq 0$, $X^k\in \cC(\{1,\ldots,X^{k-1},Y,\ldots,YX^{k-1}\}).$
	Hence, $\Rank M=\Rank N$ which proves \eqref{pt3-1703-2319} and concludes the proof of the implication 
	$\eqref{pt3-1003-1125}\Rightarrow \eqref{pt3-1803-0930}$.
\end{proof}

The following is the proof of Corollary \ref{cor-0304-1917}.

\begin{proof}[Proof of Corollary \ref{cor-0304-1917}]
	Let $\phi$ be the linear transformation from the proof of Theorem \ref{y2=1}.
	By \eqref{031121-1416}, \eqref{021121-0801} and \eqref{rank2-2303-0815} we have that
		$\Rank N=\Rank \phi(N)= 
		\Rank (\alpha_2 A_1-B_1) 
		+ 
		\Rank (B_1-\alpha_1 A_1).$
	Under the assumption $\Rank N\in \{2k-1,2k\}$, at least one of $\alpha_2 A_1-B_1$ and $B_1-\alpha_1 A_1$ is of rank $k$ and hence invertible. 
	The corollary follows by applying Theorem \ref{y2=1}.
\end{proof}

The following example demonstrates the use of Theorem \ref{y2=1} and its proof to construct a representing measure supported on the union of two parallel lines for the sequence $\beta$.
The \textit{Mathematica} file with numerical computations for the following example can be found on the link \url{https://github.com/ZalarA/TMP_parallel_lines}.

\begin{example}\label{ex-1103-824}
	Let 
	%\begin{align*}
	$\beta
	=\big(\frac{8}{11}, 
	\frac{12}{11}, \frac{4}{11},  
	\frac{28}{11}, \frac{6}{11}, \frac{4}{11},  
	\frac{72}{11},\frac{14}{11}, \frac{6}{11}, \frac{4}{11},
	\frac{196}{11},\frac{36}{11}, \frac{14}{11}, \frac{6}{11},\frac{4}{11},
	\frac{7164}{143},\frac{98}{11}, \frac{36}{11}, \frac{14}{11},\frac{6}{11}, \frac{4}{11}$,
	 $\frac{1331888}{9295},\frac{3582}{143}, \frac{98}{11}, \frac{36}{11},\frac{14}{11}, \frac{6}{11},\frac{4}{11}
	\big)$
	%\end{align*} 
	be a bivariate sequence of degree 6.
	We will prove below that $\beta$ has a 6--atomic representing measure on the union of parallel lines $y=0$ and $y=1$.
	Let $P$ be the permutation matrix such that moment matrix $PM_3P^T$ has rows and columns indexed in the order $\vec{X}^{(0)}, Y\vec{X}^{(1)},Y^2\vec{X}^{(2)},Y^3$:
	\begin{equation*}
	PM_3P^T=
	\begin{blockarray}{ccccccccccc}
		& 1&X&X^2 &X^3&Y&YX&YX^2&Y^2&Y^2X&Y^{3}\\[0.5em]
		\begin{block}{c(cccccccccc)}
	 1 &\frac{8}{11} & \frac{12}{11} & \frac{28}{11} & \frac{72}{11} & \frac{4}{11} & \frac{6}{11} & \frac{14}{11} & \frac{4}{11} & \frac{6}{11} & \frac{4}{11}\\[0.5em]
	X&    \frac{12}{11} & \frac{28}{11} & \frac{72}{11} & \frac{196}{11} & \frac{6}{11} & \frac{14}{11} & \frac{36}{11} & \frac{6}{11} & \frac{14}{11} & \frac{6}{11}
	\\[0.5em]
	X^2 &  \frac{28}{11} & \frac{72}{11} & \frac{196}{11} & \frac{7164}{143} & \frac{14}{11} & \frac{36}{11} & \frac{98}{11} & \frac{14}{11} & \frac{36}{11} & \frac{14}{11} \\[0.5em]
	X^3&  \frac{72}{11} & \frac{196}{11} & \frac{7164}{143} & \frac{1331888}{9295} & \frac{36}{11} & \frac{98}{11} & \frac{3582}{143} & \frac{36}{11} & \frac{98}{11} & \frac{36}{11} \\[0.5em]
	Y& \frac{4}{11} & \frac{6}{11} & \frac{14}{11} & \frac{36}{11} & \frac{4}{11} & \frac{6}{11} & \frac{14}{11} & \frac{4}{11} & \frac{6}{11} & \frac{4}{11} \\[0.5em]
	YX&  \frac{6}{11} & \frac{14}{11} & \frac{36}{11} & \frac{98}{11} & \frac{6}{11} & \frac{14}{11} & \frac{36}{11} & \frac{6}{11} & \frac{14}{11} & \frac{6}{11} \\[0.5em]
	YX^2&  \frac{14}{11} & \frac{36}{11} & \frac{98}{11} & \frac{3582}{143} & \frac{14}{11} & \frac{36}{11} & \frac{98}{11} & \frac{14}{11} & \frac{36}{11} & \frac{14}{11} \\[0.5em]
	Y^2&   \frac{4}{11} & \frac{6}{11} & \frac{14}{11} & \frac{36}{11} & \frac{4}{11} & \frac{6}{11} & \frac{14}{11} & \frac{4}{11} & \frac{6}{11} & \frac{4}{11} \\[0.5em]
	Y^2X& \frac{6}{11} & \frac{14}{11} & \frac{36}{11} & \frac{98}{11} & \frac{6}{11} & \frac{14}{11} & \frac{36}{11} & \frac{6}{11} & \frac{14}{11} & \frac{6}{11} \\[0.5em]
	Y^{3}&  \frac{4}{11} & \frac{6}{11} & \frac{14}{11} & \frac{36}{11} & \frac{4}{11} & \frac{6}{11} & \frac{14}{11} & \frac{4}{11} & \frac{6}{11} & \frac{4}{11} \\[0.5em]
	\end{block}
	\end{blockarray}.
	\end{equation*}
	$PM_3P^T$
	is psd with the eigenvalues
		$169.371$, $6.0285$, $1.10273$, $0.289673$, $0.112422$, $0.0231819$, $0$, $0$, $0$, $0$
	and has the column relations
		$$Y^3=Y,\quad Y^2X=YX,\quad Y^2=Y,\quad 
			X^3=\frac{57}{13}\cdot X^2- \frac{283}{65}\cdot X+\frac{12}{65}\cdot 1.$$
	The transformation $\phi$ is the identity, i.e., $\phi(x,y)=(x,y)$.
	The matrices $F=\left(\begin{array}{cc}
		D & d \\ d^T & d^TD^\dagger d
		\end{array}\right)$ and $\phi(M)_{11}$ from the proof of Theorem \ref{y2=1} are equal to 
	$A_{\widehat \beta^{(1)}}\in S_4$ and $A_{\widehat \beta^{(2)}}\in S_4$, respectively, where
	%\begin{align*}
		$\widehat \beta^{(1)}
			=\left(\frac{4}{11}, \frac{6}{11}, \frac{14}{11}, \frac{36}{11}, \frac{98}{11},
				\frac{3582}{143}, \frac{665944}{9295}\right)\in \RR^7$
	and
		$\widehat \beta^{(2)}
			=\left(\frac{8}{11}, \frac{12}{11}, \frac{28}{11}, \frac{72}{11}, \frac{196}{11},
				\frac{7164}{143}, \frac{1331888}{9295}\right)\in \RR^7.$
	%\end{align*}
%
%	i.e., 
%		$$
%		\left(
%		\begin{array}{cccc}
%		  \frac{4}{11} & \frac{6}{11} & \frac{14}{11} & \frac{36}{11} \\[0.5em]
%		  \frac{6}{11} & \frac{14}{11} & \frac{36}{11} & \frac{98}{11} \\[0.5em]
%		  \frac{14}{11} & \frac{36}{11} & \frac{98}{11} & \frac{3582}{143} \\[0.5em]
%		  \frac{36}{11} & \frac{98}{11} & \frac{3582}{143} & \frac{665944}{9295} \\ [0.5em]
%		\end{array}
%		\right).
%		$$
	The matrix 
	%$(M_k)_{\{1,X,X^2\}}-D$ 
	$\left(P\phi(M_k)P^T\right)|_{\{1,X,X^2\}}-D$
	is positive definite with the eigenvalues $\approx 10.312$, $0.205027$, $0.0284288$ and so the	
	case \eqref{case2-0703-1002} from the proof of Theorem \ref{y2=1} applies.
%	\begin{itemize}
%		\item 
	A calculation shows that $\widehat \beta^{(2)}=2\widehat \beta^{(1)}$
	and hence
			$A_{\widehat \beta^{(2)}-\widehat \beta^{(1)}}=A_{\widehat \beta^{(1)}}$.
			$A_{\widehat \beta^{(1)}}$ is a psd Hankel matrix satisfying the column relation 
			\begin{equation}\label{col-rel-1603-1455}
				X^3=\frac{57}{13}\cdot X^2-\frac{283}{65}\cdot X+\frac{12}{65}\cdot 1.
			\end{equation}
			By \eqref{031121-2051} of Theorem \ref{Hamburger}, 
			$\widehat \beta^{(1)}$ and $\widehat \beta^{(2)}-\widehat \beta^{(1)}$ have the unique $\RR$--representing measure
			$\rho_1\delta_{x_1}+\rho_2\delta_{x_2}+\rho_3\delta_{x_3}$,
			where 
				$(x_1, x_2, x_3)\approx (0.0445476, 1.17328, 3.53217)$
			are the solutions of the equation \eqref{col-rel-1603-1455},
%			
%			the atoms in the representing measure are its zeroes 
%				$$\left(\begin{array}{ccc} x^{(1)}_1&x^{(1)}_2&x^{(1)}_3\end{array}\right)\approx
%					\left(\begin{array}{ccc} 0.0445476 & 1.17328 & 3.53217 \end{array}\right)$$
%			with densities 
				\begin{align*}
				\left(\begin{array}{ccc} \rho_1^{(1)} & \rho_2^{(1)} & \rho_3^{(1)}\end{array}\right)^T
					&=
					V_{x}^{-1}
					\left(\begin{array}{ccc} \frac{4}{11}& \frac{6}{11} & \frac{14}{11}\end{array}\right)^T
					\approx 
					\left(\begin{array}{ccc} 0.0541354 & 0.233231 & 0.0762695 \end{array}\right)
				\end{align*}
			and $x=(x_1,x_2,x_3)$.
%			where $V(x^{(1)}_1,x^{(1)}_2,x^{(1)}_3)$ is a $3\times 3$ Vandermonde matrix, generated by 
%			$x^{(1)}_1,x^{(1)}_2,x^{(1)}_3$.
%		\item Since $F$ is equal to 
%		$(M^{(\text{lex})}_k)_{\{1,X,X^2,X^3\}}$ 
%		%$M_k(\{1,X,X^2,X^3\})$ 
%		it is also represented by the atoms and densities from the point above.
%	\end{itemize}
	Thus, $\beta$ has a representing measure on the union of the lines $y=0$ and $y=1$:
		$$\rho_1\delta_{(x_1,1)}+\rho_2\delta_{(x_2,1)}+\rho_3\delta_{(x_3,1)}+
			\rho_1\delta_{(x_1,0)}+\rho_2\delta_{(x_2,0)}+\rho_3\delta_{(x_3,0)}.$$ 
\end{example}

\section{The TMP on the union of three parallel lines}
\label{S4}

Let $k\geq 3$ and  
		$\beta:=\beta^{(2k)}=(\beta_{i,j})_{i,j\in \ZZ_+,i+j\leq 2k}$ 
	be a real bivariate sequence of degree $2k$
	such that $\beta_{0,0}>0$ and let $M_k$ be its associated moment matrix.
	To establish the existence of a representing measure for $\beta$ supported
	on the union of three parallel lines, we can assume, after applying the appropriate affine linear transformation, 
	that the variety is
		$$K=\{(x,y)\in \RR^2\colon (y-\alpha_1)(y-\alpha_2)(y-\alpha_3)=0\},$$
	where $\alpha_1,\alpha_2,\alpha_3\in \RR$ are pairwise distinct nonzero real numbers with $\alpha_1<\alpha_3$ and $\alpha_2<\alpha_3$.
	Hence, $M_k$ must satisfy the column relations 
	\begin{equation}\label{col-rel-1203-1831}
		Y^3X^i=(\alpha_1+\alpha_2+\alpha_3)\cdot Y^2X^i-(\alpha_1\alpha_2+\alpha_1\alpha_3+\alpha_2\alpha_3)\cdot YX^i+\alpha_1\alpha_2\alpha_3\cdot X^i
	\end{equation}
	for $i=0,\ldots,k-3$.
	We write 
		$$\vec{X}^{(i)}:=(1,X,\ldots,X^{k-i})\qquad \text{and}\qquad Y^j\vec{X}^{(i)}:=(Y^j,Y^jX,\ldots,Y^jX^{k-i})$$
	for $i=0,\ldots,k-1$ and $j\in \NN$ such that $j\leq i$.
	%We use the notation $Y^j\vec{X}^{(i)}$ to denote
	%	$$Y^j\vec{X}^{(i)}:=(Y^j,Y^jX,\ldots,Y^jX^{k-i}).$$
	%We write 
	%	$$\vec{X}:=(1,X,\ldots, X^{k}), \;\; Y\vec{X}^{(1)}:=(Y,YX,\ldots, YX^{k-1})
	%	\;\;\text{and}\;\;
	%	Y^2\vec{X}^{(1)}':=(Y^2,Y^2X,\ldots, Y^2X^{k-2}).$$
	In the presence of all column relations \eqref{col-rel-1203-1831}, the column space $\cC(M_k)$ is spanned by the columns in the set $\cT=\{\vec{X}^{(0)},Y\vec{X}^{(1)},Y^2\vec{X}^{(2)}\}$.
	Let $P$ be a permutation matrix such that moment matrix $PM_kP^T$ has rows and columns indexed in the order $\vec{X}^{(0)}, Y\vec{X}^{(1)},Y^2\vec{X}^{(2)},\ldots,Y^k$.
	Let
	\begin{equation}\label{notation-Mk-2402-1651}
		M:=
		{(PM_kP^T)|}_{\cT}=
		\bordermatrix{
		& \vec{X}^{(0)} & Y\vec{X}^{(1)} & Y^2 \vec{X}^{(2)}\cr
			(\vec{X}^{(0)})^T & A_{00} & A_{01} & A_{02} \cr
			( Y\vec{X}^{(1)})^T & (A_{01})^T & A_{11} & A_{12}\cr
			(Y^2 \vec{X}^{(2)})^T & (A_{02})^T & (A_{12})^T & A_{22}}=
		\left(\begin{array}{cc} A_{00} & B \\ B^T & C\end{array}\right),
\end{equation}
	be the restriction of the moment matrix $PM_kP^T$ to the rows and columns in the set $\cT$.
%	We write 
%		$$\vec{X}^{(1)}:=(1,X,\ldots, X^{k-1}),\;\;\vec{X}^{(1)}':=(1,X,\ldots, X^{k-2})\;\; \text{and}\;\; Y\vec{X}^{(1)}':=(Y,YX,\ldots, YX^{k-2}).
%		%\;\;\text{and}\;\;
%		%Y^2\vec{X}^{(1)}':=(Y^2,Y^2X,\ldots, Y^2X^{k-2}).
%		$$
	Applying the affine linear transformation
		$\phi(x,y)=(x,y-\alpha_1)$,
	the moment matrix 
	$\phi(M_k)$ satisfies the column relations 
	\begin{equation}\label{091121-2003}
		Y(Y-\widetilde \alpha_2)(Y-\widetilde\alpha_3)X^i=0 %\qquad \text{for}\quad i=0,\ldots,k-3,
	\end{equation}
	for $i=0,\ldots,k-3$, where $\widetilde \alpha_2:=\alpha_2-\alpha_1$, $\widetilde \alpha_3:=\alpha_3-\alpha_1$.
	On the level of moments the relations \eqref{091121-2003} mean that
	\begin{equation}\label{171121-1921}
		\widetilde \beta_{i+j,3+\ell}=
			(\widetilde \alpha_2+\widetilde \alpha_3)\widetilde\beta_{i+j,2+\ell}-
			\widetilde \alpha_2\widetilde\alpha_3\widetilde\beta_{i+j,1+\ell}
	\end{equation}
	for every $i,j,\ell\in \NN\cup\{0\}$ such that $i+j+3+\ell\leq 2k$.
	The matrix ${(P\phi(M_k)P^T)|}_{\cT}$, which we denote by $\phi(M)$, is equal to
	\begin{equation}\label{notation-Mk-3103-0727}
		\phi(M)
		=J^TMJ
		:=
\bordermatrix{
		& \vec{X}^{(0)} & Y\vec{X}^{(1)} & Y^2 \vec{X}^{(2)}\cr
			(\vec{X}^{(0)})^T & A_{00} & \widetilde A_{01} & \widetilde A_{02} \cr \rule{0pt}{\normalbaselineskip} 
			( Y\vec{X}^{(1)})^T & (\widetilde A_{01})^T & \widetilde A_{11} & \widetilde  A_{12}\cr \rule{0pt}{\normalbaselineskip} 
			(Y^2 \vec{X}^{(2)})^T & (\widetilde A_{02})^T & (\widetilde A_{12})^T & \widetilde A_{22}
		}=
		\left(\begin{array}{cc} A_{00} & \widetilde B \\\rule{0pt}{0.8\normalbaselineskip}  (\widetilde B)^T & \widetilde C\end{array}\right)\\
	\end{equation}
	where
		$$
		J=
		\left(\begin{array}{c|c|c}
		I_{k+1} & -\alpha_1 \left(\begin{array}{c} I_{k}\\ \mbf{0}_{1\times k} \end{array}\right) &
		 \alpha_1^2 \left(\begin{array}{c}I_{k-1}\\  \mbf{0}_{2\times (k-1)}\end{array}\right)
		\\[2ex]
		\cline{1-3}& \\[-0.8em]
		%\hline
		\mbf{0}_{k\times (k+1)} & I_{k} & -2\alpha_1  \left(\begin{array}{c}I_{k-1}\\
		 \mbf{0}_{1\times (k-1)}\end{array}\right)
		\\[2ex]
		\hline& \\[-0.8em]
		\mbf{0}_{(k-1)\times (k+1)} & \mbf{0}_{(k-1)\times k}  &  I_{k-1}
		\end{array}\right)$$
	and $\mbf{0}_{k_1\times k_2}$ is a zero matrix of size $k_1\times k_2$.
	
	%Let $\cTt:=\{Y\vec{X}^{(1)},Y^2\vec{X}^{(2)}\}$.

	\begin{proposition}\label{171121-2006}
		The matrix ${\phi(M)|}_{\cT\setminus \{\vec{X}^{(0)}\},\cT}=\left(\begin{array}{cc} (\widetilde B)^T & \widetilde C\end{array}\right)$ satisfies the relations
		\begin{equation}\label{171121-2039}
			Y^2X^{i}=
			(\widetilde \alpha_2+\widetilde \alpha_3)YX^{i}-
			\widetilde \alpha_2\widetilde\alpha_3 X^i
		\end{equation}
		for every $i=0,\ldots,k-2$.
	\end{proposition}	
	
	\begin{proof}
	 The relations \eqref{171121-1921} imply the statement of the proposition.
	\end{proof}
%	where
%	\begin{align*}
%		\widetilde A_{01}&=A_{01}-\alpha_1 \cdot A_{00}|_{\{\vec{X}\},\{\vec{X}^{(1)}\}}
%			=\left(\begin{array}{c} \widehat A_{01} \\ \widetilde a_{01}^T\end{array}\right),
%			=\left(\begin{array}{c} \widehat A_{01} \\ \begin{array}{cc}\widehat a_{01}^T & 
%				 \widetilde \beta_{2k-1,1}\end{array} \end{array}\right),\\
%		\widetilde A_{02}&=A_{02}-\alpha_1 \cdot A_{01}|_{\{\vec{X}\},\{Y\vec{X}^{(1)}'\}}
%			=\left(\begin{array}{c} \widehat A_{02} \\ \widetilde a_{02}^T\end{array}\right),\\
%		\widetilde A_{11}&=A_{11}-2\alpha_1 \cdot A_{01}^T|_{\{Y\vec{X}^{(1)}\},\{\vec{X}^{(1)}\}}+\alpha_1^2 \cdot A_{00}|_{\{\vec{X}^{(1)}\},\{\vec{X}^{(1)}\}},\\
%		\widetilde A_{12}&=A_{12}-2\alpha_1 \cdot A_{11}|_{\{Y\vec{X}^{(1)}\},\{Y\vec{X}^{(1)}'\}}+\alpha_1^2 \cdot A_{01}^T|_{\{Y\vec{X}^{(1)}\},\{\vec{X}^{(1)}'\}},\\
%		\widetilde A_{22}&=A_{22}-2\alpha_1 \cdot A_{12}^T|_{\{Y^2\vec{X}^{(1)}'\},\{Y\vec{X}^{(1)}'\}}+\alpha_1^2 \cdot A_{02}^T|_{\{Y^2\vec{X}^{(1)}'\},\{\vec{X}^{(1)}'\}},
%	\end{align*}
%	with matrices $\widehat A_{01}\in \RR^{k\times k}$, $\widehat A_{02}\in \RR^{k\times (k-1)}$ and vectors 
%		 $\widetilde a_{01}\in \RR^{k}$, $\widehat a_{01},\widetilde a_{02}\in \RR^{k-1}$.
	Note that $J$ is invertible and hence
	\begin{equation}\label{161121-1155}
		\Rank \phi(M)=\Rank M.
	\end{equation}
	Further on, we write
	\begin{align*}
		N
		&=
		%=(\phi(M))(\mathbf{t})|_{\{\vec{X}^{(1)},Y\vec{X}^{(1)},Y^2\vec{X}^{(1)}',Y^2X^{k-1}\}}=
		\bordermatrix{
		& \vec{X}^{(1)} & Y\vec{X}^{(1)} & Y^2 \vec{X}^{(2)} & Y^2X^{k-1}\cr\rule{0pt}{\normalbaselineskip} 
			(\vec{X}^{(1)})^T & \widehat A_{00} & \widehat A_{01} & \widehat A_{02} & c\cr\rule{0pt}{\normalbaselineskip} 
			(Y\vec{X}^{(1)})^T & (\widehat A_{01})^T & \widetilde A_{11} &  \widetilde A_{12} & b \cr\rule{0pt}{\normalbaselineskip} 
			(Y^2\vec{X}^{(2)})^T& (\widehat A_{02})^T & (\widetilde A_{12})^T & \widetilde A_{22} & a\cr\rule{0pt}{\normalbaselineskip} 
			Y^2X^{k-1}& c^T & b^T & a^T & \widetilde\beta_{2k-2,4}		
		}
		\in \RR^{(3k-3)\times (3k-3)},\\
		\widetilde{B}_{00}
		&=(\widetilde\alpha_2\widetilde\alpha_3)^{-1} \cdot \left((\widetilde\alpha_2+\widetilde \alpha_3)\cdot \widehat{A}_{01}-
			\left(\begin{array}{cc}\widehat A_{02} & c\end{array}\right)\right)
			\in \RR^{(k-1)\times (k-1)},\\
		h
		&=(\widetilde\alpha_2\widetilde\alpha_3)^{-1}\cdot \left((\widetilde\alpha_2+\widetilde \alpha_3)\cdot \widehat{a}_{01}-\widetilde a_{02}\right)
			\in \RR^{k-1},
		%\widehat C
		%&=
		%\left(\begin{array}{cc} 
		%		\widetilde C & d\\ 
		%		d^T & \widetilde \beta_{2k-2,4}
		%\end{array}\right)\in \RR^{(2k-2)\times (2k-2)},
	\end{align*}
	where
	\begin{align*}
	 \widehat A_{00}
		&= {(A_{00})}|_{\{\vec X^{(1)}\}},\quad 
	\widetilde A_{01}
		=\left(\begin{array}{c} \widehat A_{01} \\  \rule{0pt}{0.9\normalbaselineskip}
			 (\widetilde a_{01})^T\end{array}\right)
		=\left(\begin{array}{c} \widehat A_{01} \\   \rule{0pt}{0.9\normalbaselineskip}
			 \begin{array}{cc}(\widehat a_{01})^T & 	 \widetilde \beta_{2k-1,1}\end{array} \end{array}\right),\quad
	\widetilde A_{02}
		=\left(\begin{array}{c} \widehat A_{02} \\ \rule{0pt}{0.9\normalbaselineskip} (\widetilde a_{02})^T\end{array}\right),\\
	a
		&=
		\left(\begin{array}{cccc}  \widetilde \beta_{k-1,4}& \cdots &  \widetilde \beta_{2k-4,4} &  \widetilde \beta_{2k-3,4}\end{array}\right)^T,\qquad
b
		=
		\left(\begin{array}{cccc} \widetilde  \beta_{k-1,3}& \cdots & \widetilde  \beta_{2k-3,3} & \widetilde  \beta_{2k-2,3}\end{array}\right)^T,\\
		c
		&=
		\left(\begin{array}{ccc} \widetilde  \beta_{k-1,2}& \cdots & \widetilde  \beta_{2k-2,2} \end{array}\right)^T%\qquad\qquad	\quad	
		%d^T=\left(\begin{array}{cc} b^T & a^T\end{array}\right).	
	\end{align*}
	and
	\begin{align}\label{171121-2029}
	\begin{split}
		\widetilde\beta_{2k-3,4}	
			&=(\widetilde\alpha_2+\widetilde\alpha_3)\cdot\widetilde\beta_{2k-3,3}-\widetilde\alpha_2\widetilde\alpha_3\cdot\widetilde\beta_{2k-3,2},\\
		\widetilde\beta_{2k-2,3}	
			&=(\widetilde\alpha_2+\widetilde\alpha_3)\cdot\widetilde\beta_{2k-2,2}-\widetilde\alpha_2\widetilde\alpha_3\cdot\widetilde\beta_{2k-2,1},\\
		\widetilde\beta_{2k-2,4}
			&=(\widetilde\alpha_2+\widetilde\alpha_3)\cdot\widetilde\beta_{2k-2,3}-\widetilde\alpha_2\widetilde\alpha_3\cdot\widetilde\beta_{2k-2,2}.
	\end{split}
	\end{align}
%	Finally, we write
%		$$
%		B(\mathbf{t})=\left(\begin{array}{ccc} \widetilde B & c & \mathbf{t}\end{array}\right)\quad\text{where }\mathbf{t}\in \RR 
%		\qquad\text{and}\qquad
%		\widehat C=
%		\left(\begin{array}{cc} 
%				\widetilde C & d\\ 
%				d^T & \widetilde \beta_{2k-2,4}
%		\end{array}\right).$$
	
	\begin{theorem}\label{y3-2402} 
	Let $K:=\{(x,y)\in \RR^2\colon (y-\alpha_1)(y-\alpha_2)(y-\alpha_3=0\}$, $\alpha_1,\alpha_2,\alpha_3\in \RR$, $\alpha_1<\alpha_3$ and
	$\alpha_2<\alpha_3$,
	be a union of three parallel lines and $\beta:=\beta^{(2k)}=(\beta_{i,j})_{i,j\in \ZZ_+,i+j\leq 2k}$, where $k\geq 3$.
	Assume also the notation above. Then the following statement are equivalent:
	\begin{enumerate}	
		\item\label{pt1-0704-0915} $\beta$ has a $K$--representing measure.
		\item\label{pt2-0704-0915}   $M_k$ is positive semidefinite, the relations
	\begin{equation}\label{moment-relations-2303-1028}
		\beta_{i,j+3}=(\alpha_1+\alpha_2+\alpha_3)\cdot \beta_{i,j+2}-\alpha_1\alpha_2\alpha_3 \cdot \beta_{i,j+1}
	\end{equation}
	hold for every $i,j\in \ZZ_+$ with $i+j\leq 2k-3$,
%	 and one of the following statements holds: and recursively generated, satisfies the column relation 
%			\eqref{col-rel-1203-1831}, 
			$N$ is positive semidefinite and one of the following statements holds:
		\begin{enumerate}
			\item\label{pt2a-0704-0922}  $\widehat A_{01}-\widetilde{\alpha}_2 \widetilde B_{00}$ is invertible.
			\item\label{pt2b-0704-0922} $(\alpha_3-\alpha_1) \widetilde B_{00}-\widehat A_{01}$ is invertible.
			\item\label{cond-2403-1008}  
			Denoting by $v\in \RR^{k-2}$ any vector such that
			\begin{equation}\label{111121-1425}
					((\alpha_3-\alpha_1) \widetilde B_{00}-\widehat A_{01})
					\left(\begin{array}{cc}v^T & 1\end{array}\right)^T=0,
			\end{equation}
			defining real numbers
			\begin{align}\label{eq-t'-0304-1557}
				t'
				&=\frac{(\widehat a_{01})^Tv-\widetilde\alpha_3 h^T v+\widetilde\beta_{2k-1,1}}{\alpha_3-\alpha_1},\\
				u'
				&=
					\left(\begin{array}{ccc} h^T & t' & (\widetilde a_{01})^T  \end{array}\right)
					\left(\begin{array}{cc} 
						\widetilde B_{00} & \widehat A_{01} \\\rule{0pt}{\normalbaselineskip}
						(\widehat A_{01})^T & \widetilde A_{11} 
					\end{array}\right)^\dagger
					\left(\begin{array}{c} h \\ t' \\ \widetilde a_{01}\end{array}\right)\label{eq-u'-0304-1558}	
			\end{align}
			and a Hankel matrix
			\begin{equation}\label{181121-2108}
				A_{\gamma}:=
				A_{00}-
				\left(\begin{array}{c|c} \widetilde B_{00} & \begin{array}{c} h \\ t'\end{array}
							\\[2ex]
					\hline& \\[-0.8em]
					\begin{array}{cc} h^T & t'\end{array} & u' 
				\end{array}\right)\in S_{k+1},
			\end{equation}
			where $\gamma\in\RR^{2k+1},$ it holds that 
			\begin{equation}\label{111121-1729}
				A_{\gamma}\text{ is invertible} \qquad\text{or}\qquad \Rank A_{\gamma}=\Rank A_{\gamma}(k).
			\end{equation}
%			\begin{enumerate}
%				\item \label{cond-1603-0958} $A_{\gamma}$ is invertible,
%				\item \label{cond2-2103-2210}  $\Rank A_{\gamma}=\Rank A_{\gamma}(k)$.
%			\end{enumerate}
%			where $t'$ is defined by \eqref{eq-t'-0304-1557}.
		\end{enumerate}
	\end{enumerate}
	Moreover, if the $K$--representing measure for $\beta$ exists, then there is a $(\Rank M_k)$--atomic $K$--representing 
	measure if $\Rank N\leq \Rank M_k$
	and $(\Rank M_k+1)$--atomic otherwise.
\end{theorem}

The following corollary states that in the \textit{pure case}, i.e. the only column relations of $M_k$ come from the union of three lines by recursive generation,
the representing measure for $\beta$ supported on these lines exists if and only if $M_k$ is psd, rg and $N$ is psd. 

\begin{corollary}\label{cor-0304-2008}
	Let $K:=\{(x,y)\in \RR^2\colon (y-\alpha_1)(y-\alpha_2)(y-\alpha_3)=0\}$, $\alpha_1,\alpha_2,\alpha_3\in \RR$, $\alpha_1<\alpha_3$ and
	$\alpha_2<\alpha_3$,
	be a union of three parallel lines and $\beta:=\beta^{(2k)}=(\beta_{i,j})_{i,j\in \ZZ_+,i+j\leq 2k}$, where $k\geq 3$.
	Assume that $M_k$ has a column relation $(Y-\alpha_1)(Y-\alpha_2)(Y-\alpha_3)=\mbf{0}$, is recursively generated and 
	$(M_k)|_{\cC^{(2)}}\succ 0$. Assume also the notation of Theorem \ref{y3-2402}. Then $\beta$ has a $K$--representing measure if and only if 
	$N$ is positive semidefinite.	

	Moreover, if $\beta$ has a $K$--representing measure, then a $(\Rank M_k)$--atomic $K$--representing 
	measure exists.
\end{corollary}

\begin{remark}
	\begin{enumerate}
	\item\label{pt1-091221-2202}
		As already described in Section \ref{S1}, the main idea behind the proof of Theorem \ref{y3-2402} 
	is applying the ALT such that one of the lines becomes $y=0$ and then studying the existence of the decompositions 
	$\beta=\widetilde \beta+\widehat \beta$ such that $\widetilde\beta$, $\widehat \beta$ have representing measures 
	supported on $y=0$ and the union of two other horizontal lines, respectively. It turns out that all the moments of 
	$\widetilde\beta$, $\widehat \beta$ are uniquely determined except $\widetilde\beta_{2k-1,0}$, $\widetilde\beta_{2k,0}$
	$\widehat\beta_{2k-1,0}$, $\widehat\beta_{2k,0}$.
	Before proving Theorem \ref{y3-2402} we establish some preliminary results supporting this idea and
	characterizing the conditions $\widetilde\beta_{2k-1,0}$, $\widetilde\beta_{2k,0}$, $\widehat\beta_{2k-1,0}$, $\widehat\beta_{2k,0}$
	have to satisfy for the existence of the representing measure for $\beta$.
	In the proof of Theorem \ref{y3-2402} we show that these conditions are equivalent to the conditions in \eqref{pt2-0704-0915} of Theorem \ref{y3-2402}
	by using our solution to the TMP--2pl and the THMP.
	Then we establish Corollary \ref{cor-0304-2008} and 
	finally, we give some examples demonstrating the use of Theorem \ref{y3-2402}.
	\item\label{pt2-091221-2318}
		Let us compare our Theorem \ref{y3-2402} with Yoo's results \cite{Yoo17a}.
	In \cite{Yoo17a}, the sextic TMP--3pl is studied, i.e.\ $2k=6$. Since having another column relation in $M_3$ except the 
	one defining the given three parallel lines leads to either the singular quartic TMP solved in \cite{CF02}
	or the sextic TMP already covered by the results in \cite{CY15}, the author focuses on the pure case, when $M_3$ has only one column
	relation defined by the three lines. In the non-pure case he additionally assumes that the variety $\cV(M_3)$ is symmetric about the $y$-axis. 
	In the pure case he also reduces the problem to the study of the decompositions 
	$\beta=\widetilde \beta+\widehat \beta$ described in \eqref{pt1-091221-2202} above and separates the analysis
	to four cases according to the strict positivity or positivity of the determinants of certain submatrices of $M_3(\widetilde \beta)$ and $M_3(\widehat \beta)$.
	In the solutions of two of the cases, the existence of a representing measure is characterized by the solvability of the system of two inequalities in $\widetilde \beta_{5,0}$ and $\widetilde \beta_{6,0}$,
	which are quadratic in $\widetilde \beta_{5,0}$ and linear in $\widetilde \beta_{6,0}$ and hence easily handled. 
	However, our Corollary \ref{cor-0304-2008} above covers all four cases at once by only checking psd and rank conditions
	on two matrices which depend only on $\beta$. Moreover, Corollary \ref{cor-0304-2008} covers pure cases of all even degrees $2k\geq 6$.
	Concerning the non-pure cases, our Theorem  \ref{y3-2402} also covers all even degrees $2k\geq 6$ and we do not need to assume any
	additional condition on the variety $\cV(M_k)$. The conditions in Theorem  \ref{y3-2402} are also explicit in $\beta$ and numerically easy to test,
	since only positive (semi)definiteness and invertibility of certain matrices needs to be checked, a vector from a kernel of certain matrix and a Moore-Penrose pseudoinverse 
	of a certain matrix computed.
	\item 
		In \cite{Yoo17b} the author extends his results from \cite{Yoo17a} described in \eqref{pt2-091221-2318} above from the 
	sextic TMP--3pl to the sextic TMP on a reducible cubic column relation. Also here he focuses on the pure case and 	
	characterizes the existence of a representing measure by the solvability of the system of inequalities in $\widetilde \beta_{5,0}$ and $\widetilde \beta_{6,0}$.
	It would be interesting to study if our approach can be also used for the TMP on a reducible cubic column relation to extend the results from
	\cite{Yoo17b} to any even degree $2k\geq 6$ and also not necessarily pure case.  
	\item
		Similarly as in the case of the TMP--2pl we can ask ourselves what is the solution of the TMP--3pl 
	in case a sequence $\beta=(\beta_{i+1})_{i,j\in \ZZ_+,i+j\leq 2k-1}$ of degree $2k-1$, $k\geq 2$, is given.
	The question is when $\beta$ can be extended to the degree $2k$ sequence $\widetilde \beta$ satisfying the condition \eqref{pt2-0704-0915}
	in Theorem \ref{y3-2402}. Note that the moments $\widetilde\beta_{2k-j-3,j+3}$ with $0\leq j\leq 2k-3$ are uniquely determined by \eqref{moment-relations-2303-1028}.
	So the only undetermined moments are $\widetilde\beta_{2k,0}$, $\widetilde\beta_{2k-1,1}$ and $\widetilde\beta_{2k-1,2}$.
	Characterizing when $M_k$ and $N$ are psd in terms of $\widetilde\beta_{2k,0}$, $\widetilde\beta_{2k-1,1}$ and $\widetilde\beta_{2k-1,2}$ can be done, but since one of the entries of
	$\widetilde B_{00}$ also depends on $\widetilde\beta_{2k-1,2}$, it is not obvious when one of the conditions \eqref{pt2a-0704-0922}, \eqref{pt2b-0704-0922} or \eqref{cond-2403-1008}  
	will be true. This is a possible question for future research.
	\end{enumerate}
\end{remark}
 
	We define a matrix function 
		$$F:\RR^{(k+1)\times (k+1)}\to \RR^{3k\times 3k},\qquad
			F(\mathbf{Z})=\left(\begin{array}{cc}
			\mathbf{Z} 			& \widetilde B 		\\ \rule{0pt}{\normalbaselineskip}
			(\widetilde B)^T 		& \widetilde C 		
		\end{array}\right).$$
	We have that
	\begin{equation*}%\label{notation-Mk-0103-1803}
		\phi(M)=F(\mathbf{Z})+\left((A_{00}-\mathbf{Z})\oplus \mathbf{0}_{2k-1}\right),
	\end{equation*}
	where $\mathbf{0}_{2k-1}$ represents a $(2k-1)\times (2k-1)$ matrix with only zero entries.
	If $\phi(\beta)$ has a $K$--representing measure $\mu$, then it is supported on the union of parallel lines
	$y=0$, $y=\widetilde \alpha_2$ and $y=\widetilde \alpha_3$.
	Since the moment matrix generated by the measure supported on $y=0$ can be nonzero only
	when restricted to the columns and rows indexed by ${\vec{X}}^{(0)}$, it follows that the restriction of the
	moment matrix generated by $\mu|_{\{y=\widetilde \alpha_2\}\cup \{y=\widetilde \alpha_3\}}$ (resp.\ $\mu|_{\{y=0\}}$)  
	to the columns and rows from $\cT$ 
	is of the form 	$F(B_{00})$ (resp.\ 
	$(A_{00}-B_{00})\oplus \mathbf{0}_{2k-1}$), %for some $t\in \RR$ and 
	where $B_{00}\in S_{k+1}$
	is a Hankel matrix.
	%the atoms are of the form $(x,0)$ or $(x,\widetilde\alpha_1)$ or $(x, \widetilde\alpha_2)$. 
	%All the atoms of the form $(x,\widetilde\alpha_1)$ and $(x, \widetilde\alpha_2)$ generate the moment matrix
		%$F(B_{00})$ %for some $t\in \RR$ and 
	%for some Hankel matrix $B_{00}\in \RR^{(k+1)\times (k+1)}$, 
	%while the resatoms of the form $(x,0)$ generate
	%the moment matrix $(A_{00}-B_{00})\oplus 0_{2k-1}$.
	This discussion establishes the implication $(\Rightarrow)$ of the following lemma.

	\begin{lemma}	\label{existence-of-a-measure-2403-1104-v2}
		$\beta$ has a $K$--representing measure if and only if there exist a Hankel matrix 
		$B_{00}\in S_{k+1}$,
		such that:
		\begin{enumerate}
			\item  The sequence 
				%$\widehat{\beta}^{(2k)}=(\widehat\beta_{i,j})_{i,j\in \ZZ^2_+,i+j\leq 2k}$ 
				with the moment matrix $F(B_{00})$ has a $\widetilde{\phi(K)}$--representing measure, where
				\begin{equation}\label{variety-2403-0933}
					\widetilde{\phi(K)}=\{(x,y)\in \RR^2\colon (y-\widetilde\alpha_2)(y-\widetilde\alpha_3)=0\}.
				\end{equation}
			\item The sequence with the moment matrix $A_{00}-B_{00}$ has a $\RR$--representing measure.
		\end{enumerate}
	\end{lemma}

	\begin{proof}
	The implication $(\Rightarrow)$ follows from the discussion in the paragraph before the
	lemma. It remains to establish the implication $(\Leftarrow)$.
	Let $M_k^{(1)}$ (resp.\ $M_k^{(2)}$) be the moment matrix generated by the measure $\mu_1$ (resp.\ $\mu_2$) supported on $\widetilde{\phi(K)}$ (resp.\ $y=0$)
	such that ${(PM_k^{(1)}P^T)|}_{\cT}=F(B_{00})$ 
	(resp.\ ${(PM_k^{(2)}P^T)|}_{\cT}=(A_{00}-B_{00})\bigoplus \mathbf{0}_{2k-1}$),
	where $P$ is the permutation matrix such that moment matrices $PM_k^{(i)}P^T$, $i=1,2$, have rows and columns indexed in the order $\vec{X}^{(0)}$, $Y\vec{X}^{(1)}$, $Y^2\vec{X}^{(2)}$,$\ldots$, $Y^k$.
	Since all points on the variety $\widetilde{\phi(K)}\cup \{y=0\}$ satisfy the equations $y(y-\widetilde \alpha_2)(y-\widetilde \alpha_1)x^i=0$, $i=0,\ldots,k-3$, the moment matrix
	$\widetilde M_k=M_k^{(1)}+M_k^{(2)}$ corresponding to the measure $\mu_1+\mu_2$ has column relations \eqref{091121-2003}. 
	Since ${(P\phi(M_k)P^T)|}_{\cT}=\phi(M)={(P\widetilde{M}_kP^T)|}_{\cT}$ and
	both $\phi(M_k)$, $\widetilde{M}_k$ satisfy the column relations \eqref{091121-2003}, %$X^iY^2=X^iY$ for $i=0,\ldots,k-2$, 
	it follows that 
		$$P\phi(M_k)P^T=\left(\begin{array}{cc} \phi(M) & \phi(M)W \\ W^T \phi(M)  & W^T\phi(M)W\end{array}\right)=P\widetilde{M}_kP^T$$ 
	for some matrix $W\in \RR^{3k\times \frac{(k-2)(k-1)}{2}}$ and hence $\phi(M_k)=\widetilde M_k$. This concludes the proof of the implication $(\Leftarrow)$.
	\end{proof}
		 
	The existence of a $\widetilde{\phi(K)}$--representing measure for $F(B_{00})$ is characterized by 
	Theorem \ref{y2=1}, while the existence of a $\RR$--representing measure for $A_{00}-B_{00}$
	by Theorem \ref{Hamburger}. So it remains to study when there exists $B_{00}$ such that
	$F(B_{00})$ satisfies the condition \eqref{pt3-1803-0930} of Theorem \ref{y2=1} and  
	$A_{00}-B_{00}$ satisfies the condition \eqref{pt5-v2206} of Theorem \ref{Hamburger}.

	We define the matrix function 
	$$H:\RR^2\to \RR^{(k+1)\times(k+1)},\qquad 
		H(\mathbf{t},\mathbf{u})=\left(\begin{array}{c|c} \widetilde B_{00} & \begin{array}{c} h \\ \mathbf{t}\end{array}
						\\[2ex]
						\hline& \\[-0.8em]
						 \begin{array}{cc} h^T & \mathbf{t}\end{array}& \mathbf{u}\end{array}\right)=
						\left(\begin{array}{cc} \widetilde B_{00} & h(\mathbf{t})\\\rule{0pt}{0.9\normalbaselineskip}
						 (h(\mathbf{t}))^T & \mathbf{u}\end{array}\right),$$
	where $h(\mathbf{t})^T=\left(\begin{array}{cc} h^T & \mathbf{t} \end{array}\right)$.
%	\begin{equation}\label{def-alpha-0604-1410}
%		\alpha(\mathbf{t})=\displaystyle\frac{(\widetilde \alpha_2+\widetilde \alpha_3)\widetilde \beta_{2k-1,1}-\mathbf{t}}{\widetilde \alpha_2\widetilde \alpha_3}
%		\quad\text{and}\quad
%		h(\mathbf{t})^T=\left(\begin{array}{cc} h^T & \alpha(\mathbf{t})\end{array}\right).
%	\end{equation}
	The following lemma states that there are only two parameters in the matrix $B_{00}$ 
	%from Lemma \ref{existence-of-a-measure-2403-1104-v2}.
	which are not already determined by $\beta$.

	\begin{lemma}\label{lemma-0704-0926}
	Assume there is a Hankel matrix $B_{00}$ such that the sequence with the moment matrix $F(B_{00})$  admits a 
	$\widetilde{\phi(K)}$--representing measure $\mu_1$.
	Then 
	\begin{equation}\label{111121-1343}
		B_{00}=H(\widetilde \beta_{2k-1,0}(\mu_1),\widetilde \beta_{2k,0}(\mu_1)),
	 \end{equation}
%	\begin{equation}\label{decomp-1203-1843-v2}
%		B_{00}
%		=\begin{blockarray}{ccc}
%		& \vec{X}^{(1)} & X^{k}\\
%		\begin{block}{c(c|c)}
%		(\vec{X}^{(1)})^T & \widetilde B_{00} & 
%		\begin{array}{c} h \\ \alpha(\widetilde \beta_{2k-1,2}(\mu))\end{array}\\
%			\cline{2-3}
%			X^{k} & 
%		\begin{array}{cc} h^T & \alpha(\widetilde \beta_{2k-1,2}(\mu))\end{array} & \ast\\
%		\end{block}
%		\end{blockarray},
%	\end{equation}
	where $\widetilde \beta_{2k-1,0}(\mu_1)$ and $\widetilde \beta_{2k,0}(\mu_1)$ are the moments of the monomials $x^{2k-1}$ and $x^{2k}$, respectively, with respect to $\mu_1$.
\end{lemma}

\begin{proof}
	Let $M_{k+1}^{(1)}$ be the moment matrix generated by the measure $\mu_1$ and
	 $P$ the permutation matrix such that $PM_k^{(1)}P^T$ has rows and columns indexed in the order $\vec{X}^{(0)}$, $Y\vec{X}^{(1)}$, $Y^2\vec{X}^{(2)}$,$\ldots$, $Y^k$.
	%supported on $\widetilde{\phi(K)}$.
	We have that ${(PM_{k+1}^{(1)}P^{T})|}_{\cT}=F(B_{00})$ and ${(PM_{k+1}^{(1)}P^T)|}_{\cT\cup \{Y^2X^{k-1}\}}$ is equal to
	$$
		\kbordermatrix{
		& {\vec{X}}^{(0)} & \vr Y{\vec{X}}^{(1)}  & Y^2 {\vec{X}}^{(2)}  & Y^2X^{k-1}\\
			({\vec{X}}^{(0)})^T & B_{00} & \VR \widetilde A_{01} & \widetilde A_{02} & \begin{array}{c} c \\ \widetilde \beta_{2k-1,2}(\mu_1)\end{array}\\   [0.8em]
			\cline{2-6}\\[-0.8em]
			(Y{\vec{X}}^{(1)})^T & (\widetilde A_{01})^T & \VR \widetilde A_{11} &  \widetilde A_{12} & b \\ [0.8em]
			(Y^2{\vec{X}}^{(2)})^T& (\widetilde A_{02})^T & \VR (\widetilde A_{12})^T & \widetilde A_{22} & a\\ [0.8em]
			Y^2X^{k-1}& \begin{array}{cc} c^T & \widetilde \beta_{2k-1,2}(\mu_1)\end{array} & \VR b^T & a^T & \widetilde \beta_{2k-2,4}
		}.$$
%	Since $F(B_{00})$ admits a $\widetilde{\phi(K)}$--representing measure
%	If $B_{00}$ is a Hankel matrix, such that the sequence with the moment matrix $F(B_{00})$ admits a representing measure $\mu$ on $\widetilde{\phi(K)}$,
%	then the moment matrix $\widetilde F( \widetilde \beta_{2k-1,2}(\mu),B_{00})$, where $\widetilde \beta_{2k-1,2}(\mu)$ is the moment of the monomial 
%	$y^2x^{2k-1}$ w.r.t.\ $\mu$, 
	Since $\mu_1$ is supported on $\widetilde{\phi(K)}$, the matrix $M_{k+1}^{(1)}$ 
	satisfies the column relations 
	\begin{equation}\label{col-rel-0703-2225}
		Y^2X^{i}=( \widetilde \alpha_2+ \widetilde \alpha_3)\cdot YX^{i}- \widetilde \alpha_2 \widetilde \alpha_3 \cdot X^i\qquad \text{for}\quad i=0,\ldots,k-1,
	\end{equation}
	 and hence
	${B_{00}|}_{\vec{X}^{(1)}}=\widetilde B_{00}$ and ${B_{00}|}_{\{X^k\},\{\vec{X}^{(2)}\}}=h$.
	Thus, the equality \eqref{111121-1343} holds.
%	The equalities \eqref{rel-2403-2215} and \eqref{rel2-2403-2215} follow from the relations \eqref{col-rel-0703-2225}, which 
%	must be satisfied by the moment matrix $\widetilde F(\widetilde \beta_{2k-1,2}(\mu),B_{00})$.
%	%where $\mu$ is a measure supported on $\widetilde{\phi(K)}$ representing $\phi(\beta)$ and $\widetilde \beta_{2k-1,2}(\mu)$ is the moment of the monomial 
%	%$y^2x^{2k-1}$ generated by $\mu$. 
%	The equalities \eqref{decomp-1203-1843-v2} and \eqref{alpha-0703-2210} are direct consequences of 
%	\eqref{rel-2403-2215} and \eqref{rel2-2403-2215}.
\end{proof}
	
%By Lemma \ref{lemma-0704-0926}, all the entries of $B_{00}$ except the two in the right lower corner are completely determined by $\beta$. 
%We define the matrix function 
%	$$H:\RR^2\to \RR^{(k+1)\times(k+1)},\qquad 
%		H(\mathbf{t},\mathbf{u})=\left(\begin{array}{c|c} \widetilde B_{00} & \begin{array}{c} h \\ \alpha(\mathbf{t})\end{array}\\ 
%						\hline
%						 \begin{array}{cc} h^T & \alpha(\mathbf{t})\end{array}& \mathbf{u}\end{array}\right)=
%						\left(\begin{array}{cc} \widetilde B_{00} & h(\mathbf{t})\\
%						 h(\mathbf{t})^T & \mathbf{u}\end{array}\right),$$
%	where 
%	\begin{equation}\label{def-alpha-0604-1410}
%		\alpha(\mathbf{t})=\displaystyle\frac{(\widetilde \alpha_2+\widetilde \alpha_3)\widetilde \beta_{2k-1,1}-\mathbf{t}}{\widetilde \alpha_2\widetilde \alpha_3}
%		\quad\text{and}\quad
%		h(\mathbf{t})^T=\left(\begin{array}{cc} h^T & \alpha(\mathbf{t})\end{array}\right).
%	\end{equation}
By Lemmas \ref{existence-of-a-measure-2403-1104-v2} and \ref{lemma-0704-0926}, the existence of a $K$--representing measure for $\beta$ is equivalent to the existence of $t,u\in \RR$ such 
that $F(H(t,u))$ admits a $\widetilde{\phi(K)}$--representing measure and $A_{00}-H(t,u)$ admits a $\RR$--representing measure. 
The following proposition characterizes the existence of a $\widetilde{\phi(K)}$--representing measure for $F(H(t,u))$.
 
\begin{proposition}\label{prop-0304-1601}
	Assume the notation above and $t',u'\in \RR$ are real numbers.
	The sequence with the moment matrix $F(H(t',u'))$ admits a 
	$\widetilde{\phi(K)}$--representing measure if and only if 
	$F(H(t',u'))$ is psd and one of the following statements hold:
	\begin{enumerate}
		\item\label{pt1-0304-1543} 
			At least one of the matrices 
				$\widehat A_{01}-\widetilde\alpha_2\widetilde B_{00}$  or
				$\widetilde\alpha_3\widetilde B_{00}-\widehat A_{01}$
			is invertible.
		\item\label{pt1-2503-1406}  
			We have that $t'$ and $u'$ are of the form \eqref{eq-t'-0304-1557} and \eqref{eq-u'-0304-1558}, respectively.
			%$$\widetilde\alpha_3 h^T v+\widetilde\alpha_3 t'-\widehat a_{01}^Tv-\beta_{2k-1,1}=0$$
			%$$t'=\frac{\widehat a_{01}^Tv-\widetilde\alpha_3 h^T v-\beta_{2k-1,1}}{\widetilde \alpha_3}$$
%			\begin{align}%\label{eq-t'-0304-1557}
%				%t'	
%				%&=\frac{\widehat a_{01}^Tv-\widetilde\alpha_3 h^T v-\beta_{2k-1,1}}{\widetilde \alpha_3},\\
%				u'
%				&=
%					\left(\begin{array}{cc} h^T & t'  \end{array}\right)
%					\left(\begin{array}{cc} 
%						\widetilde B_{00} & \widehat A_{01} \\
%						\widehat A_{01} & \widetilde A_{11} \\
%					\end{array}\right)^\dagger
%					\left(\begin{array}{c} h \\ t' \end{array}\right).\label{eq-u'-0304-1558}
%			\end{align}
			%where
			%Both matrices 
			%	$\widetilde\alpha_3\widetilde B_{00}-\widehat A_{01}$ and 
			%	$\widehat A_{01}-\widetilde\alpha_2\widetilde B_{00}$ 
			%are singular. 
			%and $t',u'\in \RR$ are real numbers such that $F(H(t',u'))$ is positive semidefinite.
			%Then there exists a vector 
			%$v\in \RR^{k-2}$ is a vector, such that 
			%\begin{equation}\label{kernel-0304-1552}
			%	(\widetilde\alpha_3\widetilde B_{00}-\widehat A_{01})\left(\begin{array}{cc} v^T & 1\end{array}\right)^T=0.
			%\end{equation}
		\end{enumerate}
	Further on, if $F(H(t',u'))$ admits a 
	$\widetilde{\phi(K)}$--representing measure, then a $(\Rank F(H(t',u'))$--atomic $\widetilde{\phi(K)}$--representing measure exists.
\end{proposition}

\begin{proof}
	By Theorem \ref{y2=1}, $F(H(t',u'))$ admits a $\widetilde{\phi(K)}$--representing measure if and only if 
	$F(H(t',u'))$ is psd and one of the following statements holds:
	\begin{enumerate}[(i)]
		\item\label{101121-1419} $\widehat A_{01}-\widetilde\alpha_2\widetilde B_{00}$  or $\widetilde\alpha_3\widetilde B_{00}-\widehat A_{01}$ is invertible.
		\item\label{101121-1420}  $\Rank \left(\begin{array}{cc} H(t',u') & \widetilde A_{01}\\ \rule{0pt}{\normalbaselineskip}
					(\widetilde A_{01})^T &\widetilde A_{11}\end{array}\right)=
				\Rank \left(\begin{array}{cc} \widetilde{B}_{00} & \widehat A_{01}\\ \rule{0pt}{\normalbaselineskip}
					(\widehat A_{01})^T &\widetilde A_{11} \end{array}\right).$
	\end{enumerate}
	\eqref{101121-1419} is equivalent to \eqref{pt1-0304-1543}, 
	so it remains to prove that under the assumption that $F(H(t',u'))$ is psd and $\widehat A_{01}-\widetilde\alpha_2\widetilde B_{00}$, $\widetilde\alpha_3\widetilde B_{00}-\widehat A_{01}$
	are singular, \eqref{101121-1420} is equivalent to \eqref{pt1-2503-1406}.\\

	\noindent\textbf{Claim.} There exists $\gamma\in \RR$ such that 
		$$\cM^{(1)}:=\left(\begin{array}{cc} \widetilde\alpha_3\widetilde B_{00}-\widehat A_{01} & \widetilde\alpha_3 h(t')-\widetilde \alpha_{01}\\\rule{0pt}{\normalbaselineskip}
						 (\widetilde \alpha_3h(t')-\widetilde \alpha_{01})^T & \gamma\end{array}\right)$$ 
	is a psd Hankel matrix.\\

	\noindent \textit{Proof of Claim.}
	First we explain that $\cM^{(1)}$ is a Hankel matrix.  We have 
		$$\cM^{(1)}=\underbrace{\left(\begin{array}{cc} \widetilde\alpha_3\widetilde B_{00} & \widetilde\alpha_3 h(t')\\\rule{0pt}{\normalbaselineskip}
						 \widetilde \alpha_3 h(t')^T& \gamma\end{array}\right)}_{\cM^{(1,1)}}
					-
				\underbrace{\left(\begin{array}{cc} \widehat A_{01} &\widetilde \alpha_{01}\\\rule{0pt}{\normalbaselineskip}
						(\widetilde \alpha_{01})^T & 0\end{array}\right)}_{\cM^{(1,2)}}.$$
	Since $\cM^{(1,1)}$ differs from a Hankel matrix $B_{00}$ only in the last two cross--diagonals, it is also Hankel. 
	In the notation \eqref{vector-v}, 
	$\cM^{(1,1)}=A_{\widehat \beta^{(1)}}\in S_{k+1}$ and
	$\cM^{(1,2)}=A_{\widehat \beta^{(2)}}\in S_{k+1}$
	where
	\begin{align*}
		\widehat \beta^{(1)}
		&:=(\widetilde\alpha_3\phi(\beta)_{0,0},\widetilde\alpha_3\phi(\beta)_{1,0},\ldots,\widetilde\alpha_3\phi(\beta)_{2k-2,1},\widetilde\alpha_3 t',\gamma)\in \RR^{2k+1},\\
		\widehat \beta^{(2)}
		&:=(\phi(\beta)_{0,1},\phi(\beta)_{1,1},\ldots,\phi(\beta)_{2k-1,1},0)\in \RR^{2k+1}.
	\end{align*}
 	Thus, $\cM^{(1)}$ is Hankel.

	It remains to prove that $\cM^{(1)}$ is psd for some $\gamma\in \RR$.
	Since $F(H(t',u'))$ is psd, we have that
	\begin{align*}
		G(t',u')
		&:=	\left(\begin{array}{cc}I_k & -\widetilde\alpha_2(\widetilde\alpha_3-\widetilde\alpha_2)^{-1} \cdot I_k \\\rule{0pt}{\normalbaselineskip}
			\mathbf{0} & 
			(\widetilde\alpha_3-\widetilde\alpha_2)^{-1} \cdot I_k
		\end{array}\right)^T
			F(H(t',u'))
			\left(\begin{array}{cc}I_k & -\widetilde\alpha_2(\widetilde\alpha_3-\widetilde\alpha_2)^{-1} \cdot I_k \\\rule{0pt}{\normalbaselineskip}
			\mathbf{0} & 
			(\widetilde\alpha_3-\widetilde\alpha_2)^{-1} \cdot I_k
		\end{array}\right)\\
		&=\left(
			\begin{array}{cc}
				H(t',u') & G_{12}(t') \\\rule{0pt}{\normalbaselineskip}
				 (G_{12}(t'))^T & G_{22}
			\end{array}
			\right)
	\end{align*}
	is psd, where
	\begin{align}\label{def-F12-F22-2403-1003}
	\begin{split}
		G_{22}
		&=(\widetilde\alpha_3-\widetilde\alpha_2)^{-1}(\widehat A_{01}-\widetilde\alpha_2\widetilde B_{00}),\\
		G_{12}(t')
		&=\left(\begin{array}{c}
				G_{22}\\ (g_{12}(t'))^T
				\end{array}
				\right)
		=\left(\begin{array}{c}
				(\widetilde\alpha_3-\widetilde\alpha_2)^{-1}(\widehat A_{01}-\widetilde\alpha_2\widetilde B_{00})\\\rule{0pt}{\normalbaselineskip}
				(\widetilde\alpha_3-\widetilde\alpha_2)^{-1}(\widetilde a_{01}-\widetilde\alpha_2 h(t'))^T
				\end{array}
				\right).
	\end{split}
	\end{align}
	By the equivalence between \eqref{pt1-281021-2128} and \eqref{pt3-281021-2128}
	of Theorem \ref{block-psd} used for $\cM=G(t',u')$, it follows that 
		$K(t',u'):= G(t',u')/G_{22}$ is psd. We have that
	\begin{align*}\label{eq-0204-2106}
	\begin{split}
		K(t',u')
		&= G(t',u')/G_{22}
			=H(t',u')-G_{12}(t')G_{22}^{\dagger}(G_{12}(t'))^T\\[0.5em]
		&=H(t',u')-
			\left(\begin{array}{cc}
				G_{22}G_{22}^\dagger G_{22}& G_{22}G_{22}^\dagger g_{12}(t') \\\rule{0pt}{1.3\normalbaselineskip}
				(g_{12}(t'))^TG_{22}^\dagger G_{22} & (g_{12}(t'))^T G_{22}^{\dagger}g_{12}(t')
			\end{array}
			\right)\\[0.5em]
		&=\left(\begin{array}{cc}
				\widetilde B_{00}-G_{22}& h(t')-G_{22}G_{22}^\dagger g_{12}(t') \\\rule{0pt}{1.3\normalbaselineskip}
				(h(t')-G_{22}G_{22}^\dagger g_{12}(t'))^T & u' - (g_{12}(t'))^T G_{22}^{\dagger}g_{12}(t')
			\end{array}
			\right)\\[0.5em]
		&=\left(\begin{array}{cc}
				(\widetilde\alpha_3-\widetilde\alpha_2)^{-1}(\alpha_3\widetilde B_{00}-\widehat A_{01})& h(t')-G_{22}G_{22}^\dagger g_{12}(t') \\\rule{0pt}{1.3\normalbaselineskip}
				(h(t')-G_{22}G_{22}^\dagger g_{12}(t'))^T & u' - (g_{12}(t'))^T G_{22}^{\dagger}g_{12}(t')\end{array}\right)\\[0.5em]
		&=\left(\begin{array}{cc}
				(\widetilde\alpha_3-\widetilde\alpha_2)^{-1}(\widetilde\alpha_3\widetilde B_{00}-\widehat A_{01}) & (\widetilde\alpha_3-\widetilde\alpha_2)^{-1}(\widetilde\alpha_3h(t')-\widetilde a_{01}) \\\rule{0pt}{1.3\normalbaselineskip}
				(\widetilde\alpha_3-\widetilde\alpha_2)^{-1}(\widetilde\alpha_3h(t')-\widetilde a_{01})^T & u' - (g_{12}(t'))^T G_{22}^{\dagger}g_{12}(t')
			\end{array}
			\right),
		\end{split}
		\end{align*}
	where in the fourth equality we used that $G_{22}G_{22}^\dagger G_{22}=G_{22}$, while in the last equality we used $G_{22}G_{22}^\dagger g_{12}(t')=g_{12}(t')$ 
	(which is true since $g_{12}(t')\in \cC(G_{22})$ by $G(t',u')$ being psd) and \eqref{def-F12-F22-2403-1003}.
	Defining $\gamma:=(\widetilde\alpha_3-\widetilde\alpha_2)(u' - (g_{12}(t'))^T G_{22}^{\dagger}g_{12}(t'))$ we have that $(\widetilde\alpha_3-\widetilde\alpha_2) K(t',u')=\cM^{(1)}$, which proves the claim.
	\hfill $\blacksquare$\\

	Since $\cM^{(1)}=A_{\widehat \beta^{(1)}-\widehat \beta^{(2)}}\in S_{k+1}$ from Claim is psd and by assumption 
	$A_{\widehat \beta^{(1)}-\widehat \beta^{(2)}}(k-1)=\widetilde\alpha_3\widetilde B_{00}-\widehat A_{01}$ is singular, 
	it follows by Corollary \ref{rank-theorem-2} used for 
	$\beta=\widehat \beta^{(1)}-\widehat \beta^{(2)}$
	that 
	$\Rank A_{\widehat \beta^{(1)}-\widehat \beta^{(2)}}(k-1)=
	\Rank A_{\widehat \beta^{(1)}-\widehat \beta^{(2)}}(k-2)$.
	Therefore
	there is a vector 
		$v\in \RR^{k-2}$
	such that \eqref{111121-1425} holds.
	By Lemma \ref{extension-principle} used for $\cA=\cM^{(1)}$, it follows that 
	$\cM^{(1)}\left(\begin{array}{ccc} v^T & 1 & 0 \end{array}\right)^T=0$
	which is equivalent to $t'$ being of the form \eqref{eq-t'-0304-1557}.
	
	It remains to prove that $u'$ is of the form \eqref{eq-u'-0304-1558}.
	We have that
	\begin{align*}
	 &\Rank \left(\begin{array}{cc} H(t',u') & \widetilde A_{01}\\ [0.8em] %\rule{0pt}{\normalbaselineskip}
			(\widetilde A_{01})^T&\widetilde A_{11}\end{array}\right)
	=
	\Rank \left(\begin{array}{ccc} \widetilde B_{00} & h(t') & \widehat A_{01} \\ [0.8em]  %\rule{0pt}{1.3\normalbaselineskip} 
			(h(t'))^T & u' & (\widetilde a_{01})^T\\  [0.8em]  %\rule{0pt}{1.3\normalbaselineskip}
			(\widehat A_{01})^T&\widetilde a_{01} & \widetilde A_{11}\end{array}\right)\\
	=
	&\Rank \underbrace{\left(\begin{array}{cc|c} \widetilde B_{00} & \widehat A_{01} & h(t')\\ [0.8em] %\rule{0pt}{1.2\normalbaselineskip}
			(\widehat A_{01})^T&\widetilde A_{11} & \widetilde a_{01}
			\\[0.2ex]
						\hline& \\[-0.8em]
			h(t')^T & (\widetilde a_{01})^T & u' \end{array}\right)}_{\cM^{(2)}}
	=\Rank \left(\begin{array}{cc} \widetilde{B}_{00} & \widehat A_{01}\\ \rule{0pt}{\normalbaselineskip}
					(\widehat A_{01})^T &\widetilde A_{11} \end{array}\right),
	\end{align*}
	where the first equality follows by definitions of $H(t',u')$ and $\widetilde A_{01}$, the second by permuting rows and columns
	and the last by  the assumption \eqref{101121-1420}. 
	Finally, using \eqref{021121-1052} of Theorem \ref{block-psd} for $\cM= \cM^{(2)}$ implies that
	$u'$ is of the form \eqref{eq-u'-0304-1558}.
\end{proof}

\begin{remark}\label{171121-1959}
	 If
		$\widetilde\alpha_3\widetilde B_{00}-\widehat A_{01}$ 
	is singular, then there is at most one $t'\in \RR$ such that $F(H(t',u'))$ is psd for some $u'\in \RR$.
	Indeed, observing the first paragraph after the proof of Claim in the proof of Proposition \ref{prop-0304-1601} we see that
	$v\in \RR^{k-2}$ was an arbitrary vector such that $(\widetilde \alpha_3 \widetilde{B}_{00}-\widehat A_{01})\left(\begin{array}{cc} v^T & 1\end{array}\right)^T=0$.
	If $F(H(t',u'))$ and $F(H(t'',u''))$ are both psd for some $t',t'',u',u''\in \RR$, we see that the equality
	$\left(\begin{array}{cc}\widetilde \alpha_3 h & \widetilde \alpha_3 t'\end{array}\right)\left(\begin{array}{cc} v^T & 1\end{array}\right)^T=
	\left(\begin{array}{cc}\widetilde \alpha_3 h & \widetilde \alpha_3 t''\end{array}\right)\left(\begin{array}{cc} v^T & 1\end{array}\right)^T=0
	$ holds, which implies $t'=t''$.
\end{remark}

Now we are ready to prove Theorem \ref{y3-2402}.

\begin{proof}[Proof of Theorem \ref{y3-2402}]
	First we prove the implication $\eqref{pt1-0704-0915}\Rightarrow\eqref{pt2-0704-0915}$. We denote by $M_{k+1}$ the moment matrix 
	associated to the sequence generated by some $K$--representing measure $\mu$ for $\beta$.
	The following statements hold:
	\begin{itemize}
		\item The matrix $M_k$ is psd (see Section \ref{S1}).
		\item The extension of ${\phi(M)|}_{\cT\setminus \{X^k\}}$ with a row and column $Y^2X^{k-1}$ is equal to the matrix $N$ due to the relation
			$Y(Y-\widetilde \alpha_2)(Y-\widetilde \alpha_3)X^{k-2}$ which is satisfied by the moment matrix $M_{k+1}$. 
		\item The matrix $N$ is psd as the restriction of $M_{k+1}$.
	\end{itemize}
	%$M_k$ and $(\phi(M))(\widetilde \beta_{2k-1,2})$ must be psd (see Section \ref{S1}), where $\widetilde \beta_{2k-1,2}$
	%is the moment of the monomial $x^{2k-1}y^2$ w.r.t.\ $\mu$. In particular, $N$ must be psd being the principal submatrix of $(\phi(M))(\widetilde \beta_{2k-1,2})$.
	Using Lemmas \ref{existence-of-a-measure-2403-1104-v2} and \ref{lemma-0704-0926}, %and Proposition \ref{prop-0304-1601},
 	there exist  $t',u'\in \RR$ such that the sequence with the moment matrix $F(H(t',u'))$ 
	admits a $\widetilde{\phi(K)}$--representing measure. 
	If \eqref{pt2a-0704-0922} or \eqref{pt2b-0704-0922} of Theorem \ref{y3-2402} holds, we are done. 
	Otherwise Proposition \ref{prop-0304-1601} implies that $t'$, $u'$ 
	are of the forms \eqref{eq-t'-0304-1557}, \eqref{eq-u'-0304-1558}, respectively. By Lemma \ref{existence-of-a-measure-2403-1104-v2}, $A_{00}-H(t',u')$ admits a representing measure on $\RR$,
	which by Theorem \ref{Hamburger} implies \eqref{111121-1729} and thus \eqref{cond-2403-1008} of Theorem \ref{y3-2402} holds.

	It remains to prove the implication $\eqref{pt1-0704-0915}\Leftarrow\eqref{pt2-0704-0915}$. 
	We denote by $(\phi(M))(\mathbf{t})$ the moment matrix which extends $\phi(M)$ with the additional column and row $Y^2X^{k-1}$:
		%$\phi(M)(\beta_{2k-1,2}(\mu))$ :
		\begin{align*}%\label{notation-Mk-0203-1123}
		%\begin{split}
			(\phi(M))(\mathbf{t})	
			&=
			\kbordermatrix{
			& \vec{X}^{(0)} & \vr Y\vec{X}^{(1)} & Y^2 \vec{X}^{(2)} & Y^2X^{k-1} \\
				(\vec{X}^{(0)})^T & A_{00} & \VR \widetilde A_{01} & \widetilde A_{02} & \begin{array}{c} c \\\mathbf{t} \end{array}\\ [0.8em]
				\cline{2-6}\\[-0.8em]			(Y\vec{X}^{(1)})^T & (\widetilde A_{01})^T & \VR \widetilde A_{11} &  \widetilde A_{12} & b \\ [0.8em]
				(Y^2 \vec{X}^{(2)})^T& (\widetilde A_{02})^T & \VR (\widetilde A_{12})^T & \widetilde A_{22} & a\\[0.8em]
				Y^2X^{k-1}& \begin{array}{cc} c^T & \mathbf{t}\end{array} & \VR b^T & a^T & \widetilde{\beta}_{2k-2,4}
			}
			=
			\left(\begin{array}{c|c}
				A_{00} & \widetilde B(\mathbf{t})\\[0.3em]
					\hline& \\[-0.8em]
				(\widetilde B(\mathbf{t}))^T & \widehat C
			\end{array}\right).
		\end{align*}
	Note that $N$ is equal to the restriction
		${(\phi(M))(\mathbf{t})|}_{\cT\cup\{Y^2X^{k-1}\}\setminus \{X^k\}}.$
	By \cite[Theorem 2.3]{Dan92},
	there exists $t_0\in\RR$ such that $(\phi(M))(t_0)\succeq 0$ and
		$$ 
		%\quad\text{and}\quad
		\Rank (\phi(M))(t_0)=\max(\Rank \phi(M), \Rank N)=\max(\Rank M, \Rank N),$$
	where we used \eqref{161121-1155} in the second equality.
	Let $\widetilde{\cC^{(2)}}:=(Y\vec{X}^{(1)},Y^2\vec{X}^{(2)},Y^2X^{k-1})$ and 
		$$\cM^{(1)}=\kbordermatrix{& \vec{X}^{(0)} & \widetilde{\cC^{(2)}} \\
				(\vec{X}^{(0)})^T & \widetilde B(t_0)
				({\widehat C})^\dagger
				(\widetilde B(t_0))^T & \widetilde B(t_0)\\ \rule{0pt}{1.2\normalbaselineskip} 
				(\widetilde{\cC^{(2)}})^T & (\widetilde B(t_0))^T & \widehat C
			}.$$
	By the equivalence between \eqref{pt1-281021-2128} and \eqref{pt2-281021-2128} of Theorem \ref{block-psd}
	used for $\cM=\cM^{(1)}$,
	we have that $\cM^{(1)}\succeq 0$.
	By \eqref{prop-2604-1140-eq2} of Proposition \ref{prop-2604-1140} used for $\cM=\cM^{(1)}$, we have that 
	\begin{equation}\label{161121-1403}
		\Rank \cM^{(1)}=\Rank \widehat C.
	\end{equation}

	\noindent \textbf{Claim 1.}
	The matrix $\cM^{(1)}$ satisfies the column relations $\eqref{171121-2039}$ for $i=0,\ldots,k-1$.\\ 

	\noindent \textit{Proof of Claim 1.} 
	Since $\Rank \cM^{(1)}=\Rank \widehat C$, there is a matrix $W\in \RR^{2k\times (k+1)}$ such that
		$$\cM^{(1)}=\left(\begin{array}{cc}
				W^T\widehat CW & W^T\widehat C\\
				\widehat CW & \widehat C
			\end{array}\right).$$
	Morever, if $W'$ is any matrix satisfying $\widehat CW'=(\widetilde B(t_0))^T$, then 
		$$(W')^T\widehat C W'=(W')^T\widehat CW=W^T\widehat CW,$$
	where we used $\widehat CW'=\widehat CW$ in the first equality and $(W')^T\widehat C=W^T \widehat C$ in the second.
	
	Relations \eqref{171121-1921} and definitions \eqref{171121-2029} 
	imply that the restriction 
		$$(\cM^{(1)})|_{\widetilde{\cC^{(2)}},\cC^{(2)}\cup\{Y^2X^{k-1}\}  }=
		\left(\begin{array}{cc}
				(\widetilde B(t_0))^T & \widehat C
			\end{array}\right)$$
	satisfies the relations \eqref{171121-2039} for $i=0,\ldots,k-1$. 
	Thus, there is  $s\in \RR^{2k}$ such that $(\widetilde B(t_0))^T=\widehat CW'$ where
		$$W'=\left(\begin{array}{c|c} \begin{array}{c}
		(\widetilde \alpha_2+ \widetilde \alpha_3) (\widetilde \alpha_2 \widetilde \alpha_3)^{-1}\cdot
		I_{k} \\[0.3em]
				-(\widetilde \alpha_2 \widetilde \alpha_3)^{-1}\cdot I_{k} \end{array}& 
				s \end{array}\right)\in \RR^{2k\times(k+1)}$$
	and finally $W^T\widehat C W=(W')^T\widehat C W'$.
	Hence, $\cM^{(1)}$ also satisfies the relations \eqref{171121-2039}  for $i=0,\ldots,k-1$.\hfill $\blacksquare$\\

	By \eqref{161121-1403}, the column space $\cC(\cM^{(1)})$ is spanned by the columns in the set $\{Y\vec{X}^{(1)},Y^2\vec{X}^{(1)}\}$. 
	Using also Claim 1, the column space $\cC(\cM^{(1)})$ is spanned by the columns in the set $\{\vec{X}^{(1)},Y\vec{X}^{(1)}\}$.
	Therefore 
	\begin{equation}\label{171121-2208}
		\Rank ({(\cM^{(1)})|}_{\{\vec{X}^{(1)},Y\vec{X}^{(1)}\}})=\Rank (\cM^{(1)}).
	\end{equation}
	
	\noindent \textbf{Claim 2.}
	There exists $u_0\in \RR$ such that
	\begin{equation}\label{h(t,gamma)-0903-0808}
		H(\alpha(t_0),u_0)=\widetilde B(t_0)
				(\widehat C)^\dagger
				\widetilde B(t_0)^T,
	\end{equation}
	where $\alpha(t_0)=(\widetilde \alpha_2\widetilde \alpha_3)^{-1}\cdot\left(
		(\widetilde \alpha_2+\widetilde \alpha_3)\widetilde \beta_{2k-1,1}-t_0\right)
		$.\\

	\noindent \textit{Proof of Claim 2.} 
	By the definition of $\alpha(t_0)$ and $H(\mathbf{t},\mathbf{u})$, it follows that $H(\alpha(t_0),\mathbf{u})$ 
	agrees with $\widetilde B(t_0)
				(\widehat C)^\dagger
				(\widetilde B(t_0))^T$
	in the first $k$ columns and rows. So the only remaining entry is the lowest right corner which can be obviously chosen such that
	\eqref{h(t,gamma)-0903-0808} holds. This proves the claim.\hfill $\blacksquare$\\

%	By Claim 2,
%	\begin{equation}\label{rank-f(t0,gamma0)-0903-0945}
%		\Rank \cM^{(1)}=\Rank \left(\begin{array}{cc}
%				H(\alpha(t_0),u_0) & \widetilde B(t_0)\\
%				\widetilde B(t_0)^T & \widehat C
%			\end{array}\right)=
%			\Rank \widehat C.
%	\end{equation}
%	Using Claims 1 and 2, we conclude that
%	\begin{equation}\label{rank-f(t0,gamma0)-0903-0945-v2}
%		\Rank \widehat C=\Rank \cM^{(1)}=\Rank F(H(\alpha(t_0),u_0)).
%	\end{equation}
	We write 
		$$A_{00}=
		\bordermatrix{
		& \vec{X}^{(1)} & X^{k}\cr
		(\vec{X}^{(1)})^T & \widetilde A_{00} & \widetilde a_{00}\cr  \rule{0pt}{0.8\normalbaselineskip} 
			X^{k} & (\widetilde a_{00})^T & \beta_{2k,0}}
		$$
		and define the matrix function
	$$U:\RR^2\to \RR^{(k+1)\times (k+1)},\qquad 
		U(\mathbf{t},\mathbf{u}):=A_{00}-H(\mathbf{t},\mathbf{u})=\left(\begin{array}{cc} 
			\widetilde A_{00} - \widetilde B_{00} & \widetilde a_{00}-h(\mathbf{t}) \\ \rule{0pt}{\normalbaselineskip} 
			(\widetilde a_{00}-h(\mathbf{t}))^T & \beta_{2k,0}- \mathbf{u}
		\end{array}\right).$$
	By the equivalence between \eqref{pt1-281021-2128} and \eqref{pt2-281021-2128} of Theorem \ref{block-psd}
	used for $\cM=(\phi(M))(t_0)$,
	it follows that $A_{00}\succeq \widetilde B(t_0)
				(\widehat C)^\dagger
				\widetilde B(t_0)^T$ or equivalently $U(\alpha(t_0),u_0)\succeq 0$,
	where $\alpha(t_0)$ and $u_0$ are as in Claim 2.
	By the equivalence between \eqref{pt1-281021-2128} and \eqref{pt3-281021-2128} of Theorem \ref{block-psd}
	used for $U(\alpha(t_0),u_0)$, it follows that
	\begin{equation}\label{delta-0903-0820}
		\delta:=(\beta_{2k,0}-u_0)-
		(\widetilde a_{00}-h(\alpha(t_0)))^T(\widetilde A_{00} - \widetilde B_{00})^\dagger (\widetilde a_{00}-h(\alpha(t_0)))\geq 0
	\end{equation}
	and by \eqref{021121-1052} of Theorem \ref{block-psd} used for $\cM=U(\alpha(t_0),u_0)$, we have that
	\begin{equation}\label{rank-l(t0,gamma0)-0903-1000}
		\Rank U(\alpha(t_0),u_0)=
			\left\{\begin{array}{rl} 
				\Rank (\widetilde A_{00} - \widetilde B_{00}),&	\text{if }\delta=0,\\[0.3em]
				\Rank (\widetilde A_{00} - \widetilde B_{00})+1,&	\text{if }\delta>0.
			\end{array} \right.
	\end{equation}
	By \eqref{prop-2604-1140-eq2} of Proposition \ref{prop-2604-1140} used for $\cM=(\phi(M))(t_0)$,
	we have that
%	Note also that 
%	\begin{equation}\label{col-space-0903-0844}
%		\cC((\phi(M))(t_0))= 
%				\cC(\left(\begin{array}{cc}
%				A_{00}-H(\alpha(t_0),u_0) & B(t_0)\\
%				0 & \widehat C \\
%				\end{array}\right))=
%				\cC(\left(\begin{array}{cc}
%				U(\alpha(t_0),u_0) & B(t_0)\\
%				0 & \widehat C
%				\end{array}\right)),
%	\end{equation}	
%	and 
	\begin{equation}\label{rank-0903-0940}
			\Rank ((\phi(M))(t_0))=
			\Rank (U(\alpha(t_0),u_0)) +
			\Rank \widehat C.
	\end{equation}	
	We separate two cases:
	\begin{enumerate}
		\item[\textbf{Case 1:}] $\widehat A_{01}-\widetilde\alpha_2\widetilde B_{00}$ or $\widetilde\alpha_3\widetilde B_{00}-\widehat A_{01}$ is invertible.
		\item[\textbf{Case 2:}] $\widehat A_{01}-\widetilde\alpha_2\widetilde B_{00}$ and $\widetilde\alpha_3\widetilde B_{00}-\widehat A_{01}$ are singular.\\
	\end{enumerate}

	\noindent \textbf{Case 1.} 
	By definition of $\delta$ \eqref{delta-0903-0820}, it follows that 
	$U(\alpha(t_0),u_0+\delta)/(\widetilde A_{00} - \widetilde B_{00})=0$.
	We use Theorem \ref{block-psd} for $\cM=U(\alpha(t_0),u_0+\delta)$ twice:
	\begin{itemize}
		\item The equivalence between \eqref{pt1-281021-2128} and \eqref{pt3-281021-2128} implies that
	\begin{equation}\label{rank-2403-947}
		U(\alpha(t_0),u_0+\delta)\succeq 0.
	\end{equation}
		\item \eqref{021121-1052} gives   
	\begin{equation}\label{rank-2403-947-v2}
		\Rank U(\alpha(t_0),u_0+\delta)=\Rank (\widetilde A_{00} - \widetilde B_{00}).
	\end{equation}
	\end{itemize}
	In the notation \eqref{vector-v}, $U(\alpha(t_0),u_0+\delta)=A_{\widehat \beta^{(1)}}\in S_{k+1}$
	for some $\widehat \beta^{(1)}\in \RR^{2k+1}$.
	Using \eqref{rank-2403-947} and \eqref{rank-2403-947-v2} we have that  $A_{\widehat \beta^{(1)}}\succeq 0$ and $\Rank A_{\widehat \beta^{(1)}}(k-1)=\Rank A_{\widehat \beta^{(1)}}$.
	By Theorem \ref{Hamburger} used for $\widehat \beta^{(1)}$, 
	$\widehat \beta^{(1)}$ has a 
		$(\Rank (\widetilde A_{00} - \widetilde B_{00}))$--atomic $\RR$-representing measure.
	We have that
	\begin{equation}\label{rank-f(h(t0,gamma0+delta))-0903-1003} 
		r:=\Rank F(H(\alpha(t_0),u_0+\delta))=
		\left\{\begin{array}{rl} 
			\Rank \widehat C,&		\text{if }\delta=0,\\[0.3em]
			\Rank \widehat C+1,&	\text{if }\delta>0,
		\end{array}\right.
	\end{equation}
	where we used that ${(\cM^{(1)})|}_{\cT}=F(H(\alpha(t_0),u_0)$ and \eqref{171121-2208}.
	By Proposition \ref{prop-0304-1601}, the sequence with the moment matrix $F(H(\alpha(t_0),u_0+\delta))$ admits a $r$--atomic representing measure on the variety 
	$\widetilde{\phi(K)}$ \eqref{variety-2403-0933}.
	The equalities 
	\eqref{rank-l(t0,gamma0)-0903-1000}, \eqref{rank-0903-0940}, 
	\eqref{rank-2403-947} and \eqref{rank-f(h(t0,gamma0+delta))-0903-1003} 
	imply that the sequence $\phi(\beta)$ admits a $(\Rank \phi(M)(t_0))$--atomic representing measure.
	This concludes the proof in this case.\\

	\noindent \textbf{Case 2.} 
	Since $ \widehat A_{01}-\widetilde\alpha_2\widetilde B_{00}$, $\widetilde\alpha_3\widetilde B_{00}- \widehat A_{01}$ are singular
	and $F(\alpha(t_0),u_0)$ is psd, Remark \ref{171121-1959} implies that $\alpha(t_0)$ is equal to $t'$
	defined by \eqref{eq-t'-0304-1557}. 
	The equality \eqref{171121-2208} implies that 
	\begin{align*}
	\Rank \left(\begin{array}{cc} \widetilde B_{00} & \widehat A_{01}\\  \rule{0pt}{\normalbaselineskip} 
			(\widehat A_{01})^T&\widetilde A_{11}\end{array}\right)
			&=
	\Rank \left(\begin{array}{ccc} \widetilde B_{00} & h(t') & \widehat A_{01} \\ \rule{0pt}{\normalbaselineskip} 
			(h(t'))^T & u_0 & (\widetilde a_{01})^T\\ \rule{0pt}{\normalbaselineskip} 
			(\widehat A_{01})^T&\widetilde a_{01} & \widetilde A_{11}\end{array}\right)\\
			&=
	\Rank \underbrace{\left(\begin{array}{cc|c} \widetilde B_{00} & \widehat A_{01} & h(t')\\  \rule{0pt}{\normalbaselineskip} 
			(\widehat A_{01})^T&\widetilde A_{11} & \widetilde a_{01}\\\hline\rule{0pt}{\normalbaselineskip} 
			(h(t'))^T & (\widetilde a_{01})^T & u_0 \end{array}\right)}_{\cM^{(2)}},
	\end{align*}
	where the second equality follows by permuting rows and columns.
	Using \eqref{021121-1052} of Theorem \ref{block-psd} for $\cM= \cM^{(2)}$,
	$u_0$ is equal to $u'$ defined by \eqref{eq-u'-0304-1558}.
	Now by Proposition \ref{prop-0304-1601}, $F(H(t',u'))$ admits a measure.
%	where $t',u'$ are
%	as in \eqref{eq-t'-0304-1557} and \eqref{eq-u'-0304-1558}.
%	\begin{equation}\label{cond-to-check-2403-2341}
%		F(H(t,u))\succeq 0
%		\quad\text{and}\quad
%		\Rank F(H(t,u))= \Rank F(H(t,u))|_{\{\vec{X'},Y\vec{X}^{(1)}\}}.
%	\end{equation} 
%	By Proposition \ref{prop-0304-1601}, \eqref{cond-to-check-2403-2341} holds iff $(t,u)=(t',u')$, where $t',u'$ are
%	as in \eqref{eq-t'-0304-1557} and \eqref{eq-u'-0304-1558}.
	Further on, Lemma \ref{existence-of-a-measure-2403-1104-v2} implies that the measure for $\beta$ exists only if also the sequence with 
	the moment matrix $U(t',u')$ has a $\RR$--representing measure.
	Note that $U(t',u')=A_\gamma$ with $A_{\gamma}$ as in \eqref{cond-2403-1008} of Theorem \ref{y3-2402}. 
	By Theorem \ref{Hamburger} used for $\gamma$, 
	$\gamma$ has a 
		$(\Rank U(t',u'))$--atomic $\RR$--representing measure
	iff \eqref{111121-1729} holds, in which case a $(\Rank(U(t',u')))$--atomic representing measure exists.
	Using %\eqref{rank-f(t0,gamma0)-0903-0945}, 
	\eqref{161121-1403}, \eqref{171121-2208} and \eqref{rank-0903-0940}, it follows that the measure for $\phi(\beta)$ is 
	$(\Rank \phi(M)(t'))$--atomic.
\end{proof}

Finally we prove Corolllary \ref{cor-0304-2008}.

\begin{proof}[Proof of Corollary \ref{cor-0304-2008}]
	First note that the moreover part follows immediately from the moreover part in Theorem \ref{y3-2402} by noticing that under the assumption 
	$(M_k)|_{\cC^{(2)}}\succ 0$ it holds that $\Rank N\leq \Rank M_k$.

	It remains to prove the equivalence $(\Leftrightarrow)$ in the corollary. The nontrivial implication is $(\Rightarrow)$.
	Following the proof of the implication $\eqref{pt1-0704-0915}\Leftarrow\eqref{pt2-0704-0915}$ of Theorem \ref{y3-2402} under the assumption that $M$ is positive definite, note that %$\Rank\phi(t_0)=3k$ and hence
	\begin{equation}\label{rank-0304-1954}
		\Rank F(H(\alpha(t_0),u_0))=
				\left\{ 
				\begin{array}{rl} 
					2k-1,&	\text{if }\Rank \widehat C=\Rank C,\\
					2k,&	\text{otherwise}.
				\end{array}\right.
	\end{equation}
	(We used the fact that $\Rank \widehat C\in \{\Rank C,\Rank C+1\}=\{2k-1,2k\}$).
	We have that
		$${(F(H(\alpha(t_0),\gamma_0)))}|_{\{\vec{X}^{(1)}\},\{Y\vec{X}^{(1)}\}}=
		\left(
		\begin{array}{cc}
			\widetilde{B}_{00} & \widehat{A}_{01}\\
			\widehat{A}_{01} &  \widetilde A_{11}
		\end{array}
		\right).$$

	\noindent\textbf{Claim.} $\Rank {(F(H(\alpha(t_0),u_0)))}|_{\{\vec{X}^{(1)}\},\{Y\vec{X}^{(1)}\}}\geq 2k-1$.\\

	\noindent \textit{Proof of Claim.} By \eqref{rank-0304-1954}, we only need to consider two cases:
	\begin{enumerate}
		\item $\Rank F(H(\alpha(t_0),u_0))=2k$: In this case the claim is obvious.
		\item $\Rank F(H(\alpha(t_0),u_0))=2k-1$: In this case the facts that $\Rank C=\Rank \widetilde C$ and that the moment matrix
	$F(H(\alpha(t_0),u_0))$ satisfies the column relations $\eqref{171121-2039}$ for $i=0,\ldots,k-2$, 
	imply that
		$\Rank {(F(H(\alpha(t_0),\gamma_0)))}|_{\{\vec{X}^{(1)}\},\{Y\vec{X}^{(1)}\}}=2k-1.$
		%$\Rank {\left((F(H(t_0,\gamma_0)))(\{\vec{X}^{(1)}\},\{Y\vec{X}^{(1)}\})\right)}=2k-1.$
	\end{enumerate}
	This proves the claim.\hfill $\blacksquare$\\
	
	Note that	
		$${(\phi(F(H(\alpha(t_0),\gamma_0))))}|_{\{\vec{X}^{(1)}\},\{Y\vec{X}^{(1)}\}}=
		%$${(\phi(F(H(t_0,\gamma_0))))(\{\vec{X}^{(1)}\},\{Y\vec{X}^{(1)}\})}=
		\left(
		\begin{array}{cc}
			\widetilde{B}_{00}& c(\widehat{A}_{01}-\widetilde\alpha_2\widetilde B_{00}) \\ \rule{0pt}{\normalbaselineskip} 
			c(\widehat{A}_{01}-\widetilde\alpha_2\widetilde B_{00}) & c(\widehat{A}_{01}-\widetilde\alpha_2\widetilde B_{00})
		\end{array}
		\right),$$
	where $c=(\alpha_3-\alpha_2)^{-1}$, 
	\begin{equation}\label{rank-eq-0304-1959}
		\Rank {(F(H(\alpha(t_0),\gamma_0)))}|_{\{\vec{X}^{(1)}\},\{Y\vec{X}^{(1)}\}}=\Rank {(\phi(F(H(t_0,\gamma_0))))}|_{\{\vec{X}^{(1)}\},\{Y\vec{X}^{(1)}\}}
	\end{equation}
	and
	\begin{equation}\label{col-space-0304-2000}
		\cC((\phi(F(H(\alpha(t_0),\gamma_0))))|_{\{\vec{X}^{(1)}\},\{Y\vec{X}^{(1)}\}})=
		\cC(\left(
		\begin{array}{cc}
			c(\widetilde\alpha_3\widetilde{B}_{00}-\widehat{A}_{01}) & c(\widehat{A}_{01}-\widetilde\alpha_2\widetilde B_{00}) \\ \rule{0pt}{\normalbaselineskip} 
			0 & c(\widehat{A}_{01}-\widetilde\alpha_2\widetilde B_{00})
		\end{array}
		\right)).
	\end{equation}
	Now, the Claim, \eqref{rank-eq-0304-1959} and \eqref{col-space-0304-2000} imply that at least one of 
		$\widetilde\alpha_3\widetilde{B}_{00}-\widehat{A}_{01}$ and $\widehat{A}_{01}-\widetilde\alpha_2\widetilde B_{00}$
	must be invertible and the statement of the corollary follows by Theorem \ref{y3-2402}.
\end{proof}

The following examples \ref{ex1-0504}--\ref{ex4-0504} demonstrate the use of Theorem \ref{y3-2402} and its proof to either construct a representing
measure supported on the union of three parallel lines for the sequence $\beta$ or show that a representing measure does not exist.
The \textit{Mathematica} file with numerical computations for the following examples can be found on the link \url{https://github.com/ZalarA/TMP_parallel_lines}.
In all examples we assume the notation from Theorem \ref{y3-2402} and its proof. Further on, $P$ will 
be the permutation matrix such that moment matrix $PM_3P^T$ has rows and columns indexed in the order $\vec{X}^{(0)}, Y\vec{X}^{(1)},Y^2\vec{X}^{(2)},Y^3$.

\begin{example}\label{ex1-0504}
Let 
	%\begin{align*}
	$\beta
	=\big(1, 
	\frac{3}{2}, 0,  
	\frac{7}{2}, 0, \frac{2}{3},  
	9,0, 1, 0,
	\frac{49}{2},0, \frac{7}{3}, 0,\frac{2}{3},%\\
	%&\hspace{1cm}
	69,0, \frac{191}{32}, 0,1, 0,
	 \frac{397}{2},0, \frac{49}{3}, 0,\frac{7}{3}, 0,\frac{2}{3}
	\big)$
	%\end{align*} 
	be a bivariate sequence of degree 6. 
	We will prove below that $\beta$ does not have a $\RR^2$--representing measure. % on the union of parallel lines $y=0$ and $y=1$.
	The moment matrix $PM_3P^T$ is equal to:
	\begin{equation*}
	PM_3P^T=
	\kbordermatrix{
		& 1&X&X^2 &X^3&Y&YX&YX^2&Y^2&Y^2X&Y^{3}\\[0.5em]
	1& 1 & \frac{3}{2} & \frac{7}{2} & 9 & 0 & 0 & 0 & \frac{2}{3} & 1 & 0 \\[0.5em]
	X& \frac{3}{2} & \frac{7}{2} & 9 & \frac{49}{2} & 0 & 0 & 0 & 1 & \frac{7}{3} & 0 \\[0.5em]
	X^2 & \frac{7}{2} & 9 & \frac{49}{2} & 69 & 0 & 0 & 0 & \frac{7}{3} & \frac{191}{32} & 0 \\[0.5em]
	X^3& 9 & \frac{49}{2} & 69 & \frac{397}{2} & 0 & 0 & 0 & \frac{191}{32} & \frac{49}{3} & 0 \\[0.5em]
	 Y&0 & 0 & 0 & 0 & \frac{2}{3} & 1 & \frac{7}{3} & 0 & 0 & \frac{2}{3} \\[0.5em]
	YX& 0 & 0 & 0 & 0 & 1 & \frac{7}{3} & \frac{191}{32} & 0 & 0 & 1 \\[0.5em]
	 YX^2&0 & 0 & 0 & 0 & \frac{7}{3} & \frac{191}{32} & \frac{49}{3} & 0 & 0 & \frac{7}{3} \\[0.5em]
	Y^2& \frac{2}{3} & 1 & \frac{7}{3} & \frac{191}{32} & 0 & 0 & 0 & \frac{2}{3} & 1 & 0 \\[0.5em]
	 Y^2X&1 & \frac{7}{3} & \frac{191}{32} & \frac{49}{3} & 0 & 0 & 0 & 1 & \frac{7}{3} & 0 \\[0.5em]
	Y^{3}& 0 & 0 & 0 & 0 & \frac{2}{3} & 1 & \frac{7}{3} & 0 & 0 & \frac{2}{3}
	}.
	\end{equation*}
	$PM_3P^T$ is psd with the eigenvalues
		$227.591$, $19.2501$, $1.91716$, $0.677211$, $0.648649$, $0.283553$, $0.0727021$, $0.0539334$, $0.00565968$, $0$
	and satisfies the column relation $Y^3=Y$. 
	%We assume the notation from Theorem \ref{y3-2402}.
	The transformation $\phi$ is the identity, i.e., $\phi(x,y)=(x,y)$.
	The matrix $N$ is equal to
	$$\kbordermatrix{
		& 1&X&X^2 &Y&YX&YX^2&Y^2&Y^2X&Y^{2}X^2\\[0.5em]
	1& 1 & \frac{3}{2} & \frac{7}{2} & 0 & 0 & 0 & \frac{2}{3} & 1 & \frac{7}{3}\\[0.5em]
	X& \frac{3}{2} & \frac{7}{2} & 9 &  0 & 0 & 0 & 1 & \frac{7}{3} & \frac{191}{32} \\[0.5em]
	X^2 & \frac{7}{2} & 9 & \frac{49}{2} &  0 & 0 & 0 & \frac{7}{3} & \frac{191}{32} & \frac{49}{3} \\[0.5em]
	 Y&0 & 0 & 0 &  \frac{2}{3} & 1 & \frac{7}{3} & 0 & 0 & 0 \\[0.5em]
	YX& 0 & 0 & 0 &  1 & \frac{7}{3} & \frac{191}{32} & 0 & 0 & 0 \\[0.5em]
	 YX^2&0 & 0 & 0 &  \frac{7}{3} & \frac{191}{32} & \frac{49}{3} & 0 & 0 & 0 \\[0.5em]
	Y^2& \frac{2}{3} & 1 & \frac{7}{3} &  0 & 0 & 0 & \frac{2}{3} & 1 &  \frac{7}{3} \\[0.5em]
	 Y^2X&1 & \frac{7}{3} & \frac{191}{32} &  0 & 0 & 0 & 1 & \frac{7}{3} & \frac{191}{32}  \\[0.5em]
	Y^{2}X^2& 0 & 0 & 0 & \frac{2}{3} & 1 & \frac{7}{3} & 0 & 0 & \frac{49}{3}
	}.$$
However, 
	the matrix $N$ is not psd, since the eigenvalues are
		$43.0994$, $18.8854$, $4.16403$, $0.863394$, $0.382338$, $0.132343$, $0.0775933$, $0.0656297$, $-0.00347317.$
	By Theorem \ref{y3-2402}, $\beta$ does not have a representing measure supported on $y^3=y$.
	Since any representing measure $\mu$ for $\beta$ must satisfy $\supp \;\mu\subseteq \{(x,y)\colon y^3=y\}$,  
	$\beta$ does not have any $\RR^2$--representing measure.
\end{example}

\begin{example}\label{ex2-0504}
Let 
	%\begin{align*}
	$\beta
	=\big(1, 
	\frac{15}{11}, 0,  
	3, 0, \frac{8}{11},  
	\frac{81}{11}, 0, \frac{12}{11}, 0,
	\frac{213}{11}, 0, \frac{28}{11}, 0, \frac{8}{11},
	\frac{585}{11}, 0, \frac{72}{11}, 0, \frac{12}{11}, 0,
	%&\hspace{1cm}
	\frac{107121}{715}, 0, \frac{196}{11}$, $0, \frac{28}{11}, 0, \frac{8}{11}
	\big)$
	%\end{align*} 
	be a bivariate sequence of degree 6.
	We will demonstrate below how Case 1 from the proof of Theorem \ref{y3-2402}
	can be applied to construct a 9--atomic representing measure for $\beta$ supported on the union of parallel lines $y=0$, $y=1$ and $y=-1$.
The moment matrix $PM_3P^T$ is equal to:
	\begin{equation*}
	PM_3P^T=
	\kbordermatrix{
		& 1&X&X^2 &X^3&Y&YX&YX^2&Y^2&Y^2X&Y^{3}\\[0.5em]
	 1 &1 & \frac{15}{11} & 3 & \frac{81}{11} & 0 & 0 & 0 & \frac{8}{11} & \frac{12}{11} & 0\\[0.5em]
	X&  \frac{15}{11} & 3 & \frac{81}{11} & \frac{213}{11} & 0 & 0 & 0 & \frac{12}{11} & \frac{28}{11} & 0  \\[0.5em]
	X^2 &  3 & \frac{81}{11} & \frac{213}{11} & \frac{585}{11} & 0 & 0 & 0 & \frac{28}{11} & \frac{72}{11} & 0\\[0.5em]
	X^3& \frac{81}{11} & \frac{213}{11} & \frac{585}{11} & \frac{107121}{715} & 0 & 0 & 0 & \frac{72}{11} & \frac{196}{11} & 0 \\[0.5em]
	Y& 0 & 0 & 0 & 0 & \frac{8}{11} & \frac{12}{11} & \frac{28}{11} & 0 & 0 & \frac{8}{11}\\[0.5em]
	YX& 0 & 0 & 0 & 0 & \frac{12}{11} & \frac{28}{11} & \frac{72}{11} & 0 & 0 & \frac{12}{11}  \ \\[0.5em]
	YX^2& 0 & 0 & 0 & 0 & \frac{28}{11} & \frac{72}{11} & \frac{196}{11} & 0 & 0 & \frac{28}{11}\\[0.5em]
	Y^2& \frac{8}{11} & \frac{12}{11} & \frac{28}{11} & \frac{72}{11} & 0 & 0 & 0 & \frac{8}{11} & \frac{12}{11} & 0\\[0.5em]
	Y^2X&\frac{12}{11} & \frac{28}{11} & \frac{72}{11} & \frac{196}{11} & 0 & 0 & 0 & \frac{12}{11} & \frac{28}{11} & 0 \\[0.5em]
	Y^{3}&0 & 0 & 0 & 0 & \frac{8}{11} & \frac{12}{11} & \frac{28}{11} & 0 & 0 & \frac{8}{11}
	}.
	\end{equation*}
	$PM_3P^T$ is psd with the eigenvalues
		$174.128$, $21.0215$, $1.64351$, $0.734399$, $0.379065$, $0.266664$, $0.0622926$, $0.0387818$, $0$, $0$
	and the column relations 
		$$Y^3=Y,\qquad 
		Y^2X=-\frac{10}{9}\cdot 1 +\frac{310}{27}\cdot X- \frac{95}{9}\cdot X^2 +\frac{65}{27}\cdot X^3 +\frac{2}{3}\cdot Y^2.$$
	%We assume the notation from Theorem \ref{y3-2402}.
	The transformation $\phi$ is the identity, i.e., $\phi(x,y)=(x,y)$.
	The matrix $N$ is positive definite with the eigenvalues 
	 	$42.3043$, $20.624$, $1.18246$, $0.780712$, $0.410054$,$ 0.126054$,$ 0.0568576$, $0.0557925$, $0.00526127.$
	The matrix $\phi(M)(\mathbf{t})$ is equal to 
	\begin{equation*}
	\kbordermatrix{
		& 1&X&X^2 &X^3&Y&YX&YX^2&Y^2&Y^2X&Y^{2}X^2\\[0.5em]
	 1 &1 & \frac{15}{11} & 3 & \frac{81}{11} & 0 & 0 & 0 & \frac{8}{11} & \frac{12}{11} & \frac{28}{11}\\[0.5em]
	X&  \frac{15}{11} & 3 & \frac{81}{11} & \frac{213}{11} & 0 & 0 & 0 & \frac{12}{11} & \frac{28}{11} & \frac{72}{11}  \\[0.5em]
	X^2 &  3 & \frac{81}{11} & \frac{213}{11} & \frac{585}{11} & 0 & 0 & 0 & \frac{28}{11} & \frac{72}{11} & \frac{196}{11}\\[0.5em]
	X^3& \frac{81}{11} & \frac{213}{11} & \frac{585}{11} & \frac{107121}{715} & 0 & 0 & 0 & \frac{72}{11} & \frac{196}{11} & \mathbf{t} \\[0.5em]
	Y& 0 & 0 & 0 & 0 & \frac{8}{11} & \frac{12}{11} & \frac{28}{11} & 0 & 0 & 0\\[0.5em]
	YX& 0 & 0 & 0 & 0 & \frac{12}{11} & \frac{28}{11} & \frac{72}{11} & 0 & 0 & 0  \ \\[0.5em]
	YX^2& 0 & 0 & 0 & 0 & \frac{28}{11} & \frac{72}{11} & \frac{196}{11} & 0 & 0 & 0\\[0.5em]
	Y^2& \frac{8}{11} & \frac{12}{11} & \frac{28}{11} & \frac{72}{11} & 0 & 0 & 0 & \frac{8}{11} & \frac{12}{11} & \frac{28}{11}\\[0.5em]
	Y^2X&\frac{12}{11} & \frac{28}{11} & \frac{72}{11} & \frac{196}{11} & 0 & 0 & 0 & \frac{12}{11} & \frac{28}{11} & \frac{72}{11} \\[0.5em]
	Y^{2}X^2&\frac{28}{11} & \frac{72}{11} & \frac{196}{11}& \mathbf{t} & 0 & 0 & 0 & \frac{28}{11} & \frac{72}{11} & \frac{196}{11}\\[0.5em]}.
	\end{equation*}
	The parameter $t_0$ from the proof of Theorem \ref{y3-2402} satisfies $\det \phi(M)(t_0)=0$, since $\Rank \phi(M)(t_0)=\max(\Rank M,\Rank N)\leq 9$. 
	A calculation shows that 
		$\det \phi(M)(\mathbf{t})=-\frac{204800}{337186519813}(143\mathbf{t}-7164)^2,$
	which means that $t_0=\frac{7164}{143}$. 
	By definition of $\alpha(t)$ we have $\alpha(t_0)=t_0=\frac{7164}{143}$ (here 
	$\widetilde\alpha_2=-1$ and $\widetilde \alpha_3=1$).
	The parameter $u_0$ from the proof of Theorem \ref{y3-2402} is equal to the right lower corner of the matrix
	$B(t_0)(\widehat C)^\dagger (B(t_0))^T$, where
	\begin{align*}
		B(t_0)
		&={(\phi(M)(t_0))}|_{\{\vec{X}^{(0)}\},\{Y\vec{X}^{(1)},Y^2\vec{X}^{(1)}\}}\qquad \text{and}\qquad
		\widehat C 
		={(\phi(M)(t_0))}|_{\{Y\vec{X}^{(1)},Y^2\vec{X}^{(1)}\}}.
	\end{align*}
	A calculation shows that $u_0=\frac{1331888}{9295}$, while 
		$$B(t_0)(\widehat C)^\dagger (B(t_0))^T=H(\alpha(t_0),u_0)=A_{\widehat{\beta}^{(1)}}\in S_4$$
	where
	$\widehat\beta^{(1)}=\left(\frac{8}{11}, \frac{12}{11},\frac{28}{11}, \frac{72}{11},\frac{196}{11} ,\frac{7164}{143},\frac{1331888}{9295}\right)\in \RR^{7}.$
%	\begin{equation*}
%	H(\alpha(t_0),u_0)=
%	\begin{blockarray}{ccccc}
%		& 1&X&X^2 &X^3\\[0.5em]
%		\begin{block}{c(cccc)}
%	 1 & \frac{8}{11} & \frac{12}{11} & \frac{28}{11} & \frac{72}{11}\\[0.5em]
%	X&  \frac{12}{11} & \frac{28}{11} & \frac{72}{11} & \frac{196}{11}  \\[0.5em]
%	X^2 &  \frac{28}{11} & \frac{72}{11} & \frac{196}{11} & \frac{7164}{143}\\[0.5em]
%	X^3& \frac{72}{11} & \frac{196}{11} & \frac{7164}{143} & \frac{1331888}{9295} \\[0.5em]
%	\end{block}
%	\end{blockarray}.
%	\end{equation*}
	Further on, 
		$$U(\alpha(t_0),u_0)=\phi(M)|_{\{\vec{X}^{(0)}\}}-H(\alpha(t_0),u_0)=A_{\widehat\beta^{(2)}}\in S_4$$
	where
	$\widehat\beta^{(2)}=\left(\frac{3}{11} , \frac{3}{11} , \frac{5}{11} , \frac{9}{11},\frac{17}{11},\frac{441}{143} , \frac{12137}{1859}\right)\in \RR^7.$
%	\begin{equation*}
%	\begin{blockarray}{ccccc}
%		& 1&X&X^2 &X^3\\[0.5em]
%		\begin{block}{c(cccc)}
%	 1 & \frac{3}{11} & \frac{3}{11} & \frac{5}{11} & \frac{9}{11}\\[0.5em]
%	X&  \frac{3}{11} & \frac{5}{11} & \frac{9}{11} & \frac{17}{11}  \\[0.5em]
%	X^2 &  \frac{5}{11} & \frac{9}{11} & \frac{17}{11} & \frac{441}{143}\\[0.5em]
%	X^3& \frac{9}{11} & \frac{17}{11} & \frac{441}{143} & \frac{12137}{1859} \\[0.5em]
%	\end{block}
%	\end{blockarray}.
%	\end{equation*}
	A calculation shows that $\delta=0$, where $\delta$ is equal to the value of the generalized Schur complement of the $3\times 3$ leading principal submatrix of $U(\alpha(t_0),u_0)$.
	Further on, 
	\begin{align*}
		\widetilde B_{00}-\widehat A_{01}
		&={(H(\alpha(t_0),u_0))}|_{\{\vec{X}^{(1)}\}}-(\phi(M))|_{\{\vec{X}^{(1)}\},\{Y\vec{X}^{(1)}\}}
		={(H(\alpha(t_0),u_0))}|_{\{\vec{X}^{(1)}\}}
	\end{align*} is invertible with the eigenvalues $20.6, 0.41, 0.056$ 
	and hence Case 1 from the proof of Theorem \ref{y3-2402} applies:
	\begin{itemize}
		\item The matrix $U(\frac{7164}{143},\frac{1331888}{9295})=A_{\widehat \beta^{(2)}}$ %is equal to $A_{\widehat \beta^{(1)}}\in S_4$, where
			%$\widehat \beta^{(1)}=\left(\frac{3}{11},\frac{3}{11},\frac{5}{11},\frac{9}{11},\frac{17}{11},\frac{441}{143},\frac{12137}{1859}\right)\in \RR^7$.
			%It 
%			has a column relation $X^3=\frac{57}{13}X^2-\frac{62}{13}X+\frac{6}{13}1.$
%			\begin{equation}\label{041121-1329}
%				X^3=\frac{57}{13}X^2-\frac{62}{13}X+\frac{6}{13}1.
%			\end{equation}
			%By Theorem \ref{Hamburger}, 
			%$\widehat \beta^{(1)}$ 
			has a column relation $X^3=\frac{57}{13}X^2-\frac{62}{13}X+\frac{6}{13}1$ and hence
			by Theorem \ref{Hamburger}, it has the unique $\RR$--representing measure
			$\mu_1=\rho_1\delta_{x_1}+\rho_2\delta_{x_2}+\rho_3\delta_{x_3},$
			where 
				$(x_1, x_2, x_3)\approx (0.107053, 1.62587, 2.65169)$
			are the solutions of the column relation,
				%\begin{align*}
				$\big(\rho_1^{(1)} \; \rho_2^{(1)} \; \rho_3^{(1)}\big)^T
					=
					V_{x}^{-1}
					\big(\frac{3}{11} \; \frac{3}{11} \; \frac{5}{11}\big)^T
					\approx 
					\big(0.1199 \; 0.1414 \; 0.01126\big)$
				%\end{align*}$
			and $x=(x_1,x_2,x_3)$.\\

		\item The moment matrix $F(H(\frac{7163}{143},\frac{1331888}{9295}))$ satisfies the relations
			$Y^3=Y$, $Y^2X=X$, $Y^2=1$ and $X^3=\frac{57}{13}\cdot X^2-\frac{283}{65}\cdot X+\frac{12}{65}\cdot 1.$
		After using an affine linear transformation $\phi(x,y)=(x,\frac{1}{2}y+\frac{1}{2}),$ 	
		it turns out that 
			$\phi(F(H(\frac{7163}{143},\frac{1331888}{9295})))$
		is precisely the matrix from Example \ref{ex-1103-824} above.
		Hence, 
			$F(H(\frac{7163}{143},\frac{1331888}{9295}))$ 
		is equal to the moment matrix generated by the measure  			
		\begin{align*}
			\mu_2
			&=\rho'_1\delta_{\phi^{-1}(x_1',1)}+\rho'_2\delta_{\phi^{-1}(x_2',1)}+\rho'_3\delta_{\phi^{-1}(x_3',1)}+
				\rho'_1\delta_{\phi^{-1}(x_1',0)}+\rho'_2\delta_{\phi^{-1}(x_2',0)}+\rho'_3\delta_{\phi^{-1}(x_3',0)}\\
			&=\rho'_1\delta_{(x_1',1)}+\rho'_2\delta_{(x_2',1)}+\rho'_3\delta_{(x_3',1)}+
				\rho'_1\delta_{(x_1',-1)}+\rho'_2\delta_{(x_2',-1)}+\rho'_3\delta_{(x_3',-1)},
		\end{align*}
		where $x_1'\approx 0.0445476, x_2'\approx 1.17328, x_3'\approx 3.53217$
		and $\rho'_1\approx 0.0541354$, $\rho'_2\approx 0.233231$, $\rho'_3\approx 0.0762695$.
	\end{itemize}
	We conclude that $\beta=\phi(\beta)$ has a 9--atomic representing measure $\mu_1+\mu_2$ supported on the union of the lines $y=-1,y=0$ and $y=1$.
\end{example}

\begin{example}\label{ex3-0504}
Let 
	%\begin{align*}
	$\beta
	=\big(1, 
	\frac{5}{7}, 0,  
	1, 0, \frac{4}{7},  
	 \frac{11}{7},0,  \frac{2}{7}, 0,
	\frac{19}{7},0, \frac{2}{7}, 0,\frac{4}{7},
	5,0, \frac{2}{7}, 0,\frac{2}{7}, 0,
	 \frac{67}{7},0, \frac{2}{7}, 0,\frac{2}{7}, 0,\frac{4}{7}
	\big)$
	%\end{align*} 
	be a bivariate sequence of degree 6.
	We will demonstrate below how Case 2 from the proof of Theorem \ref{y3-2402}
	can be applied to construct a 7--atomic representing measure for $\beta$ supported on the union of parallel lines $y=0$, $y=1$ and $y=-1$.
	The moment matrix $PM_3P^T$ is equal to:
	\begin{equation*}
	PM_3P^T=
	\kbordermatrix{
		& 1&X&X^2 &X^3&Y&YX&YX^2&Y^2&Y^2X&Y^{3}\\[0.5em]
	1& 1 & \frac{5}{7} & 1 & \frac{11}{7} & 0 & 0 & 0 & \frac{4}{7} & \frac{2}{7} & 0 \\[0.5em]
	X& \frac{5}{7} & 1 & \frac{11}{7} & \frac{19}{7} & 0 & 0 & 0 & \frac{2}{7} & \frac{2}{7} & 0 \\[0.5em]
	X^2 & 1 & \frac{11}{7} & \frac{19}{7} & 5 & 0 & 0 & 0 & \frac{2}{7} & \frac{2}{7} & 0 \\[0.5em]
	X^3& \frac{11}{7} & \frac{19}{7} & 5 & \frac{67}{7} & 0 & 0 & 0 & \frac{2}{7} & \frac{2}{7} & 0 \\[0.5em]
	Y&0 & 0 & 0 & 0 & \frac{4}{7} & \frac{2}{7} & \frac{2}{7} & 0 & 0 & \frac{4}{7}\\[0.5em]
	YX& 0 & 0 & 0 & 0 & \frac{2}{7} & \frac{2}{7} & \frac{2}{7} & 0 & 0 & \frac{2}{7}\\[0.5em]
	YX^2 & 0 & 0 & 0 & 0 & \frac{2}{7} & \frac{2}{7} & \frac{2}{7} & 0 & 0 & \frac{2}{7}\\[0.5em]
	Y^2& \frac{4}{7} & \frac{2}{7} & \frac{2}{7} & \frac{2}{7} & 0 & 0 & 0 & \frac{4}{7} & \frac{2}{7} & 0 \\[0.5em]
	 Y^2X& \frac{2}{7} & \frac{2}{7} & \frac{2}{7} & \frac{2}{7} & 0 & 0 & 0 & \frac{2}{7} & \frac{2}{7} & 0 \\[0.5em]
	Y^{3}& 0 & 0 & 0 & 0 & \frac{4}{7} & \frac{2}{7} & \frac{2}{7} & 0 & 0 & \frac{4}{7} \\[0.5em]
	}.
	\end{equation*}
	$PM_3P^T$ is psd with eigenvalues
		$13.3804$, $1.49602$, $1.36779$, $0.2337$, $0.218266$, $0.138494$, $0.0224999$, $0$, $0$, $0$
	and the column relations 
		$Y^3=Y$, $YX^2=YX$ and $X^3=3X^2-2X.$
	The transformation $\phi$ is the identity, i.e., $\phi(x,y)=(x,y)$.
	Moreover, $N$
	is psd with the eigenvalues 
		$4.35783$, $1.07956$, $0.97549$, $0.318458$, $0.167368$, $0.0866196$, $0.0146715$, $0$, $0$,
	a matrix $\widetilde B_{00}-\widehat A_{01}=A_{\widehat\beta^{(1)}}\in S_3$ where 
	$\widehat\beta^{(1)}=\left(\frac{4}{7} , \frac{2}{7} , \frac{2}{7} ,  \frac{2}{7} ,\frac{2}{7} \right)\in \RR^5$
%	\begin{equation*}
%	\widetilde B_{00}-\widehat A_{01}=
%	\left(\begin{array}{ccc}
%		\frac{4}{7} & \frac{2}{7} & \frac{2}{7}  \\[0.5em]
%		\frac{2}{7} & \frac{2}{7} & \frac{2}{7} \\[0.5em]
%		\frac{2}{7} & \frac{2}{7} & \frac{2}{7} \\[0.5em]
%	\end{array}\right)
%	\end{equation*}
	is singular with the eigenvalues 
		$0.97549, 0.167368, 0$
	and 
	$\widehat A_{01}=0$ is also singular.
	Hence, Case \eqref{cond-2403-1008} of Theorem \ref{y3-2402} applies.
	Computing $t'$ and $u'$ using formulas \eqref{eq-t'-0304-1557} and \eqref{eq-u'-0304-1558} (the vector $v$ is equal to 
	$(0 \; -1 )^T$), 
	we get $t'=u'=\frac{2}{7}$ and thus $A_{\gamma}$ as in \eqref{181121-2108}
	has 
	$\gamma=
	\left(
		\frac{3}{7} , \frac{3}{7} , \frac{5}{7} , \frac{9}{7} , \frac{17}{7} , \frac{33}{7} , \frac{65}{7}
	\right)\in \RR^7.$
	The matrix 
	$A_\gamma(3)$ is invertible with eigenvalues $3.35337, 0.200729, 0.017325$.
	It has a column relation $X^3=3X^2-2X$ which is satisfed by $x_1=0, x_2=1, x_3=2$.
			By Theorem \ref{Hamburger}, 
			$\gamma$ has the unique $\RR$--representing measure
			$\mu_1=\frac{1}{7}\delta_{0}+\frac{1}{7}\delta_{1}+\frac{1}{7}\delta_{2}$,
			where we obtained the densities by
				%\begin{align*}
				$(\frac{1}{7} \; \frac{1}{7} \; \frac{1}{7})^T
					=
					V_{(0,1,2)}^{-1}
					(\frac{3}{7} \; \frac{3}{7} \; \frac{5}{7})^T.$
				%\end{align*}
	It remains to compute the measure for $F(H(t',u'))$, which has the column relation $Y^2=1$.
	After using an affine linear transformation $\phi(x,y)=(x,\frac{1}{2}(y+1))$,  
	the matrix $\phi(F(H(t',u'))$ has the column relation $Y^2=Y$. 
	Using the proof of Theorem \ref{y2=1} to construct the measure for $\phi(F(H(t',u'))$, the matrices 
	$F$ and $E$ are both equal to 
	$A_{\widehat \beta^{(2)}}\in S_4$, where
	$\widehat \beta^{(2)}
			=\left(\frac{2}{7},\frac{1}{7},\frac{1}{7},\frac{1}{7},\frac{1}{7},\frac{1}{7},\frac{1}{7} \right)\in \RR^7.
	$
			$A_{\widehat \beta^{(2)}}$ is a psd Hankel matrix satisfying the column relation $X^2-X=0$
	and hence by Theorem \ref{Hamburger}, 
			$\widehat \beta^{(2)}$ has the unique $\RR$--representing measure
			$\frac{1}{7}\delta_{0}+\frac{1}{7}\delta_{1}$,
			where we obtained the densities by
				$(\frac{1}{7} \; \frac{1}{7})^T
					=
					V_{(0,1)}^{-1}
					(\frac{2}{7} \; \frac{1}{7} )^T$.
	Hence, $M_3$ admits a $(\Rank M_3)$-atomic measure on the variety $y^3=y$
	with the atoms
		$(0,0), (1,0), (2,0), (0,-1), (1,-1), (0,1), (1,1),$
	all with densities $\frac{1}{7}$.
\end{example}

\begin{example}\label{ex4-0504}
Let 
	%\begin{align*}
	$\beta
	=\big(1, 
	0,\frac{1}{7},
	1, 0, \frac{3}{7},  
	\frac{1}{14} \left(-1-2 \sqrt{23}\right),0,0 ,\frac{1}{7},
	2,0,\frac{2}{7},0, \frac{3}{7},
	\frac{1}{16} \left(4-9 \sqrt{23}\right),0, \frac{2}{7}$, $0,\frac{2}{7},\frac{1}{7},%\\
	%&\hspace{2cm}
	5,0,\frac{2}{7},0,\frac{2}{7},0,\frac{3}{7}
	\big)$
	%\end{align*} 
	be a bivariate sequence of degree 6.
	We will demonstrate below how Case \eqref{cond-2403-1008} from the proof of Theorem \ref{y3-2402}
	can be applied to prove that $\beta$ does not have a representing measure supported on $\RR^2$. % on the union of parallel lines $y=0$ and $y=1$.
	The moment matrix $PM_3P^T$ is equal to:
	\begin{equation*}
	PM_3P^T=
	\kbordermatrix{
		& 1&X&X^2 &X^3&Y&YX&YX^2&Y^2&Y^2X&Y^{3}\\[0.5em]
	1& 1 & 0 & 1 & \beta_{3,0} & \frac{1}{7} & 0 & 0 & \frac{3}{7} &
	   \frac{2}{7} & \frac{1}{7} \\[0.5em]
	X&   0 & 1 & \beta_{3,0} & 2 & 0 & 0 & 0 & \frac{2}{7} & \frac{2}{7} & 0 \\[0.5em]
	X^2 &  1 & \beta_{3,0} & 2 & \beta_{5,0}& 0 &
	   0 & 0 & \frac{2}{7} & \frac{2}{7} & 0\\[0.5em]
	X^3&  \beta_{3,0} & 2 & \beta_{5,0} & 5 & 0 &
	   0 & 0 & \frac{2}{7} & \frac{2}{7} & 0\\[0.5em]
	Y&  \frac{1}{7} & 0 & 0 & 0 & \frac{3}{7} & \frac{2}{7} & \frac{2}{7} & \frac{1}{7} & 0 &
   		\frac{3}{7} \\[0.5em]
	YX& 0 & 0 & 0 & 0 & \frac{2}{7} & \frac{2}{7} & \frac{2}{7} & 0 & 0 & \frac{2}{7} \\[0.5em]
	YX^2& 0 & 0 & 0 & 0 & \frac{2}{7} & \frac{2}{7} & \frac{2}{7} & 0 & 0 & \frac{2}{7} \\[0.5em]
	Y^2&    \frac{3}{7} & \frac{2}{7} & \frac{2}{7} & \frac{2}{7} & \frac{1}{7} & 0 & 0 & \frac{3}{7} &
	   \frac{2}{7} & \frac{1}{7}  \\[0.5em]
	Y^2X&   \frac{2}{7} & \frac{2}{7} & \frac{2}{7} & \frac{2}{7} & 0 & 0 & 0 & \frac{2}{7} & \frac{2}{7} &
   		0  \\[0.5em]
	Y^{3}& \frac{1}{7} & 0 & 0 & 0 & \frac{3}{7} & \frac{2}{7} & \frac{2}{7} & \frac{1}{7} & 0 &
   		\frac{3}{7}}.
	\end{equation*}
	where $\beta_{3,0}=\frac{1}{14} \left(-1-2 \sqrt{23}\right)$ and $\beta_{5,0}=\frac{1}{16} \left(4-9 \sqrt{23}\right)$.
	$PM_3P^T$
	is psd with the eigenvalues 
		$7.2968$, $2.0162$, $1.28926$, $0.286198$, $0.16608$, $0.0883151$, $0$, $0$, $0$, $0$,
	and the column relations 
	\begin{align*}
		Y^3 &=Y,\qquad YX^2=YX,\\
		Y^2X
			&=-0.42\cdot 1 +0.77\cdot X +0.65\cdot X^2+0.42\cdot Y- 0.42\cdot YX,\\
		Y^2	
			&=-0.42\cdot 1 +0.77\cdot X +0.65\cdot X^2 +1.42\cdot Y- 1.42\cdot YX.
	\end{align*}
	The transformation $\phi$ is the identity, i.e., $\phi(x,y)=(x,y)$. The matrix $N$ is psd with the eigenvalues
	$2.99455$, $1.6278$, $0.915861$,$ 0.304399$, $0.157391$, $0$, $0$, $0$, $0,$
	a matrix $\widetilde B_{00}-\widehat A_{01}=A_{\widehat\beta^{(1)}}\in S_3$ where $\widehat\beta^{(1)}={(\frac{2}{7},\frac{2}{7},\frac{2}{7},\frac{2}{7},\frac{2}{7})}\in\RR^5$ is psd and singular
	with the eigenvalues $\frac{2}{7}, 0, 0$ and $\widehat A_{01}=A_{\widehat\beta^{(2)}}\in S_3$ where $(\frac{2}{7},0,0,0,0)\in\RR^5$ is also psd and singular. 
	Hence, Case \eqref{cond-2403-1008} of Theorem \ref{y3-2402} applies.
	Computing $t'$ and $u'$ using formulas \eqref{eq-t'-0304-1557} and \eqref{eq-u'-0304-1558} (the vector $v$ is equal to 
	$(-1 \; 0 )^T$), 
	we get $t'=u'=\frac{2}{7}$ and thus $A_{\gamma}$ as in \eqref{181121-2108}
	has 
	$\gamma=\left(\frac{4}{7} , -\frac{2}{7} ,\frac{5}{7} ,\widetilde\beta_{3,0},
			\frac{12}{7},\widetilde\beta_{5,0},\frac{33}{7} \right)\in \RR^7,$
%	is of the form
%	\begin{equation*}
%	A_\gamma=
%	\begin{blockarray}{ccccc}
%		& 1 & X & X^2 & X^3\\[0.5em]
%		\begin{block}{c(cccc)}
%		 1& \frac{4}{7} & -\frac{2}{7} & \frac{5}{7} &\widetilde\beta_{3,0} \\[0.5em]
%		X& -\frac{2}{7} & \frac{5}{7} & \widetilde\beta_{3,0} & \frac{12}{7} \\[0.5em]
%		 X^2& \frac{5}{7} & \widetilde\beta_{3,0}& \frac{12}{7} & \widetilde\beta_{5,0} \\[0.5em]
%		 X^3& \widetilde\beta_{3,0} & \frac{12}{7} & \widetilde\beta_{5,0} & \frac{33}{7} \\[0.5em]
%	\end{block}
%	\end{blockarray},
%	\end{equation*}
	where 
		$\widetilde\beta_{3,0}=\frac{1}{14} \left(-1-2 \sqrt{23}\right)-\frac{2}{7}$
	and
		$\widetilde\beta_{5,0}=\frac{1}{16} \left(4-9\sqrt{23}\right)-\frac{2}{7}$.
	We have that  
		$\Rank(A_\gamma)=3>\Rank (A_{\gamma}(3))=2.$
	By Theorem \ref{y3-2402}, $\beta$ does not have a representing measure supported on $y^3=y$.
	Since any representing measure $\mu$ for $\beta$ must satisfy $\supp \;\mu\subseteq \{(x,y)\colon y^3=y\}$,  
	$\beta$ does not have any representing measure supported on $\RR^2$.
\end{example}

\section{The TMP on the union of parallel lines in the pure case}
\label{S5}

Let $k,n\in \NN$, $k\geq n\geq 2$ and  
		$\beta:=\beta^{(2k)}=(\beta_{i,j})_{i,j\in \ZZ_+,i+j\leq 2k}$ 
	be a real bivariate sequence of degree $2k$
	such that $\beta_{0,0}>0$ and let $M_k=M_k(\beta)$ be its associated moment matrix.
	To establish the existence of a representing measure for $\beta$ supported
	on the union of $n$ parallel lines, we can assume, after applying the appropriate affine linear transformation, 
	that the variety is
		\begin{equation*}\label{variety-2804-2029}
			K_{n,\underline{\alpha}}:=\Big\{(x,y)\in \RR^2\colon y \cdot \prod_{i=1}^{n-1}(y-\alpha_i)=0\Big\},
		\end{equation*}
	where $\underline{\alpha}=(\alpha_1,\ldots,\alpha_{n-1})\in \RR^{n-1}$ and $\alpha_i\in \RR\setminus\{0\}$ are pairwise distinct nonzero real numbers. %with $\alpha_1<\ldots<\alpha_{n-1}$.
	For $\ell=0,\ldots,n-1$ we denote by 
		$$\displaystyle c_\ell=\sum_{1\leq j_1<j_2\cdots <j_\ell\leq n-1} \alpha_{j_1}\cdots \alpha_{j_\ell}$$
	the sum of all products of $\ell$ pairwise distinct numbers from the set $\{\alpha_1,\ldots,\alpha_{n-1}\}$.
	If a $K_{n,\underline\alpha}$--representing measure for $\beta$ exists, then $M_k$ must satisfy the column relations 
	\begin{equation}\label{col-rel-3003-2124}
		Y^nX^j=c_1\cdot Y^{n-1}X^j-c_2\cdot Y^{n-2}X^j+\cdots+(-1)^{n} c_{n-1}\cdot YX^j
	\end{equation}
	for $j=0,\ldots,k-n$.
	On the level of moments the relations \eqref{col-rel-3003-2124} mean that 
	\begin{equation}\label{241121-2104}
		\beta_{i+j,n+\ell}=c_1\cdot \beta_{i+j,n-1+\ell}-c_2\cdot \beta_{i+j,n-2+\ell}+\cdots+(-1)^{n} c_{n-1}\cdot \beta_{i+j,1+\ell}
	\end{equation}
	for every $i,j,\ell\in \NN\cup\{0\}$ such that $i+j+n+\ell\leq 2k$.
	We write 
		$$\vec{X}^{(i)}:=(1,X,\ldots,X^{k-i})\qquad \text{and}\qquad Y^j\vec{X}^{(i)}:=(Y^j,Y^jX,\ldots,Y^jX^{k-i})$$
	for $i=0,\ldots,n-1$ and $j\in \NN$.
%	For $i=0,\ldots,n-1$ we write 
%		$$\vec{X}^{(i)}:=(1,X,\ldots,X^{k-i}).$$
%	We use the notation $Y^j\vec{X}^{(i)}$ to denote
%		$$Y^j\vec{X}^{(i)}:=(Y^j,Y^jX,\ldots,Y^jX^{k-i}).$$
	In the presence of column relations \eqref{col-rel-3003-2124}, the column space $\cC(M_k)$ is spanned by the columns in the set 
	$\displaystyle \cC^{(n)}:=\{\vec{X}^{(0)},Y\vec{X}^{(1)},\ldots,Y^{n-1}\vec{X}^{(n-1)}\}.$
	Let $P$ be a permutation matrix such that moment matrix $PM_kP^T$ has rows and columns indexed in the 	order $\vec{X}^{(0)}, Y\vec{X}^{(1)},Y^2\vec{X}^{(2)},\ldots,Y^k$.
	Let
	\begin{equation}\label{notation-Mk-3003-2129}
		S^{(\beta)}_{k,n}:=(PM_k(\beta)P^T)|_{\cC^{(n)}}%=\left(A_{ij}\right)_{i,j}^{n}
		=
		\left(\begin{array}{cc} A_{00} & B \\ B^T & C\end{array}\right)\in \RR^{s(k,n)\times s(k,n)},
	\end{equation}
	where 
		$s(k,n)=\sum_{i=0}^{n-1} (k+1-i)=\frac{1}{2}n(2k+3-n),$
	be the restriction of the moment matrix $PM_kP^T$ to the rows and columns in the set $\cC^{(n)}$
	and
	\begin{equation}\label{091221-1302}
	A_{00}=(PM_kP^T)|_{\{\vec{X}^{(0)}\}},\quad 
			B=(PM_kP^T)|_{\{\vec{X}^{(0)}\},\cC^{(n)}\setminus \{\vec{X}^{(0)}\}},\quad
			 C=(PM_kP^T)|_{\cC^{(n)}\setminus \{\vec{X}^{(0)}\}}.
	\end{equation}
	If a $K_{n,\underline\alpha}$--representing measure for $\beta$ exists, then it generates some extension of $\beta$ with moments of higher 		
	degrees. In particular, $S^{(\beta)}_{k,n}$ can be extended to the moment matrix by adding the rows and columns indexed by
	\begin{equation}\label{parameters-2604-1947}
		X^iY^j,\quad \text{where}\quad 2\leq j\leq n-1 \quad\text{and}\quad k+1-j\leq i\leq k-1.
	\end{equation}
	The moments corresponding to monomials 
	\begin{equation}\label{parameters-2604-2013}
		x^iy^j,\quad \text{where}\quad 2\leq j\leq n-1 \quad\text{and}\quad 2k+1-j\leq i\leq 2k-1,
	\end{equation}
	are parameters which we denote by $\mbf{t_{i,j}}$, 
	while other moments in the extension can be expressed from the original moments and
 	parameters $\mbf{t_{i,j}}$ using the relations \eqref{241121-2104} which are satisfied in the extension.
%	\begin{equation}\label{21112021-1904}
%		\beta_{\ell_1,n+\ell_2}=
%			c_1\cdot \beta_{\ell_1,n+\ell_2-1}-
%			c_2\cdot \beta_{\ell_1,n+\ell_2-2}+\cdots+
%			(-1)^{n} c_{n-1}\cdot \beta_{\ell_1,\ell_2+1},
%	\end{equation}
%	where $\ell_1,\ell_2\in \NN\cup\{0\}$.
	We write 
		$$\vec{\mbf{t}}_{(k,n)}:=(\mbf{t_{2k-1,2}},\mbf{t_{2k-2,3}},\mbf{t_{2k-1,3}},\ldots,
			\mbf{t_{2k+2-n,n-1}},\ldots,\mbf{t_{2k-1,n-1}}).$$
	Hence, the moment matrix extending $S^{(\beta)}_{k,n}$ by the rows and columns indexed by monomials in \eqref{parameters-2604-1947}
 	in the order 
		$$\underbrace{\vec{X}^{(0)}, Y\vec{X}^{(1)},Y^2\vec{X}^{(2)}\ldots,Y^k}_{\text{old rows/columns}},
		\underbrace{Y^2X^{k-1},Y^3X^{k-2},Y^3X^{k-1},\ldots,Y^{n-1}X^2,\ldots,Y^{n-1}X^{k-1}}_{\text{rows/columns added}}$$
	is a linear matrix function in the parameters $\mbf{t_{i,j}}$, which we denote
	by 
	\begin{align}\label{251121-0824}
	\begin{split}
		S^{(\beta)}_{k,n}(\vec{\mbf{t}}_{(k,n)})
			&=\left(\begin{array}{cc} 
					S^{(\beta)}_{k,n} & B(\vec{\mbf{t}}_{(k,n)})  \\ \rule{0pt}{\normalbaselineskip} 
					(B(\vec{\mbf{t}}_{(k,n)}))^T & C(\vec{\mbf{t}}_{(k,n)})
				\end{array}\right)\\
			&=\left(\begin{array}{ccc} 
					A_{00} & B & B_1(\vec{\mbf{t}}_{(k,n)}) \\ \rule{0pt}{\normalbaselineskip} 
					B^T & C & B_2(\vec{\mbf{t}}_{(k,n)})  \\ \rule{0pt}{\normalbaselineskip}  
					(B_1(\vec{\mbf{t}}_{(k,n)}))^T & (B_2(\vec{\mbf{t}}_{(k,n)}))^T &C(\vec{\mbf{t}}_{(k,n)})
				\end{array}\right)
			\in \RR^{t(k,n)\times t(k,n)},
	\end{split}
	\end{align}
	where $t(k,n)=s(k,n)+\binom{n-1}{2}$. 

	The following theorem gives a sufficient condition for the existence of a $K_{n,\underline\alpha}$--representing measure for $\beta$ in terms of the feasibility of the 
	linear matrix inequality
	$S^{(\beta)}_{k,n}(\vec{\mbf{t}}_{(k,n)})\succ 0$.

	\begin{theorem}\label{cor-0104-1556}
	Let $k,n\in \NN$, $k\geq n\geq 2$ and $K_{n,\underline\alpha}:=\big\{(x,y)\in \RR^2\colon y \cdot \prod_{i=1}^{n-1}(y-\alpha_i)=0\big\}$ be a union of $n$ parallel lines, where	
	$\underline{\alpha}=(\alpha_1,\ldots,\alpha_{n-1})\in \RR^{n-1}$ and $\alpha_i\in \RR\setminus\{0\}$ are pairwise distinct nonzero real numbers. %with $\alpha_1<\ldots<\alpha_{n-1}$,
	Let $\beta:=\beta^{(2k)}=(\beta_{i,j})_{i,j\in \ZZ_+,i+j\leq 2k}$ be a real bivariate sequence of degree $2k$ such that $M_k$ is positive semidefinite and satisfies  
	the column relations \eqref{col-rel-3003-2124} for $j=0,\ldots,k-n$.
	Assume also the notation above. 
	If there exists $\vec t_{(k,n)}\in \RR^{\binom{n-1}{2}}$ such that
		$S^{(\beta)}_{k,n}(\vec{t}_{(k,n)})\succ 0,$ 
	then  $\beta$ has a $K_{n,\underline\alpha}$--representing measure.
	\end{theorem}

	\begin{remark}
	\begin{enumerate}
	\item\label{pt1-101221-0441}
	Mimicking the idea with which we solved the TMP--$n$pl for $n=2$ and $n=3$, also in the general case $n\in \NN$
	we will apply the ALT such that one of the lines becomes $y=0$ and then study the existence of the decompositions 
	$\beta=\widetilde \beta+\widehat \beta$ such that $\widetilde\beta$, $\widehat \beta$ have representing measures 
	supported on $y=0$ and on the other $n-1$ parallel lines, respectively. 
	In the TMP--2pl only the moment $\widetilde \beta_{2k,0}$ was a parameter, 
	while in the TMP--3pl only the moments $\widetilde \beta_{2k-1,0}$ and $\widetilde \beta_{2k,0}$ were parameters.
	Both cases were small enough so that we could characterize precisely when the representing measure for $\beta$ exists.
	In the general case, the moments $\widetilde \beta_{2k-n+1,0},\ldots, \widetilde \beta_{2k,0}$ are parameters and we are not able to handle all cases 
	to characterize precisely when the representing measure for $\beta$ exists. However, at least in the special pure case, which could be also called \textit{purely pure}, because
	not only the parallel lines determine the only column relations of $M_k$ but also $M_k$ can be extended to the matrix of highest possible rank when we add the rows and columns 
	indexed by monomials in \eqref{col-rel-3003-2124}, we can prove that the representing measure exists. 
	\item
	The proof of Theorem \ref{cor-0104-1556} will be by induction on the number $n$ of parallel lines. But first we will establish some 
	lemmas which will be used in the proof.
	One of them is the analog of Lemma \ref{existence-of-a-measure-2403-1104-v2} for general $n\in \NN$. 
	This lemma can be possibly used in future to handle also not purely pure cases of $\beta$. However, a very demanding analysis will be needed to control the $(n-1)$--dimensional
	parameter space $\widetilde \beta_{2k-n+1,0},\ldots, \widetilde \beta_{2k,0}$ described in \eqref{pt1-101221-0441} above.
\end{enumerate}
\end{remark}
 
	We define a matrix function 
		$$F^{(\beta)}_{k,n}:
			\RR^{(k+1)\times (k+1)}\to \RR^{s(k,n)\times s(k,n)},\qquad
			F^{(\beta)}_{k,n}(\mathbf{Z})=\left(\begin{array}{cc}
			\mathbf{Z} 			& B 		\\ 
			B^T 		& C 		\\
		\end{array}\right),$$
	where $B$ and $C$ are defined by \eqref{091221-1302}.
	It holds that
	\begin{equation*}%\label{notation-Mk-0103-1803}
		S^{(\beta)}_{k,n}=F^{(\beta)}_{k,n}(\mathbf{Z})+\left((A_{00}-\mathbf{Z})\oplus \mbf{0}_{s(k,n)-k-1}\right),
	\end{equation*}
	where $\mbf{0}_{s(k,n)-k-1}$ stands for a $(s(k,n)-k-1)\times (s(k,n)-k-1)$ zero matrix.
	If $\beta$ admits a $K_{n,\underline\alpha}$--representing measure $\mu$, then it is supported on the union of parallel lines $y=0$, 
	$y=\alpha_1$, $\ldots$, $y=\alpha_{n-1}$.
	Since the moment matrix generated by the measure supported on $y=0$ can be nonzero only when
	restricted to the columns and rows indexed by $\vec{X}^{(0)}$, it folllows that the restriction of the moment matrix 
	generated by ${\mu|}_{\cup_{i=1}^{n-1}\{y=\alpha_i\}}$ (resp.\ ${\mu|}_{\{y=0\}}$) to the columns and rows from $\cC^{(n)}$
	is of the form $F^{(\beta)}_{k,n}(B_{00})$ (resp.\ $(A_{00}-B_{00})\oplus \mbf{0}_{s(k,n)-k-1}$), where $B_{00}\in S_{k+1}$ is a Hankel matrix.
	This discussion establishes the implication $(\Rightarrow)$ of the following lemma.

	%\begin{color}{blue}We are here.\end{color}

	\begin{lemma}	\label{existence-of-a-measure-0104-0941}
		$\beta$ has a $K_{n,\underline\alpha}$--representing measure if and only if there exist a Hankel matrix 
		$B_{00}\in S^{k+1}$,
		such that:
		\begin{enumerate}
			\item  The sequence with the moment matrix $F^{(\beta)}_{k,n}(B_{00})$ has a $\widetilde K_{n,\underline\alpha}$--representing measure, where
				\begin{equation}\label{variety-3003-2251}
					\widetilde K_{n,\underline\alpha}:=\Big\{(x,y)\in \RR^2\colon \prod_{i=1}^{n-1}(y-\alpha_i)=0\Big\}.
				\end{equation}
			\item The sequence with the moment matrix $A_{00}-B_{00}$ has a $\RR$--representing measure.
		\end{enumerate}
	\end{lemma}

	\begin{proof}
	The implication $(\Rightarrow)$ follows by the discussion in the paragraph before the
	lemma. It remains to establish the implication $(\Leftarrow)$.
	Let $P$ be a permutation matrix such that moment matrix $PM_kP^T$ has rows and columns indexed in the order $\vec{X}^{(0)}, Y\vec{X}^{(1)},Y^2\vec{X}^{(2)},\ldots,Y^k$.
	Let $M_k^{(1)}$ (resp.\ $M_k^{(2)}$) be the moment matrix generated by the measure $\mu_1$ (resp.\ $\mu_2$) supported on $\widetilde K_{n,\underline\alpha}$ (resp.\ $y=0$)
	such that ${(PM_k^{(1)}P^T)|}_{\mathcal C^{(n)}}=F^{(\beta)}_{k,n}(B_{00})$ 
	(resp.\ ${(PM_k^{(2)}P^T)|}_{\mathcal C^{(n)}}=(A_{00}-B_{00})\bigoplus \mathbf{0}_{s(k,n)-k-1}$).
	Since $\supp(\mu_i)\subset K_{n,\underline{\alpha}}$, $i=1,2$, the moment matrix
	$\widetilde M_k=M_k^{(1)}+M_k^{(2)}$ corresponding to the measure $\mu_1+\mu_2$ has column relations 
	\eqref{col-rel-3003-2124}.
	Since ${(PM_kP^T)|}_{\mathcal C^{(n)}}=S^{(\beta)}_{k,n}={(P\widetilde{M}_kP^T)|}_{\mathcal C^{(n)}}$ and
	both $M_k$, $\widetilde{M}_k$ satisfy the column relations \eqref{col-rel-3003-2124}, 
	%$X^iY^2=X^iY$ for $i=0,\ldots,k-2$, 
	it follows that 
		$$M_k=\left(\begin{array}{cc} S^{(\beta)}_{k,n} & S^{(\beta)}_{k,n}W \\ \rule{0pt}{\normalbaselineskip}  W^T S^{(\beta)}_{k,n}  & W^TS^{(\beta)}_{k,n}W\end{array}\right)=\widetilde{M}_k$$ 
	for some matrix $W\in \RR^{s(k,n)\times \frac{(k-n+1)(k-n+2)}{2}}$. This concludes the proof of the implication $(\Leftarrow)$.
	\end{proof}
	
	For $i=0,\ldots,2k-n+1$ we define the numbers $\gamma_i\in \RR$ by 
	\begin{equation}\label{23112021-1331}
		\gamma_i=(-1)^{n} c_{n-1}^{-1}\left(\beta_{i,n-1}-
		c_1\cdot \beta_{i,n-2}+\cdots+\ldots+(-1)^{n-1}c_{n-2} \cdot \beta_{i,1}\right).
	\end{equation}
	We define the matrix function  
	$$H^{(\beta)}_{k,n}:\RR^{n-1}\to S_{k+1},\qquad 
		H^{(\beta)}_{k,n}(\mathbf{u}_1,\ldots,\mathbf{u}_{n-1})=
		A_{(\gamma_0,\ldots,\gamma_{2k-n+1},\mathbf{u}_1,\ldots,\mathbf{u}_{n-1})},$$
	where $A_{(\gamma_0,\ldots,\gamma_{2k-n+1},\mathbf{u}_1,\ldots,\mathbf{u}_{n-1})}$ is defined as in
	\eqref{vector-v}.
	The following lemma describes the form of the matrix $B_{00}$ from 
	Lemma \ref{existence-of-a-measure-0104-0941}.

	\begin{lemma}\label{231121-1215}
	Assume there is a Hankel matrix $B_{00}$ such that the sequence with the moment matrix $F^{(\beta)}_{k,n}(B_{00})$  admits a 
	$\widetilde K_{n,\underline\alpha}$--representing measure $\mu_1$ where $\widetilde K_{n,\underline\alpha}$ is as in \eqref{variety-3003-2251}.
	Then 
	\begin{equation}\label{231121-1200}
		B_{00}=H^{(\beta)}_{k,n}(\widetilde\beta_{2k-n+2,0}(\mu_1),\ldots,\widetilde\beta_{2k,0}(\mu_1)),
	 \end{equation}
	where $\widetilde\beta_{2k-i,0}(\mu_1)$ 
	are the moments of the monomials $x^{2k-i}$ with respect to $\mu_1$.	
\end{lemma}

\begin{proof}
	Let $M_{k+1}^{(1)}$ be the moment matrix generated by the measure $\mu_1$. 
	Since $\supp(\mu_1)\subseteq \widetilde K_{n,\underline\alpha}$,
	the moments $\widetilde\beta_{i,j}(\mu_1)$ in $M_{k+1}^{(1)}$ satisfy the relations
	\begin{equation}\label{23112021-1239}
		\widetilde\beta_{i,j+n-1}(\mu_1)=
		c_1\cdot \widetilde\beta_{i,j+n-2}(\mu_1)-
		c_2\cdot \widetilde\beta_{i,j+n-3}(\mu_1)
		+\cdots+
		(-1)^nc_{n-1}\cdot \widetilde\beta_{i,j}(\mu_1),
	\end{equation}
	for $i,j\in \NN\cup\{0\}$ such that $0\leq i+j+n-1\leq 2k+2$.
	Equivalently, \eqref{23112021-1239} can be expressed in the form
	\begin{equation*}\label{23112021-1242}
		 \widetilde\beta_{i,j}(\mu_1)=(-1)^{n} c_{n-1}^{-1}\left(\widetilde\beta_{i,j+n-1}(\mu_1)-
		c_1\cdot\widetilde\beta_{i,j+n-2}(\mu_1)+\cdots+\ldots+(-1)^{n-1}c_{n-2} \cdot\widetilde\beta_{i,j+1}(\mu_1)\right).
	\end{equation*}
	Since $\widetilde \beta_{i,j}=\beta_{i,j}$ for every $i,j\in \NN\cup \{0\}$ with $j>0$ and $i+j\leq 2k$,
	we have that
	\begin{equation}\label{23112021-1246}
		\widetilde \beta_{i,0}(\mu_1)=(-1)^{n} c_{n-1}^{-1}\left(\beta_{i,n-1}-
		c_1\cdot \beta_{i,n-2}+\cdots+\ldots+(-1)^{n-1}c_{n-2} \cdot \beta_{i,1}\right).
	\end{equation}
	for every $i\leq 2k-n+1$.
	The equalities \eqref{23112021-1246} imply 
  	the equality \eqref{231121-1200}.
\end{proof}

Now we are ready to prove Theorem \ref{cor-0104-1556}.

\begin{proof}[Proof of Theorem \ref{cor-0104-1556}]
	Let us fix $k\in \NN$, $k\geq 2$. The proof will be by induction on the number of lines $n$. \\

\noindent\textit{Base of induction -- Theorem \ref{cor-0104-1556} holds for $n=2$:}\\

			Note that $S^{(\beta)}_{k,2}(\vec{\mbf{t}}_{(k,2)})=S^{(\beta)}_{k,2}$. So the assumption $S^{(\beta)}_{k,2}(\vec{t}_{(k,2)})\succ 0$ is equivalent to $S^{(\beta)}_{k,2}\succ 0$.
			Hence, in the notation of Section \ref{S3}, $M=S^{(\beta)}_{k,2}\succ 0$ and $N=(S^{(\beta)}_{k,2})|_{\{\vec{X}^{(1)},Y\vec{X}^{(1)}\}}\succ 0$.
			Using Corollary \ref{cor-0304-1917}, $\beta$ has a $K_{2,\underline\alpha}$--representing measure.\\

\noindent\textit{Induction step -- If Theorem \ref{cor-0104-1556} holds in the case of $n-1$ lines where $2\leq n-1<k$, then it also holds in the case of $n$ lines:}\\
 
	We assume that Theorem \ref{cor-0104-1556} holds in the case of $n-1$ lines where $2\leq n-1<k$.
	By Lemmas \ref{existence-of-a-measure-0104-0941} and \ref{231121-1215}, proving Theorem \ref{cor-0104-1556} in the case of $n$ lines is equivalent to 
	showing that there exist $u_1,\ldots,$$u_{n-1}\in \RR$ such that the sequence with a matrix $F^{(\beta)}_{k,n}(H^{(\beta)}_{k,n}(u_1,\ldots,u_{n-1}))$ 
	admits a $\widetilde K_{n,\underline\alpha}$--representing measure and
	the sequence with a matrix $A_{00}-H^{(\beta)}_{k,n}(u_1,\ldots,u_{n-1})$ admits a $\RR$--representing measure.
By the equivalence between \eqref{pt1-281021-2128} and \eqref{pt2-281021-2128}
	of Theorem \ref{block-psd} used for $\cM=S^{(\beta)}_{k,n}(\vec{t}_{(k,n)})$, 
	we have that 	
	\begin{equation}\label{231121-2024}
		S^{(\beta)}_{k,n}(\vec{t}_{(k,n)})\Big/
			\left(\begin{array}{cc} C & B_2(\vec{t}_{(k,n)}) \\ (B_2(\vec{t}_{(k,n)}))^T & C(\vec{t}_{(k,n)})\end{array}\right)
			=A_{00}-H\succ 0,
	\end{equation}
	where 
		$$H:=
		\left(\begin{array}{cc} B & B_1(\vec{t}_{(k,n)})\end{array}\right)\cdot
		\left(\begin{array}{cc} C & B_2(\vec{t}_{(k,n)}) \\ (B_2(\vec{t}_{(k,n)}))^T & C(\vec{t}_{(k,n)})\end{array}\right)^\dagger \cdot
		\left(\begin{array}{c} B^T \\ (B_1(\vec{t}_{(k,n)}))^T\end{array}\right).
		$$
	\noindent \textbf{Claim 1.}
	$H=H^{(\beta)}_{k,n}(u_1,\ldots,u_{n-1})$ for some $u_i\in\RR$, $i=1,\ldots,n-1$.\\

	\noindent \textit{Proof of Claim 1.} 
	Let $Q$ be a permutation matrix such that moment matrix $QS^{(\beta)}_{k,n}(\vec{t}_{(k,n)})Q^T$ 
	has rows and columns indexed in the 	order $\vec X^{(0)}, Y\vec X^{(1)},Y^2\vec X^{(1)}\ldots,Y^{n-1}\vec X^{(1)}.$
	Then $QS^{(\beta)}_{k,n}(\vec{t}_{(k,n)})Q^T$ is equal to
		$$
		\kbordermatrix{
		& \vec X^{(0)}&\cdots & Y^j\vec X^{(1)}& \cdots & Y^{n-1}\vec X^{(1)}\\[0.5em]
		 (\vec X^{(0)})^T& A_{00}& \cdots & A_{0j}(\vec{t}_{(k,n)}) & \cdots & A_{0,n-1}(\vec{t}_{(k,n)})\\
		\vdots & \vdots & \ddots & \vdots & \ddots & \vdots\\
		 (Y^i\vec X^{(1)})^T& A_{i0}(\vec{t}_{(k,n)})& \cdots  & A_{ij}(\vec{t}_{(k,n)}) & \cdots & A_{i,n-1}(\vec{t}_{(k,n)})\\
		\vdots & \vdots & \ddots & \vdots & \ddots & \vdots\\
		 (Y^{n-1}\vec X^{(1)})^T& A_{n-1,0}(\vec{t}_{(k,n)})& \cdots  & A_{n-1,j}(\vec{t}_{(k,n)}) & \cdots & A_{n-1,n-1}(\vec{t}_{(k,n)})},$$
	where $A_{ij}(\vec{t}_{(k,n)})$ are Hankel matrices.
	By construction of $H$ we have that
	\begin{align}\label{091221-1006}
		H|_{\{\vec{X}^{0}\},\{\vec{X}^{(1)}\}}
			&= \frac{(-1)^n}{c_{n-1}}\left(A_{0,n-1}(\vec{t}_{(k,n)})
			-c_1 A_{0,n-2}(\vec{t}_{(k,n)})+\ldots+(-1)^{n-1}c_{n-2} A_{01}(\vec{t}_{(k,n)})\right).
	\end{align}
	Hence, $H|_{\{\vec{X}^{0}\},\{\vec{X}^{(1)}\}}$ is a Hankel matrix. Since $H$ is also a symmetric matrix, it follows that $H$ is Hankel and 
	thus equal to $A_{\gamma^{(1)}}\in S_{k+1}$ for some $\gamma^{(1)}=(\gamma_0^{(1)},\ldots,\gamma_{2k}^{(1)})\in\RR^{2k+1}$.
	It remains to prove that $\gamma_i^{(1)}=\gamma_i$ for $i=0,\ldots,2k+1-n$, where $\gamma_i$ are defined by \eqref{23112021-1331}.
	By \eqref{091221-1006} we have that for $i<k$:
	\begin{align*}
		\gamma_i^{(1)}
			&= \frac{(-1)^n}{c_{n-1}}\left((A_{0,n-1}(\vec{t}_{(k,n)}))_{0,i}
			-c_1 (A_{0,n-2}(\vec{t}_{(k,n)}))_{0,i}+\ldots+(-1)^{n-1}c_{n-2} (A_{01}(\vec{t}_{(k,n)}))_{0,i}\right)\\
			&=\frac{(-1)^n}{c_{n-1}}\left(\beta_{i,n-1}
			-c_1 \beta_{i,n-2}+\ldots+(-1)^{n-1}c_{n-2} \beta_{i,1}\right)=\gamma_i,
	\end{align*}
	while for $k\leq i\leq 2k+1-n$:
	\begin{align*}
		\gamma_i^{(1)}
			&= \frac{(-1)^n}{c_{n-1}}\big((A_{0,n-1}(\vec{t}_{(k,n)}))_{i-k+1,k-1}
			-c_1 (A_{0,n-2}(\vec{t}_{(k,n)}))_{i-k+1,k-1}+\ldots\\
			&\hspace{2cm}+(-1)^{n-1}c_{n-2} (A_{01}(\vec{t}_{(k,n)}))_{i-k+1,k-1}\big)\\
			&=\frac{(-1)^n}{c_{n-1}}\left(\beta_{i,n-1}
			-c_1 \beta_{i,n-2}+\ldots+(-1)^{n-1}c_{n-2} \beta_{i,1}\right)=\gamma_i.
	\end{align*}
	This proves the claim.\hfill $\blacksquare$\\

	Using \eqref{231121-2024} and Claim 1, we have that
		$T:=A_{00}-H^{(\beta)}_{k,n}(u_1,\ldots,u_{n-1}+\delta)\succ 0$
	for $0<\delta$ small enough. Since $T$ is Hankel, it is equal to $A_{\gamma^{(2)}}$ for some $\gamma^{(2)}\in \RR^{2k+1}.$ 
	By Theorem \ref{Hamburger} used for $\gamma^{(2)}$, a $(\Rank T)$--atomic $\RR$--representing measure for $\gamma^{(2)}$ exists.
	To conclude the induction step, it remains to prove that the sequence with the moment matrix
	$$F^{(\beta)}_{k,n}(H^{(\beta)}_{k,n}(u_1,\ldots,u_{n-1}+\delta))=
		\left(\begin{array}{cc}
			H^{(\beta)}_{k,n}(u_1,\ldots,u_{n-1}+\delta)			& B 		\\ 
			B^T 		& C 		\\
		\end{array}\right)$$
	admits a $\widetilde K_{n,\underline\alpha}$--representing measure. 
	Let $P$ be the permutation matrix defined before Theorem \ref{cor-0104-1556}.
	Note that $F^{(\beta)}_{k,n}(H^{(\beta)}_{k,n}(u_1,\ldots,u_{n-1}+\delta))$ is equal to $(PM_{k}(\widetilde\beta)P^T)|_{\mathcal C^{(n)}}$
	where $\widetilde\beta$ is a bivariate sequence of degree $2k$ that differs from $\beta$ only in the numbers $\widetilde\beta_{i,0}$ where $0\leq i\leq 2k$.
	Moreover, since $M_{k}(\beta)$ satisfies the column relations \eqref{col-rel-3003-2124} for $j=0,\ldots,k-n$ where the numbers $\beta_{i,0}$, $0\leq i\leq 2k$,
	do not occur, 
	it follows that $M_{k}(\widetilde\beta)$ also satisfies these column relations. 
	For $\widetilde\beta$ to have a $\widetilde K_{n,\underline\alpha}$--representing measure we also have to prove that $M_{k}(\widetilde\beta)$ satisfies the column relations
	\begin{equation}\label{231121-1922}
		Y^{n-1}X^j=c_1\cdot Y^{n-2}X^j-c_2\cdot Y^{n-3}X^j+\cdots+(-1)^{n} c_{n-1}\cdot X^j
	\end{equation}
	for $j=0,\ldots,k-1$.\\
	
	\noindent \textbf{Claim 2.}
	$M_{k}(\widetilde\beta)$ satisfies the column relations \eqref{231121-1922} for $j=0,\ldots,k-1$
	and 
	\begin{equation}\label{241121-2252}
		\Rank M_{k}(\widetilde\beta)=(n-1)k+1.
	\end{equation}

	\noindent \textit{Proof of Claim 2.} 
	%Since $M_k(\widetilde\beta)$ is psd, it is enough to prove that $(PM_k(\widetilde\beta)P^T)|_{\cC^{n}}$ satisfies the column relations \eqref{231121-1922} for $j=0,\ldots,k-1$.  
	Let 
		$$\cM^{(1)}=\left(\begin{array}{ccc} 
					H & B & B_1(\vec{t}_{(k,n)}) \\  
					B^T & C & B_2(\vec{t}_{(k,n)}) \\  
					(B_1(\vec{t}_{(k,n)}))^T & (B_2(\vec{t}_{(k,n)}))^T &C(\vec{t}_{(k,n)})
				\end{array}\right)
					=\left(\begin{array}{ccc} 
					H & \widehat{B}(\vec{t}_{(k,n)}) \\  
					(\widehat{B}(\vec{t}_{(k,n)}))^T & \widehat{C}(\vec{t}_{(k,n)}) 
				\end{array}\right).$$
	By the equivalence between \eqref{pt1-281021-2128} and \eqref{pt2-281021-2128} of Theorem \ref{block-psd}
	used for $\cM=\cM^{(1)}$,
	we have that $\cM^{(1)}\succeq 0$.
	Further on,
	\begin{equation}\label{23112021-1915}
		\Rank \cM^{(1)}=\Rank \widehat{C}(\vec{t}_{(k,n)})=(n-1)k,
	\end{equation}
	where we used \eqref{prop-2604-1140-eq2} of Proposition \ref{prop-2604-1140} for $\cM=\cM^{(1)}$ in the first equality
	and invertibility of $S^{(\beta)}_{k,n}(\vec{t}_{(k,n)})$ in the second equality.
	Since $(PM_k(\widetilde\beta)P^T)|_{\cC^{n}}\succeq \cM^{(1)}$ and $(PM_k(\widetilde\beta)P^T)|_{\cC^{n}}$ differs from $\cM^{(1)}$ only in the entry of 
	the row and column $X^{k}$, we conclude that:
	\begin{itemize}
		\item The equality \eqref{241121-2252} follows from \eqref{23112021-1915} also using the fact that the column $X^k$ is linearly dependent from the other columns
			in $\cM^{(1)}$, while in $M_k(\widetilde\beta)$ this is not true.  
		\item To show that $M_k(\widetilde\beta)$ satisfies the column relations \eqref{231121-1922} for $j=0,\ldots,k-1$, 	
			it suffices to prove that $\cM^{(1)}$ satisfies the column relations \eqref{231121-1922} for $j=0,\ldots,k-1$.
	\end{itemize}
	Since $\Rank \cM^{(1)}=\Rank \widehat{C}(\vec{t}_{(k,n)})$, there is a matrix $W\in \RR^{(n-1)k\times (k+1)}$ such that
		$$\cM^{(1)}=\left(\begin{array}{cc}
				W^T\widehat{C}(\vec{t}_{(k,n)})W & W^T\widehat{C}(\vec{t}_{(k,n)})\\
				\widehat{C}(\vec{t}_{(k,n)})W & \widehat{C}(\vec{t}_{(k,n)})
			\end{array}\right).$$
	Morever, if $W'$ is any matrix satisfying $\widehat{C}(\vec{t}_{(k,n)})W'=(\widehat{B}(\vec{t}_{(k,n)}))^T$, then 
	\begin{equation}\label{231121-2012}
		(W')^T\widehat{C}(\vec{t}_{(k,n)}) W'=(W')^T\widehat{C}(\vec{t}_{(k,n)}) W=W^T\widehat{C}(\vec{t}_{(k,n)}) W,
	\end{equation}
	where we used $\widehat{C}(\vec{t}_{(k,n)})W'=\widehat{C}(\vec{t}_{(k,n)})W$ in the first equality and $(W')^T\widehat{C}(\vec{t}_{(k,n)})=W^T \widehat{C}(\vec{t}_{(k,n)})$ in the second.
	Relations \eqref{241121-2104} and the definition of the extension of $\beta$ 
	imply that the restriction 
		$$(\cM^{(1)})|_{\cup_{i=1}^{n-1}\{Y^i\vec{X}^{(1)}\},\cup_{i=0}^{n-1}\{Y^i\vec{X}^{(1)}\}}=
		\left(\begin{array}{cc}
				(\widehat{B}(\vec{t}_{(k,n)}))^T & \widehat{C}(\vec{t}_{(k,n)})
			\end{array}\right)$$
	satisfies the relations \eqref{231121-1922} for $j=0,\ldots,k-1$. 
	Using also \eqref{231121-2012}, it follows that
	$\cM^{(1)}$ satisfies the relations \eqref{231121-1922} for $j=0,\ldots,k-1$.
	\hfill $\blacksquare$\\

	Let $\phi(x,y)=(x,y-\alpha_1)$. By Proposition \ref{251021-2254}, $\widetilde\beta$ admits a $\widetilde K_{n,\underline\alpha}$--representing measure iff
 	$\phi(\widetilde\beta)$ admits a $\phi(\widetilde K_{n,\underline\alpha})$--representing measure, where 
		$$\phi(\widetilde K_{n,\underline\alpha})=\big\{(x,y)\in \RR^2\colon y \cdot \prod_{i=2}^{n-1}(y-(\underbrace{\alpha_i-\alpha_1}_{\widetilde\alpha_i}))=0\big\}.$$
	We see that $\phi(\widetilde K_{n,\underline\alpha})=K_{n-1,\underline{\widetilde\alpha}}$ where $\underline{\widetilde\alpha}=(\widetilde \alpha_2,\ldots,\widetilde \alpha_{n-1})\in \RR^{n-2}$.
	By Claim 2, $\cC(M_k(\phi(\widetilde\beta)))$ is spanned by the columns from the set $\{\vec{X}^{(0)},Y\vec X^{(1)},\ldots,Y^{n-2}\vec X^{(1)}\}$
	and this columns are linearly independent.
	Note that
		$S^{(\phi(\widetilde\beta))}_{k,n-1}(\vec{t}_{k,n-1})=(PM_k(\phi(\widetilde\beta)P^T)|_{\{\vec{X}^{(0)},Y\vec X^{(1)},\ldots,Y^{n-2}\vec X^{(1)}\}}\succ 0,$
	where 
		$\vec{t}_{k,n}=(\vec{t}_{k,n-1},t_{2k+2-n,n-1},\ldots,t_{2k-1,n-1}).$
	By the induction hypothesis used for $\phi(\widetilde \beta)$ and the set $K_{n-1,\widetilde\alpha}$, $\phi(\widetilde \beta)$ admits a representing measure on $K_{n-1,\widetilde\alpha}$,
	which concludes the proof of the induction step. \\

	This concludes the proof of the theorem.
\end{proof}

Note that in the case of two lines we have that $S^{(\beta)}_{k,2}(\vec{\mbf{t}}_{(k,2)})=S^{(\beta)}_{k,2}$ and hence
Theorem \ref{cor-0104-1556} is precisely the pure case solved by Corollary \ref{cor-0304-1917} (which was also used to prove the base
of induction in Theorem \ref{cor-0104-1556}).  In the case of three lines $S_{k,3}(\vec{\mbf{t}}_{(k,3)})$ is equal to $\phi(M)(\mbf{t})$
from the proof of Theorem \ref{y3-2402}. Using Corollary \ref{cor-0304-2008} under the assumption $S_{k,3}\succ 0$,
$\beta$ admits a representing measure iff also 
${\phi(M)(\mbf{t})|}_{\mathcal C^{(n)}\cup \{Y^2X^{k-1}\}\setminus \{X^k\}}\succeq 0$ which is independent of $\mbf{t}$.
The smallest $n$ where solving the linear matrix inequality $S_{(k,n)}(\vec{\mbf{t}}_{(k,n)})\succ 0$ in Theorem \ref{cor-0104-1556} is necessarily $n=4$, 
which we illustrate in the following example.

\begin{example}
	Let $n=4$, $k\geq n$, $\underline \alpha=(\alpha_1,\alpha_2,\alpha_3)\in \RR^3$ and $\beta$ be a bivariate sequence of degree $2k$.
	If the measure for $\beta$ supported on 
		$K_{4,\underline \alpha}=\big\{(x,y)\in \RR^2\colon y (y-\alpha_1)(y-\alpha_2)(y-\alpha_3)=0\big\}$
	exists, then it generates some extension of $\beta$, where the moments
	\begin{align*}
		\beta_{2k-3,4}
			&=(\alpha_1+\alpha_2+\alpha_3)\cdot\beta_{2k-3,3}-(\alpha_1\alpha_2+\alpha_1\alpha_3+\alpha_2\alpha_3)\cdot\beta_{2k-3,2}+\alpha_1\alpha_2\alpha_3 \cdot \beta_{2k-3.1},\\
		\beta_{2k-4,5}	
			&=(\alpha_1+\alpha_2+\alpha_3)\cdot\beta_{2k-4,4}-(\alpha_1\alpha_2+\alpha_1\alpha_3+\alpha_2\alpha_3)\cdot\beta_{2k-4,3}+\alpha_1\alpha_2\alpha_3 \cdot \beta_{2k-4,2},\\
		\beta_{2k-5,6}	
			&=(\alpha_1+\alpha_2+\alpha_3)\cdot\beta_{2k-5,5}-(\alpha_1\alpha_2+\alpha_1\alpha_3+\alpha_2\alpha_3)\cdot\beta_{2k-5,4}+\alpha_1\alpha_2\alpha_3 \cdot \beta_{2k-5,3},\\
		\beta_{2k-4,6}	
			&=(\alpha_1+\alpha_2+\alpha_3)\cdot\beta_{2k-4,5}-(\alpha_1\alpha_2+\alpha_1\alpha_3+\alpha_2\alpha_3)\cdot\beta_{2k-4,4}+\alpha_1\alpha_2\alpha_3 \cdot \beta_{2k-4,3},\\
		\beta_{2k-3,5}	
			&=(\alpha_1+\alpha_2+\alpha_3)\cdot\beta_{2k-3,4}-(\alpha_1\alpha_2+\alpha_1\alpha_3+\alpha_2\alpha_3)\cdot\beta_{2k-3,3}+\alpha_1\alpha_2\alpha_3 \cdot \beta_{2k-3,2},\\
		\beta_{2k-3,6}	
			&=(\alpha_1+\alpha_2+\alpha_3)\cdot\beta_{2k-3,5}-(\alpha_1\alpha_2+\alpha_1\alpha_3+\alpha_2\alpha_3)\cdot\beta_{2k-3,4}+\alpha_1\alpha_2\alpha_3 \cdot \beta_{2k-3,3},
	\end{align*}
	are already determined by $\beta$, while the moments
	\begin{align*}
		\beta_{2k-2,4}(\mathbf{t_1})
			&=(\alpha_1+\alpha_2+\alpha_3)\cdot \mathbf{t_1}-(\alpha_1\alpha_2+\alpha_1\alpha_3+\alpha_2\alpha_3)\cdot\beta_{2k-2,2}+\alpha_1\alpha_2\alpha_3 \cdot \beta_{2k-2,1},\\
		\beta_{2k-2,5}(\mathbf{t_1})	
			&=(\alpha_1+\alpha_2+\alpha_3)\cdot\beta_{2k-2,4}(\mathbf{t_1})-(\alpha_1\alpha_2+\alpha_1\alpha_3+\alpha_2\alpha_3)\cdot \mathbf{t_1}+\alpha_1\alpha_2\alpha_3 \cdot \beta_{2k-2,2},\\
		\beta_{2k-2,6}(\mathbf{t_1})	
			&=(\alpha_1+\alpha_2+\alpha_3)\cdot\beta_{2k-2,5}(\mathbf{t_1})-(\alpha_1\alpha_2+\alpha_1\alpha_3+\alpha_2\alpha_3)\cdot\beta_{2k-2,4}(\mathbf{t_1})+\alpha_1\alpha_2\alpha_3 \cdot\mathbf{t_1},
	\end{align*}
	also depend on $\mathbf{t_1}$, which is the moment of $x^{2k-2}y^{3}$. 
	We denote by $\mathbf{t_2}$, $\mathbf{t_3}$ the moments correponding to $x^{2k-1}y^{2}$,  $x^{2k-1}y^{3}$, respectively.
	The moment matrix $S_{k,4}(\mathbf{t_1},\mathbf{t_2},\mathbf{t_3})$ extending $S_{k,4}$ with the additional columns and rows  $Y^2 X^{k-1}, Y^3X^{k-2}, Y^3X^{k-1}$
	is equal to
	\begin{tiny}
	\begin{equation*}\label{notation-Mk-3003-2252}
		\begin{blockarray}{cccccccc}
		& \vec{X} & Y\vec{X}^{(1)} & Y^2 \vec{X}^{(2)} & Y^3 \vec{X}^{(3)}&Y^2X^{k-1} & Y^3X^{k-2}& Y^3X^{k-1}\\
		\begin{block}{c(c|cccccc)}
			\vec{X}^T & A_{00} & A_{01} & A_{02} & A_{14} & \begin{array}{c} d_1 \\ \mathbf{t_2} \end{array} & \begin{array}{c} d_2 \\ \mathbf{t_1}  \end{array} & \begin{array}{c} d_3 \\ \mathbf{t_1} \\ \mathbf{t_3}  \end{array}\\ [0.5em]
			\cline{2-8}
			(Y\vec{X}^{(1)})^T & A_{01}^T & A_{11} &  A_{12}  & A_{24} & \begin{array}{c} c_1 \\ \mathbf{t_1} \end{array} &c_2 & \begin{array}{c} c_3 \\ \beta_{2k-2,4}(\mathbf{t_1}) \end{array} \\[0.5em]
			\cline{2-8}\\[-0.8em]
			(Y^2\vec{X}^{(2)})^T& A_{02}^T & A_{12}^T & A_{22} &  A_{34} & b_1  & b_2 & b_3\\[0.5em]
			(Y^3\vec{X}^{(3)})^T& A_{14}^T & A_{24}^T & A_{34}^T& A_{44} & a_1  & a_2 & a_3\\[0.5em]
			Y^2 X^{k-1} & \begin{array}{cc} d_1^T & \mathbf{t_2} \end{array} & \begin{array}{cc}c_1^T & \mathbf{t_1}\end{array} & b_1^T & a_1^T &\beta_{2k-2,4}(\mathbf{t_1})  & \beta_{2k-3,5} & \beta_{2k-2,5}(\mathbf{t_1}) \\[0.5em]
			Y^3X^{k-2}& \begin{array}{cc} d_2^T & \mathbf{t_1}\end{array} & c_2^T & b_2^T & a_2^T & \beta_{2k-3,5} & \beta_{2k-4,6} & \beta_{2k-3,6}\\[0.5em]
			Y^3X^{k-1}& \begin{array}{ccc} d_3^T & \mathbf{t_1} & \mathbf{t_3}\end{array} & \begin{array}{cc}c_3 ^T & \beta_{2k-2,4}(\mathbf{t_1})\end{array} & b_3^T & a_3^T & \beta_{2k-2,5}(\mathbf{t_1}) & \beta_{2k-3,6} & \beta_{2k-2,6}(\mathbf{t_1})\\[0.5em]
		\end{block}
		\end{blockarray},	
%	\left(\begin{array}{ccc} 
%			A_{00} & B  & d(t)\\ 
%			B^T & C & e\\
%		     d(t)^T & e^T & \beta_{2k-4,6}\end{array}\right),
	\end{equation*}
	\end{tiny}
	where
	\begin{align*}
		a_1
		&=\left(\begin{array}{cccc} \beta_{k-1,5}& \cdots & \beta_{2k-4,5}\end{array}\right)^T,\qquad
		b_1
		=\left(\begin{array}{cccc} \beta_{k-1,4}& \cdots & \beta_{2k-3,4}\end{array}\right)^T,\\	
		c_1
		&=\left(\begin{array}{ccc} \beta_{k-1,3}& \cdots & \beta_{2k-3,3} \end{array}\right)^T,\qquad
		d_1
		=\left(\begin{array}{ccc} \beta_{k-1,2}& \cdots & \beta_{2k-2,2} \end{array}\right)^T,\\
		a_2
		&=\left(\begin{array}{cccc} \beta_{k-2,6}& \cdots & \beta_{2k-5,6}\end{array}\right)^T,\qquad
		b_2
		=\left(\begin{array}{cccc} \beta_{k-2,5}& \cdots & \beta_{2k-4,5}\end{array}\right)^T,	\\	
		c_2
		&=\left(\begin{array}{ccc} \beta_{k-2,4}& \cdots & \beta_{2k-3,4}\end{array}\right)^T,\qquad
		d_2
		=\left(\begin{array}{ccc} \beta_{k-2,3}& \cdots & \beta_{2k-3,3} \end{array}\right)^T,\\
		a_3
		&=\left(\begin{array}{cccc} \beta_{k-1,6}& \cdots & \beta_{2k-4,6}\end{array}\right)^T,\qquad
		b_3
		=\left(\begin{array}{cccc} \beta_{k-1,5}& \cdots & \beta_{2k-3,5}\end{array}\right)^T,	\\	
		c_3
		&=\left(\begin{array}{ccc} \beta_{k-1,4}& \cdots & \beta_{2k-3,4} \end{array}\right)^T,\qquad
		d_3
		=\left(\begin{array}{ccc} \beta_{k-1,3}& \cdots & \beta_{2k-3,3} \end{array}\right)^T
	\end{align*}
	are vectors. 
	By Theorem \ref{cor-0104-1556}, a $K_{4,\underline\alpha}$--representing measure for $\beta$ exists if there are $t_1,t_2,t_3\in \RR$ such that the matrix $S_{k,4}(t_1,t_2,t_3)$
	is positive definite.
\end{example}

Under the assumption $S_{k,n}\succ 0$, the sufficient condition $S_{k,n}(\vec{\mbf{t}}_{(k,n)})\succ 0$ from Theorem \ref{cor-0104-1556}
for the existence of a $K_{n,\underline{\alpha}}$--representing measure is not always necessary already for $n=3$ by the following example (for $n=2$ it is necessary due to a trivial reason 
$S_{k,2}(\vec{\mbf{t}}_{(k,2)})=S_{k,2}$).
For the \textit{Mathematica} file with the numerical computations see \url{https://github.com/ZalarA/TMP_parallel_lines}.

\begin{example}\label{counter-0305-0808}
	Let $k=n=3$.
	The intersection of the varieties
	\begin{align*}
		0&=(y-1)(y-2)(y-3)\qquad\text{and}\qquad
		0=y^2x^2+x(x+1)(x+2),
	\end{align*}
	are 9 real points, which can be checked using \textit{Mathematica}:
	\begin{align*}
		p_1&=(0,1),\; p_2=(0,2),\; p_3=(0,3),\; p_4\approx (-3.41,1),\; p_5\approx (-0.59,1),\\
		p_6&\approx (-11.83,3),\; p_7\approx (-0.17,3),\; p_8\approx (-6.70,2),\; p_{9}\approx (-0.298,2),
	\end{align*}
	We generate the moment matrix $M_3$ with the atoms $p_i$, $i=1,\ldots,9$, %and $p_{10+j}=(j,0)$, where $j=0,\ldots,4$, 
	all with densities $\frac{1}{9}$.\\

	\noindent \textbf{Claim 1.} The submatrix 
		$S_{3,3}={(PM_3P^T)|}_{\mathcal C^{(2)}}$ 
	cannot be extended with the row and column $Y^2X^2$ to a positive definite matrix.\\ 
	
	\noindent \textit{Proof of Claim 1.} 
	Since the atoms $p_i$, $i=1,\ldots,9$, lie on the union of lines $y=1$, $y=2$ and $y=3$, the moment matrix $M_3$ has a column
	relation $(Y-1)(Y-2)(Y-3)=\mbf 0$. The extension $S_{3,3}(\vec{\textbf{t}}_{3,3})$ defined by \eqref{251121-0824} 
	has only one unknown entry $\textbf{\textbf{t}}_{5,2}$ in the column $Y^2X^2$ and row $X^3$ which we denote by $\textbf{t}$
	for short. Note that $S_{3,3}(\vec{\textbf{t}}_{3,3})=S_{3,3}(\textbf{t})$ has the same form as $(\phi(M))(\textbf{t})$ for $k=3$
	from the proof of Theorem \ref{y3-2402}. 
	The restriciton ${(S_{3,3}(\textbf{t}))|}_{\mathcal C^{(2)}\cup\{Y^2X^2\}\setminus \{X^3\}}$ 
	is independent of $\mbf{t}$ and since the atoms in the measure lie on the variety $y^2x^2+x(x+1)(x+2)=0$, 
	it has the column relation 
		$Y^3X^2+YX(X+1)(X+2)=\mbf{0}.$
	By Lemma \ref{extension-principle} used for the psd extension $\cM=S_{3,3}(t)$ of $S_{3,3}$, $S_{3,3}(t)$ must also satisfy this column relation. %\eqref{rel-0305-0822}.
	This proves the claim.\hfill $\blacksquare$\\

	\noindent \textbf{Claim 2.} $S_{3,3}$ is positive definite.\\ 
	
	\noindent \textit{Proof of Claim 2.} 
%	Let us assume that $S_{4,4}$ is not positive definite. Then it satisfies a column relation $P(X,Y)=\mbf{0}$, where $P\in \RR[x,y]_{\leq 4}$
%	is a nonzero polynomial of degree 4. Since $P$ must pass through the points $p_{10},\ldots,p_{14}$, which lie on the line $y=0$, the line $y=0$ must be a component of 
%	the variety $\mathcal Z_{P}:=\{(x,y)\in \RR^2\colon P(x,y)=0\}$. Otherwise there would be at most four points in the intersection $\mathcal Z_{P}\cap \{y=0\}$, since
%	$\deg P\cdot \deg y=4$.
%	Hence, $P(x,y)=yQ(x,y)$, where $Q\in \RR[x,y]_{\leq 3}$ is a polynomial of degree 3. Moreover, $Q$ must pass through the points $p_1,\ldots,p_9$. 
	The matrix $S_{3,3}$ is equal to $\sum_{i=1}^9 \frac{1}{9}v_i v_i^T$, where
		$$v_i=\left(\begin{array}{ccccccccc}
			1 & x_i & x_i^2 & x_i^3 & y_i & y_ix_i & y_ix_i^2 & y_i^2 & y_i^2 x_i 
		\end{array}\right)^T$$
	and $p_i=(x_i,y_i)$.
	The vectors $v_1,\ldots,v_9$ are linearly independent, which can by checked using \textit{Mathematica} 
	by computing the determinant  ($\approx 1.83\cdot 10^8$) of the $9\times 9$ matrix with columns $v_1,\ldots,v_9$
	and hence $S_{3,3}$ is positive definite.\hfill $\blacksquare$
\end{example}

\end{document}